\documentclass[a4paper,twoside,11pt,reqno]{amsart}
\usepackage{mathrsfs,mathtools,amssymb,latexsym,eucal,extarrows,caption,dsfont} 
\usepackage[headings]{fullpage}
\usepackage[dvips]{graphicx} 
\usepackage[usenames]{color}
\usepackage[T1]{fontenc} 
\usepackage[utf8]{inputenc}
\usepackage[shortlabels]{enumitem}
\usepackage{tikz-cd} 

\usepackage[pdfpagelabels, pdftex]{hyperref}
\definecolor{darkblue}{rgb}{0,0,.85} 
\definecolor{darkred}{rgb}{0.84,0,0}
\hypersetup{
  pdftitle={},
  pdfauthor={},
  pdfsubject={},
  pdfkeywords={},
  colorlinks=true,    
  linkcolor=darkred,     
  citecolor=darkblue,     
  filecolor=darkblue,      
  urlcolor=darkblue,       
  breaklinks=true,
  bookmarksopen=true,
  bookmarksnumbered=true,
  pdfpagemode=UseOutlines,
  plainpages=false}
\usepackage[capitalise]{cleveref} 

\usepackage[kerning=true,spacing=true]{microtype}
    \microtypecontext{spacing=nonfrench}

\usetikzlibrary{arrows.meta, positioning, decorations.markings} 

    \renewcommand{\mathbb}{\mathds} 
    \newcommand{\bA}{{\mathbb A}}
    \newcommand{\bC}{{\mathbb C}}
    
    \newcommand{\bN}{{\mathbb N}}
    
    \newcommand{\bQ}{{\mathbb Q}}
    \newcommand{\bR}{{\mathbb R}}
    \newcommand{\bZ}{{\mathbb Z}}
    
    \newcommand{\cA}{{\mathcal A}}
    \newcommand{\cC}{{\mathcal C}}
    \newcommand{\cF}{{\mathcal F}}
    
    \newcommand{\cM}{{\mathcal M}}
    \newcommand{\cO}{{\mathcal O}}
    \newcommand{\cP}{{\mathcal P}}
    \newcommand{\cX}{{\mathcal X}}
    
    \newcommand{\fp}{{\mathfrak p}}
    \newcommand{\fq}{{\mathfrak q}}

    \newcommand{\surj}{\twoheadrightarrow} 
    \newcommand{\inj}{\hookrightarrow}
    \newcommand{\isomto}{\xlongrightarrow{\,\smash{\raisebox{-0.65ex}{\ensuremath{\displaystyle\sim}}}\,}}
    \newcommand{\la}{\longrightarrow} 
    \tikzset{ 
        open/.style = {decoration = {markings, mark = at position 0.5 with { \node[transform shape, scale = .7] {$\circ$}; } }, postaction = {decorate} }
    }

    \DeclareMathOperator{\id}{id}
    \DeclareMathOperator{\pr}{pr}
    \DeclareMathOperator{\Hom}{Hom}
    \DeclareMathOperator{\Aut}{Aut}
    
    \DeclareMathOperator{\coker}{coker}
    \DeclareMathOperator{\im}{im}    
    \DeclareMathOperator{\Gal}{Gal}
    \DeclareMathOperator{\Spec}{Spec}
    \DeclareMathOperator{\Bl}{Bl} 
            
    \newcommand{\sep}{{\rm sep}} 
    \newcommand{\sh}{{\rm sh}}   
    \newcommand{\gp}{{\rm gp}} 
    \newcommand{\qcqs}{{\rm qcqs}}
    \newcommand{\ket}{\text{\rm k\'et}} 
    \newcommand{\sat}{{\rm sat}} 
    \newcommand{\inte}{{\rm int}}   
    \newcommand{\rt}{{\rm t}}  

    \newcommand{\cat}[1]{\mathbf{#1}}    

    \newcommand{\Gsets}[1]{{#1\text{-}\sets}}   
    \newcommand{\KEtsite}[1]{#1_{\ket}} 
    
    \newcommand{\FEt}{\cat{F\acute{E}t}}   
    \newcommand{\Mon}{\cat{Mon}}
    \newcommand{\Sat}{\cat{Sat}}
    \newcommand{\sets}{\cat{sets}} 

    \newcommand{\ov}[1]{\overline{#1}}
    \newcommand{\vtheta}{\vartheta}

    \newcommand{\stacks}[2][Tag]{\cite[\href{https://stacks.math.columbia.edu/tag/#2}{#1~#2}]{StacksProject}}

    \newcommand{\typeV}{(\mathrm{V})}
    \newcommand{\typeVd}{(\mathrm{V}_{\rm div})}

    \makeatletter
    \def\subsubsection{\@startsection{subsubsection}{3}%
      \z@{.5\linespacing\@plus.7\linespacing}{-.5em}%
      {\normalfont\bfseries}}
    \makeatother
    
    \crefformat{subsection}{\S#2#1#3}
    
    \newcommand{\crefpart}[2]{%
      \namecref{#1}~\hyperref[#2]{\labelcref*{#1}~\ref*{#2}}%
    }

    \newtheorem{thm}{Theorem}[subsection]
    \newtheorem{prop}[thm]{Proposition}
    \newtheorem{lem}[thm]{Lemma}
    \newtheorem{cor}[thm]{Corollary}
    \newtheorem{thmABC}{Theorem}  
    
    \theoremstyle{definition}
    \newtheorem{defi}[thm]{Definition}
    \newtheorem{defiABC}[thmABC]{Definition} 

    \newtheorem{rmk}[thm]{Remark}
    \newtheorem{rmks}[thm]{Remarks}
    \newtheorem{nota}[thm]{Notation}
    \newtheorem{ex}[thm]{Example}
    
    \numberwithin{equation}{subsection} 
    
    \title{Logarithmic geometry beyond fs} 
    
    \author{Piotr Achinger}
    \address{Piotr Achinger, Instytut Matematyczny PAN, Śniadeckich 8, 00-656 Warsaw, Poland}
    \email{pachinger@impan.pl}
    
    \author{Katharina H{\"u}bner}
    \address{Katharina H{\"u}bner, Institut f\"ur Mathematik, Goethe--Universit\"at Frankfurt, Robert-Mayer-Stra\ss e~6--8, 60325~Frankfurt am Main, Germany}
    \email{huebner@math.uni-frankfurt.de}
    
    \author{Marcin Lara}
    \address{Marcin Lara, Instytut Matematyczny PAN, Śniadeckich 8, 00-656 Warsaw, Poland}
    \email{marcin.lara@impan.pl}
    
    \author{Jakob Stix}
    \address{Jakob Stix, Institut f\"ur Mathematik, Goethe--Universit\"at Frankfurt, Robert-Mayer-Str.~6--8, 60325~Frankfurt am Main, Germany}
    \email{stix@math.uni-frankfurt.de}

    \newcommand{\grants}{The first author (PA) was supported by the project KAPIBARA funded by the European Research Council (ERC) under the European Union's Horizon 2020 research and innovation programme (grant agreement No 802787). The second to fourth authors (KH, ML, and JS) acknowledge support by Deutsche Forschungsgemeinschaft (DFG) through the Collaborative Research Centre TRR 326 \emph{Geometry and Arithmetic of Uniformized Structures} project number 444845124. The third author (ML) was later supported by the National Science Centre, Poland, grant number 2023/51/D/ST1/02294. For the purpose of Open Access, the author has applied a CC-BY public copyright licence to any Author Accepted Manuscript (AAM) version arising from this submission.}
    
    \date{\today} 

\begin{document}

\hrule width\hsize
\vskip 1cm

\maketitle

\begin{quotation} 
    \noindent \small {\bf Abstract} --- We develop the foundations of logarithmic structures beyond the standard finiteness conditions. The motivation is the study of semistable models over general valuation rings. The key new notion is that of a morphism of finite presentation up to saturation (sfp), which is one that is qcqs and which is locally isomorphic to the saturated base change of a finitely presented morphism between fs log schemes. As in the case of schemes, sfp maps can (locally on the base) be approximated by maps between fs log schemes of finite type over~$\mathbb{Z}$. Based on sfp maps, we define smooth, \'etale, and Kummer \'etale maps. Importantly, the maps of schemes underlying such maps are no longer of finite type in general, though surprisingly they are if the base is the spectrum of a valuation ring with algebraically closed field of fractions. These foundations allow us to extend beyond the fs case the theory of the Kummer \'etale site and of the Kummer \'etale fundamental group.
\end{quotation}

\DeclareRobustCommand{\SkipTocEntry}[5]{} 
\setcounter{tocdepth}{2} 
{\small \tableofcontents}


\section{Introduction}
\label{s:intro}

Logarithmic geometry, as envisioned by Fontaine and Illusie and developed by 
K.~Kato (see \cite{Kato1989:LogarithmicStructures,HofM,Ogus}), is a~framework allowing one to deal more easily with compactifications and degenerations in algebraic geometry. For example, a~semistable scheme over a~discrete valuation ring, equipped with the natural logarithmic structure, becomes smooth in the sense of logarithmic geometry.

However, while the theory of schemes has been famously developed in complete generality, allowing spectra of arbitrary commutative rings as local models, logarithmic geometry has been mostly limited to log structures which locally admit a~chart by a~finitely generated (or \emph{fs}, for finitely generated and saturated) monoid. Indeed, the notion of a~log structure seems to be too general to produce meaningful universal results --- e.g., sheaves of differentials are not always quasi-coherent. Even if one restricts their attention to log structures which locally admit a~chart (by a~possibly not finitely generated monoid), called quasi-coherent, one quickly runs into foundational issues, related e.g.\ to the existence of charts for morphisms (see Remarks~\labelcref{rmk:chart-lifting} and \labelcref{rmk:refer to Tsuji exmaple in appendix}).

There are situations of geometric interest where the natural log structures do not satisfy the customary finiteness conditions. For example, the standard log structure on the spectrum of a~valuation ring $K^+$ with fraction field $K$, divided modulo units, is the non-negative part $\Gamma_K^+ = (K^+\cap K^\times)/(K^+)^\times$ of its value group $\Gamma_K = K^\times/(K^+)^\times$. This monoid will not be finitely generated unless $K^+$ is a field or a~discrete valuation ring. Therefore, in order to deal with semistable models over more general valuation rings such as $\cO_{\bC_p}$, one has to work with monoids which, while not being finitely generated, are (close to being) finitely generated over the ``base'' monoid $\Gamma^+_K$. 

The goal of this paper is to develop the necessary foundations of logarithmic structures beyond the category of fs log schemes, with a particular focus on the Kummer \'etale topology and on log schemes over (not necessarily discrete) valuation rings. In the context of semistable reduction and prismatic cohomology, similar though less general approaches have been suggested and/or developed by Koshikawa \cite{Koshikawa}, Adiprasito--Liu--Pak--Temkin \cite{ALPT}, Zavyalov \cite{Zavyalov}, and Diao--Yao \cite{DiaoYao}. Our work grew out of a~need for a~more general theory of the Kummer \'etale fundamental group, which arose in the proof of finite generation of the tame fundamental group of a rigid-analytic space in \cite{AHLS_Tame}.

The new basic notion, introduced in \cref{def:sfp}, is that of an sfp (finitely presented up to saturation) morphism of saturated\footnote{We will mostly deal with saturated monoids, due to our intended applications. One can develop a~similar story for integral monoids without putting in more effort.} monoids. 

\begin{defiABC}[\cref{def:sfp}]
    A morphism of saturated monoids $P\to Q$ is {\bf sfp} (finitely presented up to saturation) if it admits a factorisation
    \[ 
        P \la Q' \la Q
    \]
    where $P\to Q'$ is a morphism of finite presentation (where $Q'$ is not necessarily saturated or integral) and where $Q'\to Q$ induces an isomorphism $(Q')^\sat\isomto Q$.
\end{defiABC}

Equivalently, $P\to Q$ is sfp if $Q = (P\oplus_{P_0} Q_0)^\sat$ for a~homomorphism of fs monoids $P_0\to Q_0$ and a map $P_0\to P$. Sfp morphisms $P\to Q$ are precisely the compact objects of the category of saturated monoids over $P$. Based on this definition, we define smooth, \'etale, and Kummer \'etale homomorphisms between saturated monoids. 

\begin{defiABC}[\cref{def:smooth-etale-Ket-monoids}] 
    Let $\Sigma$ be a set of primes and let $\vtheta \colon P\to Q$ be a morphism of saturated monoids.
    \begin{enumerate}[(a)]
        \item We say that $\vtheta$ is {\bf smooth} if it is sfp, and if the kernel and the torsion part of the cokernel of $\vtheta^\gp\colon P^\gp\to Q^\gp$ are of order prime to all $p\in \Sigma$.
        \item The map $\vtheta$ is {\bf \'etale} if it is smooth and if moreover the cokernel of $\vtheta^\gp$ is torsion. 
        \item The map $\vtheta$ is {\bf Kummer \'etale} if it is injective, \'etale, and exact.
    \end{enumerate}
\end{defiABC}

The theory of sfp morphisms of monoids becomes particularly nice over a~base valuative monoid, and even better over a~divisible valuative monoid. Recall that an integral monoid $V$ is valuative if $V^\gp = V \cup (-V)$. It is divisible if $V^\gp$ is a divisible group. Notably, the monoid $K^+\cap K^\times$ giving the log structure on the spectrum of a valuation ring $K^+$ is valuative, and is divisible if $K$ is algebraically closed. Using the foundational results of F.~Kato \cite{FKato} and T.~Tsuji \cite{Tsuji}, we show in \cref{ss:typeV-typeVd} that if $V\to Q$ is an sfp morphism, with $V$ valuative, then $V\to Q$ is integral and after a Kummer \'etale extension $V\to W$ the saturated base change $W\to (Q\oplus_V W)^\sat$ is finitely presented and saturated, see \cref{cor:RFT} (in analogy with rigid-analytic geometry, we call this result the Reduced Fibre Theorem). In particular, if $V$ is also divisible, then every sfp morphism $V\to Q$ is finitely presented and saturated (an analogue of the Grauert--Remmert finiteness theorem in rigid-analytic geometry). Due to these useful features, and motivated by the geometric applications, we introduce the following notion: a~monoid is of type $\typeV$ (resp.\ of $\typeVd$) if it is sfp over a valuative (resp.\ divisible valuative) monoid. 

In \cref{s:log-schemes} we turn our attention to log schemes. The log schemes we consider are saturated in the sense that they admit local charts by saturated monoids (see \cref{rmk:warning-sat-int-nonstandard}). 

\begin{defiABC}[\cref{defi:sfp log map} and \cref{rmk:maps between fs are sfp}]
    A morphism of saturated log schemes $Y\to X$ is {\bf locally sfp} if \'etale locally on source and target it admits a chart by an sfp morphism of monoids $P\to Q$ such that the induced strict map 
    \[
        Y\la X\times_{\Spec(\bZ[P])}\Spec(\bZ[Q])
    \]
    is of finite presentation (as a morphism of schemes). We say that $Y\to X$ is {\bf sfp} if it is locally sfp and qcqs.
\end{defiABC}

\noindent If $X$ is an fs log scheme, a (locally) sfp map $Y\to X$ is just a morphism of fs log schemes whose underlying map of schemes is (locally) of finite presentation. Crucially, we show in \cref{cor:sfp-chart-lifting} that (locally) sfp morphisms enjoy a ``chart lifting property'', which allows us to show e.g.\ that they are stable under composition and saturated base change. One of our main results concerning sfp morphisms of log schemes is that a~qcqs map $Y\to X$ is sfp if and only if  \'etale locally \emph{on $X$} it is the saturated base change of a map of fs log schemes of finite type over $\bZ$. More generally, we obtain the following approximation result in the style of \cite[Th\'eor\`eme 8.5.2]{EGAIV3}.

\begin{thmABC}[\cref{thm:sfp-approx}]
    Let $\{X_i\}$ be an affine charted inverse system (\cref{def:affine-charted-system}) of qcqs saturated log schemes and let $X = \varprojlim X_i$ be its inverse limit. Denote by $\cat{Sfp}_X$ (resp.\ $\cat{Sfp}_{X_i}$) the category of sfp morphisms with target $X$ (resp.\ $X_i$). Then, saturated base change induces an equivalence of categories
    \[ 
        \varinjlim_{i\in I} \cat{Sfp}_{X_i} \isomto \cat{Sfp}_X.
    \]
\end{thmABC}

\noindent 
This theorem allows us to reduce many questions regarding saturated log schemes to the case of fs log structures studied in the literature. One might hope that it could be supplemented with an ``absolute approximation'' result in the style of Thomason--Trobaugh \cite[Theorem~C.9]{ThomasonTrobaugh}, to the effect that every qcqs saturated log scheme can be expressed as the inverse limit $X = \varprojlim X_i$ of an affine charted inverse system where $X_i$ are fs and of finite type over $\bZ$. However, we show in \cref{ex:no-absolute-approx} that this hope is too optimistic.

Again, the notion of an sfp morphism is the base for the definitions of smooth, \'etale, and Kummer \'etale morphisms. 

\begin{defiABC}[\cref{defi:smooth etale ket}] 
    A morphism of saturated log schemes $Y\to X$ is {\bf smooth} (resp.\ {\bf \'etale}, resp.\ {\bf Kummer \'etale}) if \'etale locally on source and target it admits a chart by a smooth (resp.\ \'etale, resp.\ Kummer \'etale) morphism of monoids $P\to Q$ (with $\Sigma$ the set of primes non-invertible on $X$) such that the induced strict map 
    \[
        Y\la X\times_{\Spec(\bZ[P])}\Spec(\bZ[Q])
    \]
    is smooth (resp.\ \'etale, resp. \'etale). In particular, such a morphism is locally sfp.
\end{defiABC}

It is important to note that the morphisms of schemes underlying smooth maps might not be of finite type. For example, we prove in \cref{ss:kummer-val-rings} that if $(K, K^+)$ is a henselian valued field and $L/K$ is a~tamely ramified finite separable extension, then with the standard log structures the morphism $\Spec(L^+)\to\Spec(K^+)$ is Kummer \'etale. However, we show in \cref{lem:ext-of-val-mon-fg} that it very rarely happens that $L^+$ is finite over $K^+$ if $K^+$ is not a discrete valuation ring (see \cref{ss:kummer-val-rings} for explicit examples). 

With these foundations in place, we study the Kummer \'etale topology $X_\ket$ and the Kummer \'etale fundamental group $\pi_1(X)$ in \cref{s:kummer-pi1}. The theory largely parallels the fs case.
Note, however, that \textit{finite} Kummer \'etale covers are Kummer \'etale maps $Y\to X$ whose underlying map of schemes is integral (but not necessarily finite). It therefore takes some effort to descend such maps to finite Kummer \'etale morphisms between fs log schemes (see \cref{prop:fKet-approx}). The following theorem summarizes the results.

\begin{thmABC}
    Let $X$ be a saturated log scheme.
    \begin{enumerate}[(a)]
        \item (Subcanonicality, \cref{prop:Ket-subcanonical}) The site $X_\ket$ is subcanonical (representable pre\-sheaves are sheaves).
        \item (Topological invariance, \cref{prop:top invariance ket site}) If $X_0\to X$ is a strict universal homeomorphism, then $X_\ket\simeq X_{0,\ket}$.
        \item (Coherence, \crefpart{prop:limit-Ket-topos}{propitem:limit-Ket-coherent}) If $X$ is qcqs, then $\cat{Sh}(X_\ket)$ is a coherent topos.
        \item (Approximation,
        \crefpart{prop:limit-Ket-topos}{propitem:limit-Ket-limit}) If $X = \varprojlim X_i$ is the inverse limit of an affine charted inverse system of qcqs saturated log schemes, then $\cat{Sh}(X_\ket)\simeq \varprojlim \cat{Sh}(X_{i,\ket})$.

        \item \label{thmitem:descent} (Descent, \cref{prop:FEt-equals-lcc}) 
        The category $\FEt_X$ of finite Kummer \'etale maps $Y\to X$ is equivalent to the category $\cat{lcc}(X_\ket)$ of locally constant Kummer \'etale sheaves of finite sets on $X$. 
    \end{enumerate}
\end{thmABC}

As a rather formal consequence of \labelcref{thmitem:descent}, if $X$ is connected and $\ov{x}\to X$ is a log geometric point, then $\FEt_X$ is a Galois category on which $\ov{x}$ induces a fibre functor. We denote by $\pi_1(X,\ov{x})$ the corresponding fundamental group and call it the {\bf Kummer \'etale fundamental group} of $X$. We exemplify it in the case of spectra of valuation rings with their standard log structure by showing that the Kummer \'etale fundamental group agrees with the tame Galois group of the fraction field, which is a~non-noetherian variant of Abhyankar's lemma. 

\begin{thmABC}[Abhyankar's Lemma, \cref{cor:pi1-vs-Galois}]
    Let $K^+$ be a henselian valuation ring with fraction field $K$. We let $S_K = \Spec(K^+ \cap K^\times \to K^+)$ be the scheme $\Spec(K^+)$ endowed  with the log structure charted by the valuative monoid $K^+\cap K^\times$. Then, there exists a canonical isomorphism of profinite groups
    \[ 
        \Gal^\rt(K^\sep/K) \isomto \pi_1(S_K, \star)
    \]
    between the the tame quotient $\Gal^\rt(K^\sep/K)$ of the absolute Galois group $\Gal(K^\sep/K)$ and the Kummer \'etale fundamental group of $S_K$.
\end{thmABC}

The final \cref{s:semistable} is devoted to smooth log schemes over the spectrum of a valuation ring $K^+$ endowed with the standard log structure. We first introduce, in analogy with non-archimedean geometry, the following notion: a saturated log scheme is of type $\typeV$ (resp.\ $\typeVd$) if it locally admits a chart by a monoid of $\typeV$ (resp.\ $\typeVd$). Thus, if $X$ a log scheme which is locally sfp over $\Spec(K^+)$, then $X$ is of type $\typeV$, and of type $\typeVd$ if the fraction field $K$ of $K^+$ is algebraically closed. The key result, proved using Zavyalov's approximation lemma \cite[Lemma~A.2]{Zavyalov}, is the following.

\begin{thmABC}[\cref{prop:vertical-compactifying}]
    Let $K^+$ be a microbial valuation ring with fraction field $K$. Suppose that $K^+$ is discretely valued or of residue characteristic zero or that $K$ is algebraically closed. Endow $S = \Spec(K^+)$ with the log structure charted by the valuative monoid $K^+\cap K^\times$. Let $X\to S$ be a smooth and vertical morphism of saturated log schemes. Then the log structure on $X$ agrees with the one induced by the open subset $X\times_S \Spec(K)$.
\end{thmABC}

This above theorem (or rather its variant for formal schemes) will play an important role in our subsequent paper \cite{AHLS_Tame}, allowing us to relate the Kummer \'etale fundamental group of the special fibre to the (suitably defined) tame fundamental group of the rigid-analytic generic fibre. We expect the theory developed here to be useful in other contexts as well.  

\subsection{Notation and conventions}

\begin{itemize}
    \item 
    We denote the category of finite sets by $\sets$.
    \item 
    The symbols $\varinjlim$ and $\varprojlim$ denote filtered colimits and cofiltered limits, respectively. We use the symbols $\lim$ and $\operatorname{colim}$ to denote more general limits and colimits.
    \item 
    For an object $X$ of a category $\cC$, we denote by $\cC_{/X}$ the category of objects over $X$ (morphisms $Y\to X$) and by $\cC_{X/}$ the category of objects under $X$ (morphisms $X\to Z$).
    \item 
    $\bN = \{0, 1, 2, \ldots\}$ considered as a monoid with $+$.
    \item 
    We denote log schemes by single symbols e.g.\ $X$ (not $(X, \cM_X)$) and the underlying scheme is denoted by $\underline{X}$. For log schemes and maps between them, we use adjectives like regular, smooth, \'etale to mean the respective notions in log geometry (log regular, log smooth, log \'etale). In particular, if $X$ is a saturated log scheme, we denote by $\FEt_X$ the category of finite Kummer \'etale maps to $X$, see \cref{defi:fet}.

    \item The underlying topological space of a (log) scheme $X$ is denoted by $|X|$.

    \item The monoid structure on $\cM_X$ and on abstract monoids is written additively.
    
    \item 
    Fibre products of integral (resp.\ saturated) log schemes within the category of integral (resp.\ saturated) log schemes will be denoted by 
    \[
        X \times_S^{\inte} Y \qquad \text{(resp.\ } 
        X \times_S^{\sat} Y \text{ )}.
    \]
    \item Unless otherwise specified, `locally' or `\'etale locally' means `strict \'etale locally,' or equivalently `\'etale locally on the underlying scheme.'
    \item The valuation ring of a valued field $K$ is denoted $K^+$. Its value group is denoted by $\Gamma_K$, and its set of non-negative elements by $\Gamma^+_K$. The monoid structure on $\Gamma^+_K$ is typically written additively, for compatibility with the convention in log geometry.
\end{itemize}

\subsection*{Acknowledgments}

First and foremost we thank IMPAN at Sopot for excellent working conditions during our stay in September 2023 when this project started. We also thank the University of Heidelberg for hosting us during an intense few weeks in February 2024. 

We thank Jarek Buczyński for help with the proof of \cref{prop:valuative-ind-free} and Jakub Byszewski for help with \crefpart{ex:typeV-not-fg}{exitem:typeV-not-fg2b}. We thank Veronika Ertl-Bleimhofer for her comments on the manuscript and Alberto Vezzani and Michael Temkin for helpful discussions. We thank Takeshi Tsuji and Arthur Ogus for providing useful counterexamples. 

\addtocontents{toc}{\protect\SkipTocEntry}
\subsection*{Funding}

\grants 

\section{Monoids beyond fs}
\label{s:monoids}

In this section, we develop the theory of commutative monoids needed for the extension of logarithmic geometry beyond the setting of fs log schemes in \cref{s:log-schemes}.  The natural relative version of an fs monoid is that of a map of saturated monoids $P\to Q$ which is finitely presented up to saturation (sfp for short), meaning that $Q$ is the saturation of a finitely presented monoid over $P$ (\cref{def:sfp}). The reason for not working simply with maps of finite presentation is that they are not preserved under pushouts in the category of saturated monoids, while of course we want natural classes of morphisms of saturated log schemes (log smooth, Kummer \'etale etc.) to be preserved under pullback. The map $P\to Q$ being sfp means precisely that $Q$ is a compact object in the category of saturated monoids over $P$, or that it arises via saturated pushout from a~map of fs monoids $P_0\to Q_0$ (\cref{prop:sfp-conditions}). This allows one to reduce questions about sfp maps to the case of fs monoids, similarly to how one reduces questions about finitely presented maps of qcqs schemes to the case of schemes of finite type over $\bZ$.

The geometric situation in need of such an extension is that of a~semistable scheme over a~general valuation ring $K^+$. In this case, the `base' monoid is the non-negative part $\Gamma_K^+$ of the value group $\Gamma_K$ of $K$, which will not be finitely generated unless $K$ is trivially or discretely valued. However, the monoid $\Gamma_K^+$ is \emph{valuative}: for every element $x\in \Gamma_K = (\Gamma_K^+)^\gp$, either $x\in \Gamma_K^+$ or $-x\in \Gamma_K^+$. As we shall see, the theory of monoids which are sfp over a valuative monoid has numerous favourable properties. We call them monoids of type $\typeV$, and in the case when the valuative monoid is divisible, of type $\typeVd$, see \cref{ss:typeV-typeVd}. 

\begin{rmk}[Analogies with non-archimedean geometry]
Our approach is guided by analogy with non-archimedean geometry, and our terminology reflects this. A formal scheme of type $\typeV$ is one locally of the form $\operatorname{Spf}(\cA)$ for an algebra $\cA$ topologically of finite presentation (tfp) over a~complete microbial valuation ring $K^+$ \cite[\S 7.3]{Bosch}. A theorem of Nagata \cite[Theorem~3]{Nagata}, generalized by Raynaud and Gruson \cite[Corollaire I 3.4.7]{RaynaudGruson}, implies that a topologically finitely generated (tft) $K^+$-algebra which is torsion-free is automatically finitely presented. Suppose $K^+$ is of rank one, let $A$ be a~geometrically reduced affinoid $K$-algebra, and let $A^\circ$ denote its subring of powerbounded elements. By the Grauert--Remmert finiteness theorem \cite[Corollary~6.4.1/5]{BGR}, if $K$ is either discretely valued or stable \cite[\S 3.6]{BGR} and its value group $\Gamma$ is divisible (e.g.\ $K=\overline K$), then $A^\circ$ is tfp over $K^\circ=K^+$. In general, this may fail even if $A$ is a finite separable extension of $K$ \cite[3.6.1, Example]{BGR}. However, one can find a tfp $K^\circ$-algebra $\cA$ with $A = \cA_K$, and then $A^\circ$ equals the integral closure of $\cA$ in $A$ \cite[Theorem 3.1/17]{Bosch}. Thus $A^\circ$ is ``finitely presented up to integral closure.''  Moreover, by the Reduced Fibre Theorem \cite[Theorem~1.3]{BLR-IV}, after passing to a finite extension $L$ of $K$ we achieve that $A^\circ$ is~tfp.

Similarly, let $V$ be a valuative monoid, let $P'$ be a finitely presented monoid over $V$, and let $P$ be its saturation. While $V\to P$ is sfp by definition, it may happen that $P$ is no longer finitely presented (or even finitely generated) over $V$, see \cref{ex:typeV-not-fg}. However, using results of F.~Kato and T.~Tsuji we show that this is true if $V$ is valuative and divisible (\cref{cor:GR-finiteness}). Moreover, one can always make $P$ finitely presented by passing to a ``finite'' extension of $V$ (\cref{cor:RFT}). One also has a version of the finite presentation result of Nagata (\cref{cor:monoid-raynaud-gruson}), as well as a~version of ``N\'eron desingularisation'' (\cref{prop:valuative-ind-free}).
\end{rmk}

\subsection{Preliminaries on monoids}
\label{ss:monoid-prelims}

All monoids are commutative and the operations are typically written additively. We recall some basic notions \cite[Chapter I]{Ogus}.

\subsubsection*{Integral and saturated monoids}

For a monoid $P$, we denote by $P\to P^\gp$ its initial homomorphism into a group, and call $P$ {\bf integral} if this map is injective. We call $P$ {\bf fine} if it is integral and finitely generated. A monoid $P$ is {\bf saturated} if it is integral and for every $p\in P^\gp$ and $n\geq 1$ such that $np\in P$, we have $p\in P$. A monoid $P$ is {\bf fs} (fine and saturated) if it is finitely generated and saturated. The subgroup of invertible elements of $P$ is denoted by $P^\times$, and we say that $P$ is {\bf sharp} if $P^\times = 0$. 

For an integral monoid $M$ and a subset $S$, the {\bf saturation of $S$ in $M$} consists of all $p\in M$ for which there exists an $n\geq 1$ such that $np$ belongs to the submonoid of $M$ generated by $S$. The inclusion of integral (resp.\ saturated) monoids into all monoids admits a left adjoint, the {\bf integralisation} $P\mapsto P^\inte = \im(P\to P^\gp)$ (resp.\ the {\bf saturation} $P\mapsto P^\sat$ defined to be the saturation of $P^\inte$ as a subset of $P^\gp$).

\smallskip

\begin{lem}
\label{lem:injective-implies-gp-injective}
    Let $\vtheta \colon P \to Q$ be an injective map of integral monoids. Then $\vtheta^\gp \colon P^\gp \to Q^\gp$ is injective.
\end{lem}

\begin{proof}
The groupification $M^\gp$ of a monoid $M$ can be constructed explicitly as $M\times M /\sim$, where 
\[
    (x, y) \sim (x',y') \iff \exists {z \in M}: \quad x + y' + z = x' + y + z.
\]
If $M$ is integral, then the cancellation law holds for $M$ and so $z$ can be omitted.

If $\vtheta^\gp([(x,y)]) = 0$ in $Q^\gp$, then $(\vtheta(x),\vtheta(y)) \sim (0,0)$ holds. In other words $\vtheta(x)$ equals $\vtheta(y)$, since $Q$ is integral. As $\vtheta$ is injective, $x$ equals $y$, and thus 
$[(x,y)] = 0$ and $\vtheta^\gp$ is injective. 
\end{proof}

\subsubsection*{Limits and colimits}

The categories of monoids, integral monoids, and saturated monoids admit all limits and colimits, see \cite[Chapter I \S1.1]{Ogus}. The inclusion functors preserve limits (being right adjoints) and importantly also filtered colimits \cite[Proposition I 1.3.6]{Ogus}. To avoid confusion, we will denote by $P\oplus_{P_0} Q_0$ the pushout of $P\leftarrow P_0\to Q_0$ in the category of monoids. In analogy with the theory of commutative rings, such a pushout will often be referred to as the \emph{base change} of $Q_0$ to $P$. If $P$, $P_0$, and $Q_0$ are saturated, the corresponding pushout $Q$ in the category of saturated monoids admits the following description (compare \cite[\S 2.4]{DiaoYao})
\[ 
    Q = (P\oplus_{P_0} Q_0)^\sat \quad \subseteq \quad Q^\gp = P^\gp\oplus_{P_0^\gp} Q_0^\gp
\]
is the saturation of the union of the images of $P$ and $Q_0$ in $Q^\gp$. We will also refer to it as the \emph{saturated base change} of $Q_0$ to $P$.

\begin{lem} 
\label{lem:ker and coker under pushout}
    Let $P \to Q$ be the pushout (resp.\ integral or saturated pushout) of a morphism of monoids (resp.\ integral or saturated monoids) $P_0 \to Q_0$ along a map $P_0 \to P$. Then the natural map
    \[
        Q_0^\gp/P_0^\gp \la Q^\gp/P^\gp
    \]
    is an isomorphism. The natural map 
    \[
        \ker(P_0^\gp \to Q_0^\gp) \la \ker(P^\gp \to Q^\gp)
    \]
    is surjective, and it is even an isomorphism if $P_0 \to P$ is an injective map of integral monoids.
\end{lem}

\begin{proof}
Since groupification is left adjoint, it commutes with pushouts. So the lemma actually contains a claim about the pushout in the category of abelian groups. It follows by the snake lemma applied to the diagram (and \cref{lem:injective-implies-gp-injective} for the last assertion)
\[
    \begin{gathered}[b]
        \begin{tikzcd}
        0 \ar[r] & P_0^\gp \arrow[r,"{x \mapsto (x,-x)}"] \ar[d] & P_0^\gp \oplus P^\gp \ar[r] \ar[d] & P^\gp \ar[r] \ar[d] & 0 \\ 
        0 \ar[r] & \im(P_0^\gp \to Q_0^\gp \oplus P^\gp)  \ar[r]  & Q_0^\gp \oplus P^\gp \ar[r] & Q^\gp \ar[r] & 0 \ . 
        \end{tikzcd}
        \\[-\dp\strutbox] 
    \end{gathered}
    \qedhere
\]
\end{proof}

\subsubsection*{Short exact sequences of monoids}

By a {\bf short exact sequence} of monoids we mean a~sequence of integral monoids
\[
    0 \longrightarrow G \longrightarrow M \overset{\varphi}{\longrightarrow} N \longrightarrow 0,
\]
where~$G$ is a subgroup of $M$ and $\varphi\colon M \to N$ identifies $N$ with the set of $G$-orbits. This is equivalent to saying that the sequence of associated groups is exact and the map $\varphi\colon M\to N$ is \textbf{exact} (recalled in \crefpart{defi:exactintsatvert monoids}{defitem:exact monoid map} below). We shall also write $N=M/G$ in this situation. In the special case $G = M^\times$ being the largest subgroup of $M$, we denote $M/G = M/M^\times$ by $\overline{M}$.

An often useful fact about short exact sequences of monoids is that providing a section of the map $M\to N$ is equivalent to splitting the associated exact sequence of groups. Indeed, since $\varphi$ is exact, a section of $M^\gp \to N^\gp$ base changes along $N \to N^\gp$ to a map $N \to N \times_{N^\gp} M^\gp = M$ which is a section of $\varphi$.

\subsection{Finite presentation up to saturation}
\label{ss:sfp-monoids}

Let $f\colon P\to Q$ be a homomorphism of monoids. A {\bf generating set} for $Q$ over $P$ is a subset $S\subseteq Q$ for which the induced map of $P\oplus \bN^S\to Q$ is surjective. We say that $f$ is of {\bf finite type} (or that $Q$ is finitely generated over $P$) if $Q$ admits a~finite generating set over $P$.

A {\bf congruence} $E$ on a monoid $M$ is a submonoid $E\subseteq M\times M$ which is also an equivalence relation. In this case, the set of equivalence classes $M/E$ is a monoid and the projection $M\to M/E$ is a monoid homomorphism. For a morphism of monoids $\vtheta \colon M \to N$, the fibre product $E = M \times_N M$ is a congruence, called the {\bf kernel congruence} of $\vtheta$. If $M\to N$ is a~surjective monoid homomorphism, then $N \simeq M/E$ where $E$ is the kernel congruence. For a~subset $R\subseteq M\times M$ (called a set of {\bf relations}), the {\bf congruence generated by} $R$ is the smallest congruence $E\subseteq M\times M$ containing $R$. We write $M/R$ for $M/E$ and call it the {\bf quotient of $M$} by the set of relations $R$. The morphism $M\to M/R$ is initial among the maps $f\colon M\to N$ with the property that $f(a)=f(b)$ for every $(a,b)\in R$. 

For a monoid homomorphism $f\colon P\to Q$, a {\bf presentation} $(S, R)$ of $Q$ over $P$ consists of a~generating set $S\subseteq Q$ and a set of relations $R\subseteq (P\oplus \bN^S)^2$ such that the induced map $P\oplus \bN^S\to Q$ identifies $Q$ with the quotient $(P\oplus\bN^S)/R$. We say that $f$ is of {\bf finite presentation} (or that $Q$ is finitely presented over $P$) if $Q$ admits a presentation $(S, R)$ where both $S$ and $R$ are finite (see \cite[\S 3.10]{AdamekRosicky} for a general notion of a finitely presented object in any algebraic theory). 

A monoid $P$ is \textbf{finitely generated} (resp.\ \textbf{finitely presented}) if the map $0 \to P$ is of finite type (resp.\ of finite presentation).

\begin{nota}
    We denote by $\Sat_{P/}$ the category of saturated monoids over the monoid $P$, and by $\Mon_{P/}$ the category of monoids over $P$. Here a monoid over $P$ is a morphism $P \to Q$ of monoids, which is an object under $P$ in the categorical sense.
\end{nota}  

\begin{lem}
\label{lem:fp-AdamekRosicky}
    A morphism $P \to Q$ of monoids is of finite presentation if and only if $Q$ is a~compact object of $\Mon_{P/}$.
\end{lem}

\begin{proof}
This is \cite[Corollary~3.13]{AdamekRosicky}.
\end{proof}

\begin{lem} 
\label{lem:base change fp ft}
    Base change (pushout in the category of monoids) preserves finite type (resp.\ finite presentation).
\end{lem}

\begin{proof}
Let $P_0 \to Q_0$ be a morphism of monoids and let $P \to Q$ be the base change along $P_0 \to P$. If $P_0 \to Q_0$ is of finite type generated by a finite set $S \subseteq Q_0$, then $Q$ is also finitely generated over $P$ by the image of $S$ under the map $Q_0\to Q$. 

Let now $P_0 \to Q_0$ be of finite presentation. By \cref{lem:fp-AdamekRosicky} we must show that the pushout along $P_0 \to P$ of a compact object in $\Mon_{P_0/}$ is a compact object in $\Mon_{P/}$. 
But this is formal since base change preserves compact objects: for any filtered colimit $T = \varinjlim_i T_i$ of monoids over $P$, we have
\[
    \Hom_P(Q, T) = \Hom_{P_0}(Q_0,T) = \varinjlim_i \Hom_{P_0}(Q_0, T_i) = \varinjlim_i \Hom_P(Q,T_i). \qedhere
\]
\end{proof}

\begin{lem}
\label{lem:maps between fp is fp}
    Any morphism $P \to Q$ between finitely presented monoids is of finite presentation.
\end{lem}

\begin{proof}
By \cref{lem:fp-AdamekRosicky} it is equivalent to show that a morphism $P \to Q$ of compact objects of the category of monoids is a compact object in the category $\Mon_{P/}$ of monoids over $P$. But this is again just formal. 
\end{proof}

\begin{prop}
\label{prop:fp-conditions}
    Let $P \to Q$ be a morphism of monoids. 
    The following are equivalent:
    \begin{enumerate}[(a)]
        \item 
        \label{propitem:fp}
        The morphism $P\to Q$ is of finite presentation.
        \item 
        \label{propitem:fp-cptobj}
        The monoid $Q$ is a compact object of $\Mon_{P/}$.
        \item 
        \label{propitem:fp-bc}
        There exists a pushout square in the category of monoids
        \[
            \begin{tikzcd}
                Q & Q_0 \ar[l] \\
                P\ar[u] & P_0 \ar[l] \ar[u] 
            \end{tikzcd}
        \]
         where $P_0$ and $Q_0$ are finitely presented monoids.
    \end{enumerate}
\end{prop}

\begin{proof}
We already cited \cite[Corollary~3.13]{AdamekRosicky} for the equivalence \labelcref{propitem:fp} $\Leftrightarrow$ \labelcref{propitem:fp-cptobj}.  

\smallskip

\labelcref{propitem:fp-bc} $\Rightarrow$ \labelcref{propitem:fp}: 
the morphism $P_0 \to Q_0$ is finitely presented by \cref{lem:maps between fp is fp}. Then the base change is finitely presented by \cref{lem:base change fp ft}. 

\smallskip

\labelcref{propitem:fp} $\Rightarrow$ \labelcref{propitem:fp-bc}: 
Let $S \subseteq Q$ be a finite generating set of $Q$ over $P$ such that that the congruence $E \subseteq (P \oplus \bN^S)^2$ defining $Q$ as a quotient of $P \oplus \bN^S$
is generated by the finite set $R \subseteq E$. 
Let $\pi_1, \pi_2\colon (P\oplus \bN^S)^2\to P$ denote the two projections, and let $T = \pi_1(R)\cup \pi_2(R)$ be the set of elements that occur as $P$-components of elements in tuples in $R$ 

Define $P_0 = \bN^T$ and consider $T$ as a subset of $\bN^T$ as the set of generators. Each element of $R$ has a unique preimage in $(T \times \bN^S)^2$ under the map induced by $P_0 \oplus \bN^S \to P \oplus \bN^S$ that maps a~generator of $\bN^T = P_0$ to the respective element in $P$. Let $R_0$ be the union of these preimages, so that  $R_0$ tautologically lifts $R$. 
Define further the quotient 
\[
    Q_0 = (P_0 \oplus \bN^S)/R_0.
\]
Then there is a commutative square as in the claim \labelcref{propitem:fp-bc}
induced by the morphism $P_0 \to Q_0$ and the inclusions $T \to P$ and $S \to Q$. The square is a pushout square by construction and the definition of finite presentation. This shows \labelcref{propitem:fp-bc} since $P_0$ and $Q_0$ are finitely presented by construction. \end{proof}

\begin{defi} 
\label{def:sfp}
    Let $\vtheta\colon P\to Q$ be a morphism of saturated monoids.
    \begin{enumerate}[(1)]
        \item A {\bf sat-generating set} for $Q$ over $P$ is a subset $S\subseteq Q$ such that $Q$ equals the saturation in $Q^\gp$ of the sub-monoid generated by $S$ and the image of $P$. 
        \item 
        We say that $\vtheta\colon P\to Q$ is {\bf of finite type up to saturation (sft)} if it admits a finite sat-generating set, or equivalently if there exists a factorisation 
        \begin{equation} \label{eqn:sft-monoid-def}
                P\la Q'\la Q 
        \end{equation}
        where the map $P\to Q'$ is of finite type and the map $Q'\to Q$ induces an isomorphism $(Q')^\sat\isomto Q$. 
        \item A {\bf sat-presentation} for $Q$ over $P$ is a pair $(S, R)$ where $S\subseteq Q$ is a sat-generating set and $R\subseteq (P\oplus \bN^S)^2$ is a subset such that $Q' = P\oplus \bN^S/R$ fits into a factorization \labelcref{eqn:sft-monoid-def} where again $(Q')^\sat\isomto Q$ (we do not assume that $Q'\to Q$ is injective).
        \item We say that $\vtheta\colon P\to Q$ is {\bf of finite presentation up to saturation (sfp)} if it admits a~sat-presentation $(S, R)$ where both $S$ and $R$ are finite, or equivalently if there exists a~factorisation \labelcref{eqn:sft-monoid-def} where $P\to Q'$ is of finite presentation. 
    \end{enumerate}
\end{defi}

\begin{rmk}
It is not true that every sfp morphism is of finite type, see \crefpart{ex:typeV-not-fg}{exitem:typeV-not-fg2} and \crefpart{ex:typeV-not-fg}{exitem:typeV-not-fg2b}, where the morphisms are even Kummer \'etale (\cref{def:smooth-etale-Ket-monoids}).
\end{rmk}

\begin{lem}
\label{lem:sfp stable under bc}
    The saturated base change of an sft (resp.\ sfp) morphism is again sft (resp.\ sfp).
\end{lem}

\begin{proof}
This follows from the definition of sft (resp.\ sfp) maps, \cref{lem:base change fp ft} and the fact that saturation commutes with (saturated) pushouts.  
\end{proof}

We first characterise sfp morphisms with source an fs monoid. 

\begin{prop}
    \label{prop:fs version of sorite for sfp}
    Let $P \to Q$ be a morphism of saturated monoids with $P$ an fs monoid. The following are equivalent:
    \begin{enumerate}[(a)]
        \item
        \label{propitem:fs version: fs}
        The monoid $Q$ is fs.
        \item
        \label{propitem:fs version: sft}
        The morphism $P \to Q$ is sft.
        \item
        \label{propitem:fs version: ft}
        The morphism $P \to Q$ is of finite type.
        \item
        \label{propitem:fs version: sfp}
        The morphism $P \to Q$ is sfp.
        \item
        \label{propitem:fs version: fp}
        The morphism $P \to Q$ is of finite presentation.
    \end{enumerate}    
\end{prop}
\begin{proof}
\labelcref{propitem:fs version: fs} $\Rightarrow$ \labelcref{propitem:fs version: fp}: 
  By \cite[Theorem~I~2.1.7]{Ogus}, every fs monoid is finitely presented. Assertion \labelcref{propitem:fs version: fp} follows from \cref{lem:maps between fp is fp}. 

  The implications 
  \labelcref{propitem:fs version: fp} $\Rightarrow$ \labelcref{propitem:fs version: sfp}  $\Rightarrow$ \labelcref{propitem:fs version: sft} and \labelcref{propitem:fs version: fp} $\Rightarrow$  \labelcref{propitem:fs version: ft} $\Rightarrow$ \labelcref{propitem:fs version: sft} are trivial. So we are left to prove the following. 

  \smallskip
  
  \labelcref{propitem:fs version: sft} $\Rightarrow$ \labelcref{propitem:fs version: fs}: 
    Pick a factorisation $P \to Q' \to Q$ with $P \to Q'$ finitely generated and $Q' \to Q$ an isomorphism after saturation. As an fs monoid, $P$ is finitely generated, hence $Q'$ is finitely generated. Then $Q$ is fs, being the saturation of a finitely generated monoid by Gordan's Lemma (cf.~\cite[Lemma 3.1.1]{Stix2002:Thesis}). 
\end{proof}

For general sfp morphisms, we have the following important characterisations.

\begin{prop} 
\label{prop:sfp-conditions}
    Let $P \to Q$ be a morphism of saturated monoids. 
    The following are equivalent:
    \begin{enumerate}[(1)]
        \item 
        \label{propitem:sfp}
        The morphism $P\to Q$ is sfp.
        \item 
        \label{propitem:sfp-cptobj}
        The monoid $Q$ is a compact object of $\Sat_{P/}$.
        \item 
        \label{propitem:sfp-bc}
        There exists a pushout square in the category of saturated monoids
        \[
            \begin{tikzcd}
                Q & Q_0 \ar[l] \\
                P\ar[u] & P_0 \ar[l] \ar[u] 
            \end{tikzcd}
        \]
        (so $Q = (P\oplus_{P_0}Q_0)^\sat$) where $P_0\to Q_0$ is a morphism of fs monoids.
    \end{enumerate}
    The above conditions \labelcref{propitem:sfp}--\labelcref{propitem:sfp-bc} further imply and, if $P\to Q$ is injective, are equivalent to the following condition:
    \begin{enumerate}[(1),resume]
        \item
        \label{propitem:sfp-sft}
        The morphism $P\to Q$ is sft.
    \end{enumerate}
\end{prop}

\begin{proof}
\labelcref{propitem:sfp} $\Rightarrow$ \labelcref{propitem:sfp-bc}:
Let $P \to Q' \to Q$ be a factorisation such that $P \to Q'$ is finitely presented and $(Q')^\sat \isomto Q$ is an isomorphism. By \cref{prop:fp-conditions} we find finitely presented monoids $P'_0 \to Q'_0$ and a map $P'_0 \to P$ such that $Q' = P \oplus_{P'_0} Q'_0$. The map $P_0 = (P'_0)^\sat \to Q_0 = (Q'_0)^\sat$ is a map of fs monoids by Gordan's Lemma. 
Since saturation commutes with pushout, we find $Q = (Q')^\sat = (P \oplus_{P_0} Q_0)^\sat$ as requested.

\smallskip

\labelcref{propitem:sfp-bc} 
$\Rightarrow$ \labelcref{propitem:sfp-cptobj}:  
By \cref{prop:fs version of sorite for sfp} the morphism $P_0 \to Q_0$ in the diagram in \labelcref{propitem:sfp-bc} is finitely presented. Hence $Q_0$ is a compact object in $\Mon_{P_0/}$ by \cref{lem:fp-AdamekRosicky}. Since the inclusion  $\Sat_{P_0/} \to \Mon_{P_0/}$ preserves filtered colimits, its left adjoint functor \textit{saturation} preserves compact objects. Therefore the fs monoid $Q_0$ is a compact object in $\Sat_{P_0/}$. Base change along $P_0 \to P$ preserves compact objects and so shows \labelcref{propitem:sfp-cptobj}. 

\smallskip

Although \labelcref{propitem:sfp} $\Rightarrow$ \labelcref{propitem:sfp-sft} is obvious, we include a proof of \labelcref{propitem:sfp-cptobj} $\Rightarrow$ \labelcref{propitem:sfp-sft} because this serves as the first part of our proof of \labelcref{propitem:sfp-cptobj} $\Rightarrow$ \labelcref{propitem:sfp}.

\labelcref{propitem:sfp-cptobj} $\Rightarrow$ \labelcref{propitem:sfp-sft}:
For any finite set $S \subseteq Q$ we denote by $Q(S)_0$ the submonoid of $Q$ generated by $S$ over $P$, and by $Q(S) = Q(S)_0^\sat$ its saturation. Note that $Q(S)_0$ is integral by definition and that $Q(S) \to Q$ is injective. Then clearly 
\[
    Q = \varinjlim_S Q(S)
\]
as a colimit of saturations of monoids finitely generated over $P$. As $P \to Q$ is a compact object, we have
\[
    \Hom_P(Q,\varinjlim_S Q(S)) = \varinjlim_S \Hom_P(Q,Q(S)).
\]
Therefore, the identity $\id_Q$ lifts  to a morphism $Q \to Q(S)$ for large enough $S$. It follows that the injective map $Q(S) \to Q$ is an isomorphism, i.e. $S$ sat-generates $Q$ over $P$. 

\smallskip

\labelcref{propitem:sfp-cptobj} $\Rightarrow$ \labelcref{propitem:sfp}:
We already know that $P \to Q$ is sft by the proof of \labelcref{propitem:sfp-cptobj} $\Rightarrow$ \labelcref{propitem:sfp-sft} above. 
Pick any finite sat-generating set $S$ of $Q$ over $P$, and consider the kernel congruence $E \subseteq (P \oplus \bN^S)^2$ of the morphism
\[
    P \oplus \bN^S \to Q, \qquad (p,n) \mapsto p + \sum_{s \in S} n_s s.
\]
with image $Q_0 = (P \oplus \bN^S)/E$. 
For any finite subset $R \subseteq E$ we denote by $Q_R$ the saturation of the quotient 
\[
Q_{R,0} = (P \oplus \bN^S)/R.
\]
It follows that
\[
    \varinjlim_R Q_R = \big(\varinjlim_R Q_{R,0}\big)^\sat \to Q_0^\sat = Q
\]
is an isomorphism over $P$. Again, by $Q$ being a compact object over $P$, we have
\[
    \Hom_P(Q,\varinjlim_R Q_R) = \varinjlim_R \Hom_P(Q,Q_R)
\]
so again the identity $\id_Q$ lifts to a map $\iota \colon Q \to Q_R$.  

As for $R \subseteq R'$ the map $Q_{R,0} \surj Q_{R',0}$ is surjective, also the map $Q_R^\gp \surj Q_{R'}^\gp$ is surjective. 
The filtered colimit of finitely generated abelian groups and surjective transition maps
\[
    \varinjlim_R (Q_R^\gp/P^\gp) = (\varinjlim_R Q_R)^\gp/P^\gp = Q^\gp/P^\gp
\]
becomes constant for $R \gg 0$, because $\bZ$ is a noetherian ring. 

We choose $R$ large enough so that 
both properties hold: $Q_R^\gp/P^\gp \isomto Q^\gp/P^\gp$ is an isomorphism and $\id_Q$ lifts to $\iota \colon Q \to Q_R$ over $P$, i.e.\  splitting the natural map $r\colon Q_R \to Q$. The restrictions of $r^\gp$ and $\iota^\gp$ to the 
images of $P^\gp \to Q^\gp$ and $P^\gp \to Q_R^\gp$ are mutually inverse isomorphisms (since $\iota^\gp$ is injective). By the $5$-lemma we conclude that $r^\gp \colon Q_R^\gp \isomto Q^\gp$ is an isomorphism and this shows that the map $r\colon Q_R \to Q$ is  injective. As a retract map, $r$ is also surjective. Now $P \to Q_R=Q$ is sfp by construction, and this concludes the proof of \labelcref{propitem:sfp}.

\smallskip

We now assume that $P \to Q$ is injective.

\labelcref{propitem:sfp-sft} $\Rightarrow$ \labelcref{propitem:sfp}: 
Let $S$ be a finite sat-generating set of $Q$ over $P$. The morphism $P \oplus \bN^S \to Q$ induces a surjection 
\[
    P^\gp \oplus \bZ^S \surj Q^\gp
\]
with kernel $K$. By assumption $P \to Q$ is injective, and, by \cref{lem:injective-implies-gp-injective}, also $P^\gp \inj Q^\gp$ is injective. Hence the second projection induces an injective map $K \inj \bZ^S$. Thus $K$ is a finitely generated abelian group.

Let $f_i - g_i \in K$, for $i=1, \ldots, n$ be generators of $K$ with $f_i,g_i \in P \oplus \bN^S$. Using 
\[
    R = \{(f_i,g_i) \ ; \ i= 1, \ldots,n\}
\]
we define the monoid
\[
    Q' = (P \oplus \bN^S)/R
\]
which is finitely presented over $P$ by construction. Since $Q$ is saturated, the relations $R$ hold in $Q$ and we obtain an induced morphism $Q' \to Q$. Then 
\[
    Q'^\gp = (P^\gp \oplus \bZ^S)/K \isomto Q^\gp.
\]
Therefore the natural map $(Q')^\sat \to Q$ is injective, and it is surjective because the image contains $S$. This shows that $P \to Q$ is sfp.
\end{proof}

\begin{rmk}
    The reader may have noticed that the description of compact objects in $\cat{Mon}_{/P}$ as the finitely presented morphisms (\cref{lem:fp-AdamekRosicky}) was much easier to obtain than the corresponding statement about compact objects in $\cat{Sat}_{/P}$ being precisely the sfp morphisms (\cref{prop:sfp-conditions}). The intuitive reason for this is that saturated monoids do not form an ``algebraic theory'' (more precisely, an \emph{equational class} or a \emph{variety} in the sense of universal algebra), as the saturation axioms contain an existential quantifier. Trying to deduce the statement about compact objects in $\cat{Sat}_{P/}$ from the one about compact objects in $\cat{Mon}_{P/}$ using the natural adjunction between the two categories, one can completely formally obtain the following statement: $Q$ is a compact object in $\cat{Sat}_{P/}$ if and only if $Q$ is a \emph{retract} of the saturation of a finitely presented monoid over $P$. Without arguments as in last step of the proof of \cref{prop:sfp-conditions} \labelcref{propitem:sfp-cptobj}$\Rightarrow$\labelcref{propitem:sfp}, employing crucially the noetherianity of $\bZ$, it is not obvious why we can get rid of the ``retract'' part. 
\end{rmk}

\begin{rmk}
\label{rmk:ALPT}
    The proof for \labelcref{propitem:sfp-sft} $\Rightarrow$ \labelcref{propitem:sfp} in 
    \cref{prop:sfp-conditions} shares similarities with the proof of  \cite[Lemma~2.1.7]{ALPT}: an injective homomorphism $P\to Q$ of integral monoids which is of finite type (or, equivalently, of finite type up to integralisation) is of finite presentation up to integralisation (ifp). This is an instance of the parallel story of notions based on integral monoids and integralisation rather than saturated monoids and saturation that we remarked in a footnote in the introduction. 
\end{rmk}

\begin{cor}
\label{cor:sfp stable composition}
    The composition of sfp morphisms is again sfp.  
\end{cor}
\begin{proof}
    Let $P \to Q$ and $Q \to R$ be sfp morphisms. Then by \crefpart{prop:sfp-conditions}{propitem:sfp-cptobj} $R$ is a compact object in $\Sat_{Q/}$ and $Q$ is a compact object in $\Sat_{P/}$. It is formal that then $R$ is a compact object in $\Sat_{P/}$ and hence $P \to R$ is an sfp morphism by \cref{prop:sfp-conditions}.    
\end{proof}

\begin{cor} 
\label{cor:fs_compact}
    The fs monoids are the compact objects of the category of saturated monoids.
\end{cor}
\begin{proof}
    By \cref{prop:sfp-conditions}, a compact object is the same as a monoid  $P$ with $0 \to P$ an sfp morphism. These are precisely the fs monoids. 
\end{proof}

\begin{cor} 
\label{cor:sfp_comes_from_finite_level}
    Let $P = \varinjlim_{i \in I} P_i$ be a filtered colimit of saturated monoids.
    \begin{enumerate}[(a)]
        \item \label{coritem:sfp-a} Let $0\in I$ and let $P_0\to Q_0$ and $P_0\to R_0$ be two sfp morphisms. Denote by $P_i\to Q_i$, $P_i\to R_i$ ($i\geq 0$), $P\to Q$, $P\to R$ the respective saturated base changes. Then the map 
        \[ 
            \varinjlim_{i\geq 0} \Hom_{P_i}(Q_i, R_i) \longrightarrow \Hom_P(Q, R)
        \]
        is bijective.
        \item \label{coritem:sfp-b} For every sfp morphism $P \to Q$, there exists an index $i \in I$ and an sfp morphism $P_i \to Q_i$ fitting into a pushout square
        \[
         \begin{tikzcd}
             Q          & Q_i   \ar[l]  \\
             P  \ar[u]  & P_i   \ar[l]  \ar[u]
         \end{tikzcd}
        \]
        in the category of saturated monoids. 
    \end{enumerate}
\end{cor}

The above statement can be rephrased as follows. Let $\cat{Sfp}_P$ denote the category of sfp morphisms $P\to Q$. Then saturated base change induces an equivalence of categories
\[ 
    \varinjlim \cat{Sfp}_{P_i} \isomto \cat{Sfp}_P.
\]
For a variant of this result for integral monoids see \cite[Lemma~2.1.9]{ALPT}.

\begin{proof}
    \labelcref{coritem:sfp-a}:
    We compute
    \begin{align*} 
        \varinjlim_{i\geq 0} \Hom_{P_i}(Q_i, R_i) 
        & = \varinjlim_{i\geq 0} \Hom_{P_0}(Q_0,R_i) \\
        & = \Hom_{P_0}(Q_0, \varinjlim_{i\geq 0} R_i) = \Hom_{P_0}(Q_0, R) =  \Hom_P(Q, R)
    \end{align*}
    because $Q_0$ is a compact object 
    in $\cat{Sfp}_{P_0}$ by \crefpart{prop:sfp-conditions}{propitem:sfp-cptobj}.
        
    \labelcref{coritem:sfp-b}:
    By \cref{prop:sfp-conditions} we can find a pushout square in the category of saturated monoids 
    \[
     \begin{tikzcd}
         Q          & Q'   \ar[l]  \\
         P  \ar[u]  & P'   \ar[l]  \ar[u]
     \end{tikzcd}
    \]
    where $P' \to Q'$ is a morphism of fs monoids.
    Since~$P'$ is a compact object in the category of saturated monoids by \cref{cor:fs_compact}, there is an index $i \in I$ such that $P' \to P$ factors through $P' \to P_i$.
    We denote by $P_i \to Q_i$ the saturated base change of $P' \to Q'$ to~$P_i$. By \cref{prop:fs version of sorite for sfp} the map $P' \to Q'$ is sfp, and $P_i \to Q_i$ is sfp by \cref{lem:sfp stable under bc}. This gives the desired pushout square.
\end{proof}

\subsection{Faces and prime ideals}

An {\bf ideal} of a monoid $P$ is a subset $K\subseteq P$ such that $P+K = K$. It is {\bf prime} if $K\neq P$ and if $p+p'\in K$ implies that either $p\in K$ or $p'\in K$. The set of prime ideals of $P$ is denoted by $\Spec(P)$. The topology generated by the sets 
\[
    \{K \in \Spec(P) \ | \ x \notin K\} \qquad (x\in P)
\]
makes it into a $T_0$-space with generic point $\emptyset$ and unique closed point $P \setminus P^\times$.

A {\bf face} of $P$ is a submonoid $F\subseteq P$ such that $p+p'\in F$ implies that $p\in F$ and $p'\in F$. Faces of $P$ are precisely the complements of its prime ideals. We will often identify $\Spec(P)$ with the set of faces of $P$.

For a face $F\subseteq P$ of an integral monoid, we define the localization $P_F$ of $P$ at $F$ as 
\[
    P_F = P + F^\gp \quad \subseteq P^\gp,
\]
and the quotient $P/F$ as the quotient monoid 
\[
    P/F = P_F/F^\gp = \im(P_F \to P^\gp/F^\gp).
\]

We say that a face $F$ of $P$ is \textbf{generated as a face} by a subset $S\subseteq F$ if it is the intersection of all faces containing $S$. In case $S$ is finite, we may replace $S$ with the sum $s$ of all of its elements, and then the face generated by $S$ is
\[
    F = \{ x \in P \ | \ \exists {y\in P}, \ \exists {n\geq 1}: x+y=ns \}.
\]

\begin{rmk}
\label{rmk:spec of saturation map}
    For an integral monoid $P$, the saturation map $P \to P^\sat$ induces a bijection $\Spec(P^\sat) \to \Spec(P)$. The inverse assigns to a face $F$ of $P$ the face generated by $F$ in $P^\sat$.  
\end{rmk}

\begin{lem} 
\label{lem:compare-faces}
    Let $\vtheta\colon P\to Q$ be a homomorphism of saturated monoids, let $F$ be a face of~$Q$, and let $S\subseteq Q$ be a sat-generating set relative to $P$. Then, $F$ equals the saturation of the submonoid of $Q$ generated by $F\cap \vtheta(P)$ and $F\cap S$.
    
    In particular, if $P \to Q$ is sft, then  the map of faces $\vtheta^{-1}(F) \to F$ is  sft.  
\end{lem}

\begin{proof}
By definition of a sat-generating set, for an element $q\in Q$, there exists an $n\geq 1$ and an equation of the form
\[ 
    nq = \vtheta(p) + \sum_{s\in S} m_s\cdot  s 
\]
for some $p\in P$ and a function $m\colon S\to \bN$ with finite support. Suppose that $q\in F$. By definition of a face, it follows that $\vtheta(p)\in F$ and $s\in F$ for every $s\in S$ such that $m_s>0$.  We thus have written $nq$ as a combination of elements in $F\cap \vtheta(P)$ and $F\cap S$.
\end{proof}

\begin{cor} 
\label{cor:finite-fibers-on-Spec}
    Let $\vtheta\colon P\to Q$ be an sft map. Then the fibres of the map $\Spec(Q)\to \Spec(P)$ are finite with cardinality bounded by $2^{\# S}$ with the size $\# S$ of any sat-generating set $S$ relative to~$P$. In particular, if $P$ has finitely many faces, then so does $Q$.
\end{cor}
\begin{proof}
   \Cref{lem:compare-faces} shows that faces $F\subseteq Q$ with given face $\vtheta^{-1}(F)$ are uniquely determined by the subset $F \cap S$.
\end{proof}

\subsection{Elementary studies of certain classes of morphisms}
\label{ss:smooth ket maps}

We first recall the definition of some classical notions for morphisms of monoids.

\begin{defi}
\label{defi:exactintsatvert monoids}
    A homomorphism $\vtheta\colon P\to Q$ of integral monoids is
    \begin{enumerate}[(a)]
        \item \label{defitem:exact monoid map}
        \textbf{exact} \cite[Definition I 2.1.15]{Ogus} if the square
        \[ 
            \begin{tikzcd}
                P\ar[d] \ar[r] & Q \ar[d] \\
                P^\gp\ar[r] & Q^\gp
            \end{tikzcd}
        \]
        is cartesian, or equivalently if $P=(\vtheta^\gp)^{-1}(Q)$,
        \item \label{defitem:integral monoid map}
        \textbf{integral} \cite[Definition~I~4.6.2]{Ogus} if for every homomorphism $P\to P'$ into an integral monoid $P'$, we have
        \[ 
            P'\oplus_{P} Q = (P'\oplus_P Q)^\inte,
        \]
        \item \label{defitem:saturated monoid map}
        \textbf{saturated} \cite[Definition~3.12 and Proposition~3.14]{Tsuji}
        if $P\to Q$ is integral, $P$ and $Q$ are saturated, and for every homomorphism $P\to P'$ into a saturated monoid $P'$, we have
        \[ 
            P'\oplus_{P} Q = (P'\oplus_P Q)^\sat,
        \]
        \item \label{defitem:vertical monoid map}
        \textbf{vertical} if the image of $P$ generates $Q$ as a face: for every $q\in Q$ there exist $q'\in Q$ and $p\in P$ such that $q+q' = \vtheta(p)$, or equivalently if $Q$ and $-\vtheta(P)$ generate $Q^\gp$.
    \end{enumerate}
\end{defi}

\begin{rmks}
\label{rmks:sorite exact int sat vert}
\begin{enumerate}[(1)]
    \item 
    An often useful fact about exact morphisms $\vtheta \colon P \to Q$ is that providing a section of $\vtheta$ is equivalent to providing a section of 
    $\vtheta^\gp \colon P^\gp \to Q^\gp$.  
    \item \label{rmkitem:sorite exact}
    The class of exact morphisms is closed under composition, integral pushout and, if $P$ and $Q$ are saturated, saturated pushout
    \cite[Proposition~I~4.2.1]{Ogus}.
    \item 
    The class of integral (resp.\ saturated) morphisms is closed under composition, integral (resp.\ saturated) pushout.
    \item \label{rmksitem:sorite vert}
    The class of vertical morphisms is closed under composition and integral pushout. Moreover, if $P$ and $Q$ are saturated, then being vertical is also stable under saturated pushout \cite[Proposition~I~4.3.3]{Ogus}. 
\end{enumerate}
\end{rmks}

\begin{ex} \label{ex:Ogus-sat-pushout-of-int}
    The class of integral morphisms is in general not preserved by saturated pushout. We learned this from Arthur Ogus with essentially the following example. Consider the morphism $\vtheta_0 \colon \bN^2  \to \frac{1}{2}\bN^2$ which is integral e.g.\ by \cite[Proposition~I~4.6.7]{Ogus}. Let $P \subseteq \frac{1}{2}\bN^2$ be the submonoid generated by $\bN^2$ and $(\frac 1 2, \frac 1 2)$.  Consider the saturated base change 
    \[
        \vtheta \colon P \la Q = \left(P \oplus_{\bN^2} {\textstyle \frac{1}{2}}\bN^2\right)^\sat
    \]
    of $\vtheta_0$ along the inclusion $\bN^2 \to P$. Then $Q \simeq \frac{1}{2}\bN^2 \times \bZ/2\bZ$ with $\vtheta(P)$ contained in the first factor. The morphism $\vtheta$ is not integral as can easily be verified by computing the base change in monoids with the ``sum of the coordinates'' map  $P \to \frac{1}{2}\bN$.    
\end{ex}

We shall now use the notion of an sfp map to define smooth, \'etale, and Kummer \'etale maps between saturated monoids. These definitions are motivated by the classical notions in log geometry.  

\begin{defi} 
\label{def:smooth-etale-Ket-monoids}
    Henceforth, we fix a possibly empty set of primes $\Sigma$ and set 
    \[
        p = \prod_{\ell\in \Sigma} \ell
    \]
    as a supernatural number (formal product). In the applications the set $\Sigma$ will be clear out of context, so we drop it in the notation.

    A morphism of saturated monoids $P\to Q$ is 
    \begin{enumerate}
        \item {\bf smooth} if it is sfp and the kernel and the torsion part of the cokernel of $P^\gp\to Q^\gp$ are finite groups of order prime to $p$,
        \item {\bf \'etale} if it is smooth and the cokernel of  $P^\gp\to Q^\gp$ is torsion,
        \item {\bf Kummer \'etale} if it is injective, \'etale, and exact. 
    \end{enumerate}
\end{defi}

\begin{ex}[``Semistable reduction''] \label{ex:semistable-monoid}
    Let $P$ be a saturated monoid, let $\pi\in P$, and let $n\geq 0$. Define 
    \[
        P_n(\pi) = P[e_1, \ldots, e_n]/(e_1+\cdots+e_n = \pi)
    \]
    as the quotient of $P\oplus \bN^n$ by the congruence relation generated by $(\pi, 0)\sim (0, (1, \ldots, 1))$. Then $P_n(\pi)$ is saturated and $P\to P_n(\pi)$ is smooth, vertical, and saturated. In the case $P=\bN$ and $\pi=1$, these assertions are classical: the first three follow directly from the definitions. That the morphism is saturated follows from the characterization of \cite[Theorem~I.6.3(8)]{Tsuji} (use \cite[Theorem~I.2.1]{Tsuji} first to check that the morphism is integral): for any field $k$, the morphism
    \[
        \bA_{P_n[\pi]} = \Spec(k[t][x_1, \ldots, x_n]/(x_1\cdot\ldots x_n - t)) \longrightarrow \Spec(k[t]) = \bA_{\bN}
    \]
    is flat and has reduced special fibre. The general case follows by base change along the map $\bN\to P$ sending $1$ to $\pi$, see \cite[Example IV 3.1.17]{Ogus}.
\end{ex}

\begin{lem}
\label{lem:bc of smooth etale Ket monoid maps}
    Let $\cP$ be one of the properties:
    smooth, \'etale, injective, Kummer \'etale. 
    Then the saturated base change of a morphism with property $\cP$ has property $\cP$. 
\end{lem}
\begin{proof}
    Let $P_0 \to Q_0$ and $P_0 \to P$ be morphisms of saturated monoids. We set 
    \[
        Q = (P \oplus_{P_0} Q_0)^\sat
    \]
    for the saturated pushout. Recall from \cref{lem:ker and coker under pushout} that $Q_0^\gp/P_0^\gp \isomto Q^\gp/P^\gp$ is an isomorphism and the morphism 
    $\ker(P_0^\gp \to Q_0^\gp) \surj \ker(P^\gp \to Q^\gp)$ is surjective.

    As being sfp is stable under saturated base change by \cref{lem:sfp stable under bc}, clearly smooth and \'etale are also stable under saturated base change.  The claim for injective follows because by \cref{lem:injective-implies-gp-injective}, the map $P \to Q$ is injective if and only if $P^\gp \to Q^\gp$ is injective. Now Kummer \'etale follows as well being defined as the combination of base change stable properties (for ``exact'' see Remark~\labelcref{rmks:sorite exact int sat vert}\labelcref{rmkitem:sorite exact}). 
\end{proof}

\begin{prop}
\label{prop:properties of sfp that descend to fs}
    Let $P\to Q$ be an sfp morphism of saturated monoids and let $\cP$ be one of the following properties: smooth, \'etale, injective, exact, vertical, Kummer \'etale. Then $P\to Q$ has property $\cP$ if and only if there exists a morphism of fs monoids $P_0\to Q_0$ with property $\cP$ and a~map $P_0\to P$ such that $Q = (P\oplus_{P_0}Q_0)^\sat$. 
\end{prop}

\begin{proof}
If such a map $P_0 \to Q_0$ with property $\cP$ as in the lemma exists, then $P \to Q$ has property $\cP$ by \cref{rmks:sorite exact int sat vert} and \cref{lem:bc of smooth etale Ket monoid maps}.

For the converse direction we assume that $P \to Q$ has property $\cP$. Let $P_0 \to Q_0$ be a morphism of fs monoids with $Q = (P \oplus_{P_0} Q_0)^\sat$  as in \crefpart{prop:sfp-conditions}{propitem:sfp-bc}. We write $P=\varinjlim_i P_i$ as a~filtered colimit of fs monoids over $P_0$, namely the saturations of submonoids that are finitely generated over $P_0$. We set $Q_i = (P_i \oplus_{P_0} Q_0)^\sat$ which yields $Q = \varinjlim_i Q_i$ as a filtered colimit of fs monoids. Moreover, by possibly replacing $P_0$ by one of the $P_i$ and the system by all $P_j$ with $j \geq i$, we may and do assume that all $P_i$ are submonoids of $P$: the transfer maps in $P = \varinjlim_iP_i$ are injective. We abbreviate
\[
K_i = \ker(P_i^\gp \to Q_i^\gp) \quad \text{ and } \quad C_i = \coker(P_i^\gp \to Q_i^\gp).
\]
For $j\geq i$, transitivity of pushouts yields $Q_j = (P_j \oplus_{P_i} Q_i)^\sat$. 
We apply \cref{lem:ker and coker under pushout} and obtain 
isomorphisms $C_i \isomto C_j$ and $K_i \isomto K_j$ because by construction $P_i^\gp \to P_j^\gp$ is injective as subgroups of $P^\gp$.

Since groupification commutes with colimits and filtered colimits are exact, we find an exact complex
\[
    \begin{tikzcd}
        0 \ar[r] & \varinjlim_i K_i \ar[r] & P^\gp \ar[r] & Q^\gp \ar[r] & \varinjlim_i C_i \ar[r] & 0.
    \end{tikzcd}
\]
This shows that all maps $P_i \to Q_i$ and $P \to Q$ have isomorphic kernels and cokernels of their groupification. Being sfp
as maps between fs monoids by \crefpart{prop:sfp-conditions}{propitem:sfp-bc}, we deduce that $P_i \to Q_i$ has property $\cP$ for $\cP$ equal to: smooth, \'etale, injective. It remains to argue for $\cP$ equal to exact and $\cP$ equal to vertical. 

Suppose that $\vtheta \colon P \to Q$ is exact. For all $i$ we set $P_i' = \smash{P_i^\gp \times_{Q_i^\gp} Q_i}$. This is an fs monoid, being the preimage of the fs monoid $Q_i$ under the homomorphism of finitely generated abelian groups $P_i^\gp \to Q_i^\gp$. Moreover, since $P \to Q$ is exact, we obtain the factorisation $P_i \to P_i' \to P$. In the diagram
\[
    \begin{tikzcd}[column sep=1cm, row sep=1cm]
        Q & Q_i \ar[l] & Q_i \ar[l,"\sim",swap] \\
        P \ar[u,"\vtheta"] & P_i' \ar[l] \ar[u,"\vtheta'_i"] & P_i \ar[l] \ar[u,"\vtheta_i"]
    \end{tikzcd}
\]
both the outer rectangle and the right square are pushout diagrams of saturated monoids. Indeed, the saturated pushout 
$P'_i \oplus^\sat_{P_i} Q_i$ is the saturation in  $Q_i^\gp = {P'_i}^\gp \oplus_{P^\gp_i} Q_i^\gp$ of the submonoid generated by $Q_i$ and $P_i'$, and this is still only $Q_i$. Hence also the left square is a~pushout diagram of saturated monoids. So $\vtheta$ is the saturated base change of the exact map $\vtheta_i' \colon P_i' \to Q_i$.

Finally, suppose that $\vtheta\colon P\to Q$ is vertical. Let $S$ be a finite sat-generating set for $Q$ over $P$. Then $\vtheta$ is vertical if and only if for every $s\in S$ there exist $q_s\in Q$ and $p_s\in P$ such that 
\[
    s + q_s = \vtheta(p_s).
\]
We may assume that $S$ is the image of a finite sat-generating set $S_i$ of $Q_i$ over $P_i$. Increasing $i$, we may ensure the existence of elements $q'_r\in Q_i$ and $p'_r\in P_i$ (for $r\in S_i$) such that $r + q'_r = \vtheta_i(p'_r)$. This implies that $P_i\to Q_i$ is vertical.
\end{proof}

Combining this with \cref{cor:sfp_comes_from_finite_level}, we obtain the following results.

\begin{prop} \label{prop:monoid-descend-prop-P}
    Let $P = \varinjlim_{i\in I} P_i$ be a filtered colimit of saturated monoids, let $0\in I$ be the smallest element, let $P_0\to Q_0$ be an sfp map, and let $P_i\to Q_i$ ($i\geq 0$) and $P\to Q$ be obtained by saturated base change. Then $P\to Q$ has property $\cP$ (as in \cref{prop:properties of sfp that descend to fs}) if and only if $P_i\to Q_i$ has property $\cP$ for $i\gg 0$.
\end{prop}

\begin{proof}
The ``if'' part follows from \cref{rmks:sorite exact int sat vert} and \cref{lem:bc of smooth etale Ket monoid maps}. For the converse, suppose that $P\to Q$ has property $\cP$. By \cref{prop:properties of sfp that descend to fs}, we find a cocartesian square in the category of saturated monoids
\[ 
    \begin{tikzcd}
        Q & Q'\ar[l] \\
        P\ar[u] & P' \ar[l]\ar[u]
    \end{tikzcd}
\]
with $P'\to Q'$ a morphism of fs monoids which has property $\cP$. Since $P'$ is a compact object in the category of saturated monoids, the map $P'\to P$ factors through $P_{i_0}$ for $i_0\gg 0$. Fix such a~factorisation, and for $i\geq i_0$, let $Q'_i = (Q'\oplus_{P'} P_i)^\sat$. Then by \cref{rmks:sorite exact int sat vert} and \cref{lem:bc of smooth etale Ket monoid maps}, the maps $P_i\to Q'_i$ have property $\cP$. By \cref{cor:sfp_comes_from_finite_level}, for $i\gg i_0$ we have $Q_i\simeq Q'_i$ over $P_i$, and hence for such $i$, the map $P_i\to Q_i$ has property $\cP$ as well.
\end{proof}

\begin{cor} \label{cor:monoid-sm-et-Ket-approx}
    In the situation of \cref{prop:monoid-descend-prop-P}, the category of smooth, \'etale, or Kummer \'etale maps $P\to Q$ is equivalent to the filtered colimit of the categories of maps $P_i\to Q_i$ with the same property. 
\end{cor}

\subsection{Monoids of type \texorpdfstring{$\typeV$}{(V)} and of type \texorpdfstring{$\typeVd$}{(Vdiv)}}
\label{ss:typeV-typeVd}

Using the notion of an sfp map, we shall now define monoids of type $\typeV$ (resp.\ type $\typeVd$) as those which are sfp over a valuative monoid (resp.\ divisible valuative monoid). The two assumptions valuative and divisible for the base monoid enable us to apply the results of F.~Kato (\cref{thm:F-Kato}) and T.~Tsuji (\cref{thm:Tsuji}). 

\subsubsection*{Valuative monoids}
\label{sss:valuative-monoids}

Recall that a commutative monoid $V$ is {\bf valuative} if it is integral and for every $x\in V^\gp$, either $x\in V$ or $-x\in V$, see \cite[I \S1.3]{Ogus}. For example, if $K^+$ is a valuation ring with fraction field $K$, then the monoid $K^+\setminus \{0\}$ as well as the monoid $\Gamma_K^+ = (K^+\setminus \{0\})/(K^+)^\times$ of non-negative elements in its value group $\Gamma_K = K^\times / (K^+)^\times$ are valuative monoids. If $\Gamma$ is a totally ordered abelian group, then its submonoid $\Gamma^+$ of non-negative elements is sharp and valuative, and conversely the groupification $V^\gp$ of a sharp valuative monoid $V$ is totally ordered by $x\geq y$ if $x-y\in V$. A valuative monoid is automatically saturated, but it will typically not be finitely generated (the only sharp valuative fs monoids are $0$ and $\bN$). 

\subsubsection*{Definition, basic properties, and examples}

We will call a saturated monoid $P$ {\bf divisible} if for every $n\geq 1$, the multiplication by $n$ map $n\colon P\to P$ is surjective, or equivalently if $P^\gp$ is a~divisible group.

\begin{defi} 
\label{def:monoid-typeV}
    A saturated monoid $Q$ is of {\bf type $\typeV$} (resp.\ {\bf of type $\typeVd$}) if there exists a~valuative monoid $V$ (resp.\ a divisible valuative monoid $V$) and an sfp map $V\to Q$. 
\end{defi}

Trivially, if $Q\to Q'$ is an sfp map and $Q$ is of type $\typeV$ or $\typeVd$ then so is $Q'$.

\begin{rmk}
    Every homomorphism from a valuative monoid $V$ to an integral monoid $Q$ is integral by  \cite[Proposition I 4.6.3(5)]{Ogus}. In particular, any $V \to Q$ as in \cref{def:monoid-typeV} is integral as well.
\end{rmk}

\begin{ex} 
\label{ex:typeV-not-fg}
    We illustrate the notion of a monoid of type $\typeV$, especially the flexibility arising since the monoid $V$ itself is not fixed as part of the structure.
    \begin{enumerate}[(1),parsep=0cm,itemsep=0.2cm,topsep=0.2cm]
        \item 
        \label{exitem:typeV-not-fg1}
        Every fs monoid is of type $\typeVd$, simply over the trivial valuative monoid $V=0$.
        \item 
        \label{exitem:typeV-not-fg2}
        It is not true in general that if $V\to Q$ is sfp and $V$ is valuative then $Q$ is finitely generated over $V$. For example, let $V = (\bZ \oplus \bZ)^+$ be the non-negative elements in $\bZ \oplus \bZ$ with lexicographical ordering, i.e.
        \[
            V = \left(\bN_{>0} \times \bZ\right) + \left(0 \times \bN\right) \quad\subseteq\quad \bN \times \bZ \quad\subseteq\quad V^\gp = \bZ^2
        \]
        and let 
        \[
        Q = {\textstyle \frac{1}{2}} V\subseteq \bQ^2
        \]
        (see \cref{fig:monoid1}, in which the smaller dots are in $Q$ and the larger ones are in $V$). Then $Q$ is sat-generated over $V$ by $\{(\frac 1 2, 0), (0, \frac 1 2)\}$, 
        and hence sft and even sfp by \cref{prop:sfp-conditions}. Even more, $V\to Q$ is Kummer \'etale (for $\Sigma=\{2\}$). But any set of generators of $Q$ over $V$ has to contain an infinite subset of the elements $\frac 1 2 \times (-\bN)$. 
        However, in this example $Q$ is finitely generated over a valuative monoid (namely, over itself), so in particular $Q$ is of type $\typeV$. See \cref{ex:type5} for a related example in log schemes. 

        \item
        \label{exitem:typeV-not-fg2b}
        For a similar example with a valuative monoid of rank one, let $\lambda$ be an irrational real number and let
        \[ 
            V = \{ (a, b) \in \bZ^2 \, :\, a+b\lambda\geq 0\},
        \]
        which is identified with a dense submonoid of $\bR^+$ via $(a, b)\mapsto a+b\lambda$. Again, set $Q = \frac{1}{2} V\subseteq \bQ^2$ (see \cref{fig:monoid2}, in which the smaller dots are in $Q$ and the larger ones are in $V$). Then $Q$ is not finitely generated over $V$. Indeed, let $S\subseteq Q$ be a finite subset, and let $\varepsilon = \inf\{a+b\lambda\,:\, (a,b)\in S\setminus\{0\}\} > 0$. Since $\lambda$ is irrational, there exists an element $(a, b)\in Q$ such that $0< a+b\lambda< \varepsilon$. Dividing $a$ and $b$ by a power of two, we may assume that $(a, b) \notin V$. But then $(a,b)$ does not belong to the submonoid of $Q$ generated by $V$ and $S$. We thank Jakub Byszewski for this argument. See \cref{ex:type3} for a related example in log schemes. 

        \begin{figure}
        \centering
        \begin{minipage}{.5\textwidth}
          \centering
          \includegraphics[width=.6\linewidth]{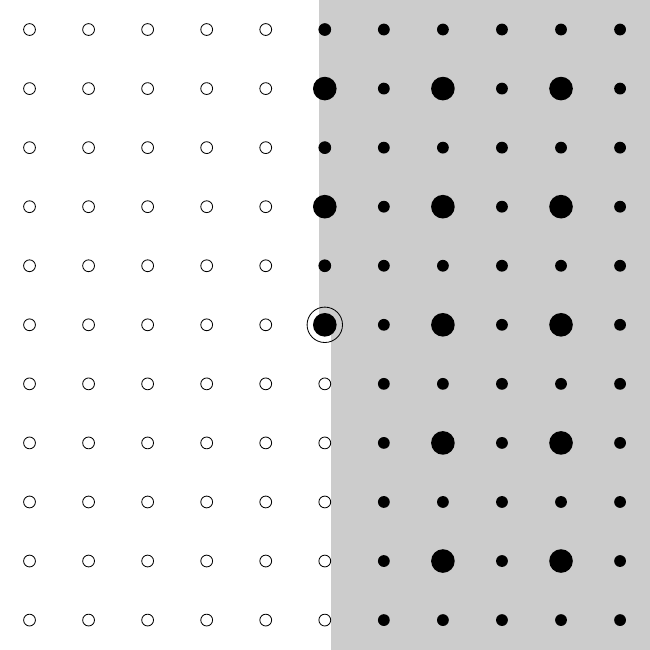}
          \captionof{figure}{The inclusion $V\subseteq Q$ in \crefpart{ex:typeV-not-fg}{exitem:typeV-not-fg2}.}
          \label{fig:monoid1}
        \end{minipage}%
        \begin{minipage}{.5\textwidth}
          \centering
          \includegraphics[width=.6\linewidth]{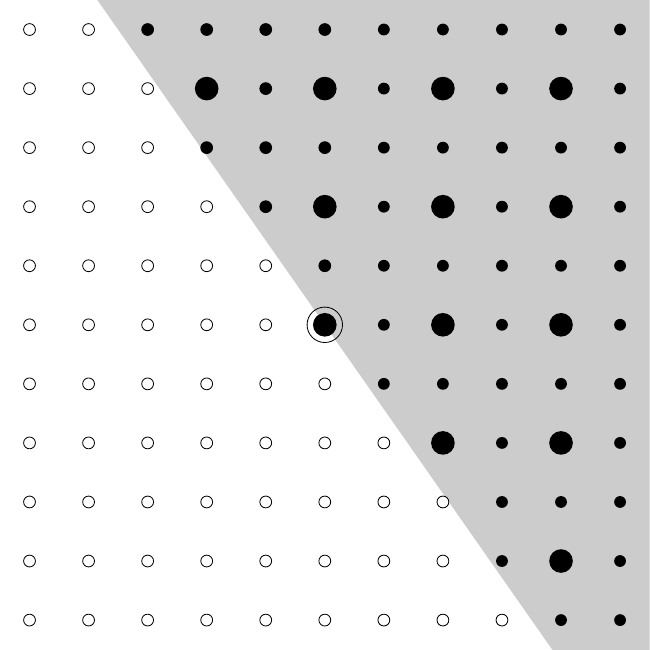}
          \captionof{figure}{The inclusion $V\subseteq Q$ in \crefpart{ex:typeV-not-fg}{exitem:typeV-not-fg2b}.}
          \label{fig:monoid2}
        \end{minipage}
        \end{figure}
                        
        \item 
        \label{exitem:typeV-not-fg3}
        In the next example we show that in general a monoid $Q$ of type $\typeV$ need not be finitely generated over a valuative monoid.
        Consider again $V = (\bZ \oplus \bZ)^+$ as in \labelcref{exitem:typeV-not-fg2} and let $Q$ be the saturation of the submonoid of $\frac{1}{2}V\oplus \frac{1}{2}\bN$ generated by $V\oplus \bN$ and the element $(1/2, 0, 1/2)$:
        \begin{align*}
            Q  & =  \left((V \oplus \bN) + \bN \cdot ({\textstyle \frac{1}{2}},0,{\textstyle \frac{1}{2}}) \right)^\sat  \\
            & = \left\{(x,y,z) \in {\textstyle \frac{1}{2}}\bN \times \bZ \times {\textstyle \frac{1}{2}}\bN \ : \  x-z \in \bZ \text{ and, if $y < 0$, then $x>0$}\right\}. 
        \end{align*}
        The underlying geometric inspiration is as follows: the monoid $Q$ occurs in the integral closure of  the algebra $R= K^+[T,S]/(S^2 - \pi T)$ in $R[1/\pi]$ for $K^+$ a valuation ring with $\Gamma_K^+ = V$ and $\pi\in K^+$ an element (pseudouniformizer) with valuation $(1,0)$. 
        
        Clearly $V \to Q$, the inclusion of the $xy$-plane, is sft by construction and  therefore sfp by \cref{prop:sfp-conditions}. Therefore $Q$ is a monoid of type $\typeV$. 

        Suppose that there is a valuative monoid $W$ and a 
        morphism $\vtheta \colon W \to Q$, such that $Q$ is finitely generated over $W$. We may replace $W$ by the image of $W$ in $Q$, since this image is still a valuative monoid over which $Q$ is finitely generated.
        Suppose that $W$ is not contained in $V$, so that there exists an element $(\alpha, n)\in W\subseteq \frac{1}{2}V\oplus \frac{1}{2}\bN$ with $n>0$. One of the elements $(-\alpha, n)$ and $(\alpha, -n)$ belongs to $W$. Since $n>0$, this must be $(-\alpha, n) \in W$ and thus $\alpha = 0$. This shows that $W \subseteq 0\times \bN$. Since $V$ is a face of $Q$, any generating set of $Q$ over $W$ contains a generating set of $V$ over $V\cap W = 0$. But $V$ is not a finitely generated monoid, a contradiction. Thus $W\subseteq V$.

        Now it is enough to show that $Q$ is not finitely generated over $V$. For this, we note that every generating set of $Q$ over $V$ must contain infinitely many of the elements 
        \[
            q_n = (1/2,-n,1/2), \quad n \in \bN,
        \]
        because every decomposition 
        $q_n = a + b$ in $Q$ must be of the form $q_n = q_{n+k} + (0,k,0)$ with $k \geq 0$.

        \item 
        Not all saturated submonoids of finitely generated groups are of type $\typeV$. Take any saturated submonoid $Q\subseteq\bZ^2$ which is not finitely generated and not valuative. Basic examples are given by  
        \[
            Q  = (\bN_{>0}\times \bN) \cup \{(0,0)\} 
            = \{(x,y) \in \bN^2 \ | \  \text{$x > 0$ if $y> 0$}\},
        \]
        or a cone with irrational slope such as
        \[ 
            Q = \{(x,y) \in \bN^2 \,:\, y\leq \sqrt{2} x\}. 
        \]
        It is easy to check that $Q$ is saturated but not fs. If $V$ is a valuative monoid and $V\to Q$ is an sfp homomorphism, then the image of $V$ in $Q$ is a valuative monoid over which $Q$ is sft (see the proof of \cref{prop:typeV-criterion}). However, the only valuative submonoids of $Q$ are of the form $\bN\cdot q$ for some $q\in Q$, and in particular are fs. Thus, if $Q$ was sft over $V$, it would be fs by \cref{prop:fs version of sorite for sfp}.
    \end{enumerate}
\end{ex}

In fact, it is very rare that a Kummer \'etale homomorphism between valuative monoids is of finite type. Let us call a totally ordered abelian group $\Gamma$ {\bf discrete} if $\Gamma=0$ or if $\Gamma^+ \setminus \{0\}$ has a~smallest element $\gamma$, or equivalently if its smallest non-zero convex subgroup is isomorphic to $\bZ$. Now, Examples~\labelcref{ex:typeV-not-fg}\labelcref{exitem:typeV-not-fg2} and \labelcref{ex:typeV-not-fg}\labelcref{exitem:typeV-not-fg2b} can be generalized as follows.  

\begin{lem} \label{lem:ext-of-val-mon-fg}
    Let $\Gamma\to \Sigma$ be an injective map of totally ordered abelian groups whose cokernel is finite (so that $\Gamma^+\to\Sigma^+$ is Kummer \'etale with $p=1$). The following are equivalent:
    \begin{enumerate}[(a)]
        \item \label{lemitem:valmon-finite} $\Sigma^+$ is finite as a $\Gamma^+$-set,
        \item \label{lemitem:valmon-fg} $\Sigma^+$ is finitely generated over $\Gamma^+$,
        \item \label{lemitem:valmon-rare} either $\Gamma = \Sigma$, or $\Gamma$ is discrete and there exists an integer $n\geq 1$ such that $\Sigma = \Gamma + \frac{1}{n}\bZ \gamma$ where $\gamma$ is the smallest positive element of $\Gamma$.
    \end{enumerate}
\end{lem}

\begin{proof}
The implication \labelcref{lemitem:valmon-finite}$\Rightarrow$\labelcref{lemitem:valmon-fg} is clear. For \labelcref{lemitem:valmon-fg}$\Rightarrow$\labelcref{lemitem:valmon-finite}, if $S\subseteq \Sigma^+$ is a finite generating set over $\Gamma+$, with $0 \in S$, and if $n\geq 1$ is an integer annihilating $\Sigma/\Gamma$, then the finite set of sums
\[ 
    \sum_{s\in S} m(s)\cdot s, \quad m\colon S\la \{0, \ldots, n-1\}
\]
generates $\Sigma^+$ as a $\Gamma^+$-set. For \labelcref{lemitem:valmon-rare}$\Rightarrow$\labelcref{lemitem:valmon-finite} we note that $\Sigma^+$ is generated as a $\Gamma^+$-set by the elements $\frac{i}{n}\gamma$ for $0\leq i<n$. It remains to show \labelcref{lemitem:valmon-finite}$\Rightarrow$\labelcref{lemitem:valmon-rare}.

We write $\pi\colon \Sigma \to \Sigma/\Gamma$ for the quotient map. For every $t\in \Sigma/\Gamma$, let $\smash{\Sigma^+_t} = \pi^{-1}(t)\cap \Sigma^+$ be the monoid of non-negative elements of $\Sigma$ congruent to $t$ modulo $\Gamma$. Then $\smash{\Sigma^+}$ decomposes as a~$\Gamma^+$-set
\[ 
    \Sigma^+ = \coprod_{t\in \Sigma/\Gamma} \Sigma^+_t,
\]
and hence \labelcref{lemitem:valmon-finite} is equivalent to $\Sigma^+_t$ being finite as a $\Gamma^+$-set for all $t$. 

We claim that $\Sigma^+_t$ is finite as a $\Gamma^+$-set if and only if it has a smallest element $\sigma_t$. The argument for this is similar to the one for the claim that an ideal in a valuation ring is finitely generated if and only if it is principal and will be omitted.  

Assume \labelcref{lemitem:valmon-finite}, and let $\sigma_t \in \Sigma^+_t$ be smallest elements as above. Suppose that $\Sigma\neq\Gamma$. Let $\sigma = \sigma_{t_0}$ be the smallest non-zero element among the $\sigma_t$, or equivalently the smallest element of $\Sigma^+\setminus \Gamma^+\neq\varnothing$. Then $[0, \sigma)\subseteq \Sigma^+_0 = \Gamma^+$. Suppose that there exists a $g\in \Gamma$ such that $0<g<\sigma$. Then $s = \sigma - g$ satisfies $0<s<\sigma$, and moreover $s \in \Sigma\setminus \Gamma$ since it is congruent to $\sigma$ modulo $\Gamma$. This contradicts the minimality of $\sigma$. We conclude that the interval $(0, \sigma)$ is empty, which means that $\sigma$ is the smallest non-zero element of $\Sigma^+$, and hence $\Sigma$ (and therefore  $\Gamma$) is discrete. 

Let $n\geq 1$ be the smallest integer such that $n\cdot \sigma \in \Gamma$. Then $\gamma = n\cdot \sigma$ is the smallest element of $\Gamma$. We claim that $\Sigma = \Gamma + \bZ \sigma$, which will imply the final assertion of \labelcref{lemitem:valmon-rare}. The injective map $\Gamma' = \Gamma + \bZ \sigma \to \Sigma$ satisfies the assumptions of the lemma and \labelcref{lemitem:valmon-finite}. Since $\Gamma'$ and $\Sigma$ have the same smallest positive element, namely $\sigma$, the argument in the previous paragraph shows that $\Gamma' = \Sigma$.
\end{proof}

\subsubsection*{F.\ Kato's theorem and its corollaries}

For the results below, we recall the definition of the (saturated) affine blowup (or dilatation) of a monoid, see \cite[\S 3]{FKato}. Let $P$ be a saturated monoid, $K\subseteq P$ an ideal, and $a\in K$ an element. We denote by $P[K-a]$ the saturation of the submonoid of $P^\gp$ generated by $P$ and $K-a$. It is fs if $P$ is. The map 
\[
    P \la P[K-a]
\]
is called an \textbf{affine blow-up} map. Note that the log schemes $\bA_{P[K-a]}$, for $a\in K$, form the standard affine cover of the (saturated) log blow-up $\Bl_K(\bA_P)$.  This explains the name of the following result. 

\begin{lem}[Valuative property of log blowups] \label{lem:monoid-VCP}
    Let $\vtheta\colon P\to V$ be a homomorphism of saturated monoids, with $V$ valuative, and let $K\subseteq P$ be a non-empty finitely generated ideal. Then there exists an $a\in K$ and a factorisation
    \[ 
        P \la P[K-a] \la V.
    \]
\end{lem}

\begin{proof}
Let $S\subseteq K$ be a finite set generating $K$ as an ideal. The map $\vtheta\colon P\to V$ induces a linear preorder on $S$: $p\geq p'$ if $\vtheta(p)-\vtheta(p')\in V$. Let $a\in S$ be an element which is minimal with respect to this preorder. The map $P^\gp\to V^\gp$ sends an element of the form $p-a$, for  $p\in K = P + S$, to $\vtheta(p)-\vtheta(a) \in V$. This implies that it sends $P[K-a]$ into $V$.
\end{proof}

\begin{thm}[{F.~Kato \cite[Theorem~1.1]{FKato}}] \label{thm:F-Kato}
    Let $P\to Q$ be a homomorphism of fs monoids. Then there exists a non-empty ideal $K\subseteq P$ such that for every $a\in K$, the induced map 
    \[ 
        P[K-a] \la \left(Q\oplus_{P} P[K-a]\right)^\sat
    \]
    is integral\footnote{Warning: the theorem involves saturated pushouts and integral morphisms. However, as we have already remarked in \cref{ex:Ogus-sat-pushout-of-int}, integral morphisms are in general not stable under saturated pushout. Therefore taking a further saturated blowup may destroy integrality of the morphism.}
    (\crefpart{defi:exactintsatvert monoids}{defitem:integral monoid map}).
\end{thm} 

In fact, \cite[Theorem~1.1]{FKato} is the log scheme version of the above result. We will deduce \cref{thm:F-Kato} from the proof of \cite[Theorem~1.1]{FKato}. Given the relationship between integrality and flatness (see the proof below), we can regard \cref{thm:F-Kato} as a monoid of analog of the Raynaud--Gruson ``flattening by blow-up'' theorem \cite{RaynaudGruson}.

\begin{proof}
The proof of \cite[Theorem~1.1]{FKato} applied to the map of log schemes $\bA_{Q}\to\bA_{P}$ constructs a non-empty ideal $K\subseteq P$ such that the morphism of schemes
\[ 
    \bA_Q\times_{\bA_P}^\sat \Bl_K \bA_P \la \Bl_K \bA_P
\]
is flat. This implies that for every $a\in K$, setting $P'=P[K-a]$ to be the affine blow-up monoid and $Q' = (P'\oplus_P Q)^\sat$, the ring homomorphism
\[ 
    \bZ[P'] \la \bZ[Q']
\]
is flat. This implies that $P'\to Q'$ is integral by \cite[Proposition~I~4.6.7 (2)$\Rightarrow$(1)]{Ogus}.
\end{proof}

\begin{cor} \label{cor:FKato}
    Let $P_0\to Q_0$ be a homomorphism of fs monoids and let $P_0\to V$ be a map into a valuative monoid $V$. Then there exists a factorisation $P_0\to P_1\to V$ where $P_0\to P_1$ is an affine blow-up $P_0[K-a]$ for some non-empty ideal $K\subseteq P_0$ and element $a\in K$ such that the saturated base change $P_1\to Q_1 = (P_1\oplus_{P_0}Q_0)^\sat$ is an integral homomorphism. 
\end{cor}

\begin{proof}
Combine \cref{thm:F-Kato} with \cref{lem:monoid-VCP}.
\end{proof}

\begin{prop} 
\label{prop:typeV-criterion}
    Let $P$ be a saturated monoid. The following are equivalent:
    \begin{enumerate}[(a)]
        \item 
        \label{propitem:typeV-sft}  $P$ is sft over a valuative submonoid $V\subseteq P$,
        \item 
        \label{propitem:typeV-sfp}        
        $P$ is sfp over a valuative submonoid $V\subseteq P$,
        \item
        \label{propitem:typeV}
        $P$ is of type $\typeV$,
        \item
        \label{propitem:typeV-sft version}
        there is a valuative monoid $V'$ and an sft morphism $V' \to P$.
    \end{enumerate}
In fact, the monoid $V$ can be taken to be the image of the morphism $V' \to P$ that is explicit in \labelcref{propitem:typeV-sft version} and implicit in \labelcref{propitem:typeV}.
\end{prop}

\begin{proof}
    By \cref{prop:sfp-conditions}
    \labelcref{propitem:sfp-sft} $\Leftrightarrow$
    \labelcref{propitem:sfp}
    we have that \labelcref{propitem:typeV-sft}$\Leftrightarrow$\labelcref{propitem:typeV-sfp}, and trivially \labelcref{propitem:typeV-sfp}   $\Rightarrow$ \labelcref{propitem:typeV}  $\Rightarrow$ \labelcref{propitem:typeV-sft version}. To show that \labelcref{propitem:typeV-sft version} implies \labelcref{propitem:typeV-sft} let $V$ be the image of $V'$ in $P$. 
    It is easy to check that $V$ is valuative. Let $S \subseteq P$ be a finite sat-generating set of $P$ over $V'$. Then $S$ is also a~finite sat-generating set for $P$ over $V$. This shows that $V \to P$ is sft.
\end{proof}

\begin{cor}[Analogue of Nagata's theorem] 
\label{cor:monoid-raynaud-gruson}
    Let $V$ be a valuative monoid and let $V\to P$ be an injective homomorphism into a saturated monoid $P$ which is finitely generated over $V$. Then $P$ is finitely presented over $V$.
\end{cor}

\begin{proof}
By \cref{prop:sfp-conditions} the morphism $V \to P$ is sfp. We may write $V = \varinjlim V_i$ as a filtered colimit of fs monoids.
Thanks to \cref{cor:sfp_comes_from_finite_level} we may assume to have a compatible colimit $P = \varinjlim P_i$ such that $V_i \to P_i$ is sfp and $P_j\simeq (P_i\oplus_{V_i} V_j)^\sat$ for $j\geq i$. This means $P_i$ is fs and the map $V_i \to P_i$ is finitely presented.  By \cref{cor:FKato}, for large enough $i$, the morphism $V_i \to P_i$ is integral. But then $P' = V \oplus_{V_i} P_i$ is integral, so that $P'\to P$ is injective. Moreover, $P' = V\oplus_{V_j} P_j$ for all $j\geq i$.

Let $S \subseteq P$ be a finite set generating $P$ as a monoid over $V$. As $P = \varinjlim P_j$, the set $S$ is contained in the image of $P_j \to P'$ for large enough $j$. In this case the map $P' \to P$ is an isomorphism, showing that $P$ is finitely presented over $V$. 
\end{proof}

\begin{rmk}
    \Cref{cor:monoid-raynaud-gruson} is a monoid version of the result of Nagata (see \cite{Nagata}, \stacks[Lemma]{053E}): a finitely generated and flat (equivalently, torsion free) algebra over a valuation ring is automatically finitely presented.  

    It is plausible that using \cite[Lemmas~2.1.7 and 2.1.9]{ALPT} one could show an integral version of \cref{cor:monoid-raynaud-gruson}, i.e.\ that if $V$ is valuative and $V\to P$ is an injective homomorphism into an integral monoid $P$ which is finitely generated over $V$, then $P$ is finitely presented over $V$. However, it seems that this would require a version of \cref{thm:F-Kato} employing un-saturated log blow-ups.
\end{rmk}

A proof of the (largely folklore) result below, a monoid analogue of N\'eron desingularisation, can be found in \cite[Theorem 6.1.31]{GabberRamero}. It is not used later on but recorded here to shed some light on valuative monoids. We present a shorter proof than the one in \cite{GabberRamero}, based on toric resolution of singularities\footnote{We thank Jarek Buczyński for help with the proof of this proposition.}. Let us call a finitely generated monoid $P$ {\bf quasi-free} if it is isomorphic to $\bN^r\times G$ where $r\geq 0$ and where $G$ is an abelian group. This is equivalent to $\Spec(\bC[P])$ being smooth over $\bC$ (see e.g.\ \cite[\S 2.1, Proposition on p.\ 29]{FultonToric}).

\begin{prop} \label{prop:valuative-ind-free}
    Let $V$ be a valuative monoid. Then $V$ is the increasing union of its quasi-free submonoids. In particular, if $V$ is sharp, then $V$ is the increasing union of its free submonoids.
\end{prop}

\begin{proof}
Let $P\subseteq V$ be a finitely generated submonoid. We shall find a quasi-free submonoid $Q\subseteq V$ containing $P$. 

To this end, we may replace $P$ with its saturation and $V$ with $V\cap P^\gp$ and hence assume that $P$ is fs and $V^\gp = P^\gp$. Now, if $Q'$ is a free submonoid of $\overline{V}$ containing the image $\overline{P}$ of $P$, then its preimage $Q\subseteq V$ is a quasi-free submonoid containing $P$. We may therefore assume that $V$ (and hence also $P$) is sharp. 

By toric resolution of singularities\footnote{The original reference \cite[Theorem~11]{KKFM1973toroidalembeddings1} would work too, except that it does not assert the constructed resolution is a blow-up, so that we cannot apply \cref{lem:monoid-VCP} but would need a version of it for a ``proper'' morphism of fans.} \cite{Niziol} there exists a non-empty ideal 
\[
    K=(a_1, \ldots, a_n)\subseteq P
\]
such that the blow-up of $\Spec(\bC[P])$ along $V(K)$ is smooth, and hence the union of $\Spec(\bC[Q_i])$ for a finite number of quasi-free monoids $Q_i = P[K-a_i]$ with $P\subseteq Q_i\subseteq P^\gp$. By \cref{lem:monoid-VCP}, this implies that there exists an element $a\in K$ such that $Q = P[K-a]$ is quasi-free and a~factorisation
\[ 
    P\la Q \la V,
\]
and since $Q^\gp = P^\gp = V^\gp$, the map $Q\to V$ is injective.
\end{proof}

\subsubsection*{Tsuji's theorem and its corollaries}
\label{ss:tsuji}

To state Tsuji's result we need some terminology. Let $P$ be an fs monoid and $\fp\subseteq P$ a prime ideal of height one, with $F=P\setminus \fp$ the corresponding face. Then $P/F = \bN$ (this is a canonical isomorphism as $\bN$ has no automorphisms), and the induced function $v_\fp\colon P\to \bN$ is the {\bf valuation along $\fp$}. Given a map of fs monoids $\vtheta \colon P\to Q$ and a~prime $\fq\subseteq Q$ of height one whose preimage $\fp\subseteq P$ is also of height one, there exists a~unique integer $e_\fq \geq 1$ such that $v_\fq\circ \vtheta = e_\fq \cdot v_\fp$, called the {\bf ramification index} of $\vtheta$ at $\fq$.

We first show that \'etale maps are ``tamely ramified'' in prime ideals of height $1$, i.e.\ the respective ramification index is prime to $p$.

\begin{lem}
\label{lem:etale is divisor tame}
    The ramification indices of an \'etale map $\vtheta \colon P \to Q$ are prime to $p$.
\end{lem}
\begin{proof}
Let $\fq \subseteq Q$ be a prime ideal of height one such that $\fp = \vtheta^{-1}(\fq)$ is also of height one. From the commutative diagram
\[ 
    \begin{tikzcd}
        P^\gp \ar[d,"v_{\fp}",swap,two heads] \ar[r,"\vtheta^\gp"] & Q^\gp \ar[d,"v_{\fq}",two heads] \ar[r] & Q^\gp/P^\gp \ar[d] \ar[r] & 0 \\
        \bZ \ar[r,"e_{\fq}",swap] & \bZ\ar[r] & \bZ/e_{\fq}\bZ \ar[r] & 0 
    \end{tikzcd}
\]
we deduce that the finite group $Q^\gp/P^\gp$ of order prime to $p$ surjects onto $\bZ/e_\fq\bZ$ and hence $e_\fq$ is prime to $p$ as well.
\end{proof}

\begin{rmk}
The assertion of \cref{lem:etale is divisor tame} is false for smooth maps. For a simple example, let $P=\bN\cdot e_1\subseteq Q=\langle e_1, e_2, -e_1+2e_2\rangle \subseteq \bZ^2$. Then $P\to Q$ is smooth for any $p$, but its ramification index along $\fp = (e_1, e_2)$ is $2$. 
\end{rmk}

\begin{thm}[{T.\ Tsuji \cite[I 5.4]{Tsuji}}] \label{thm:Tsuji}
    Let $P_1\to Q_1$ be an integral homomorphism of fs monoids and let $n\geq 1$ be an integer divisible by all of the ramification indices of $P_1\to Q_1$. Let $P_2=P_1$ and let $P_1\to P_2$ be the multiplication by $n$ map. Then, the saturated base change $P_2\to Q_2 = (P_2\oplus_{P_1}Q_1)^\sat$ is a saturated homomorphism.
\end{thm}

\begin{cor} 
\label{cor:Tsuji-cor}
    Let $P_1\to Q_1$ be an integral homomorphism of fs monoids whose ramification indices are prime to $p$. Then there exists a Kummer \'etale map $P_1\to P_2$ such that the saturated base change $P_2\to Q_2 = (P_2\oplus_{P_1}Q_1)^\sat$ is a saturated homomorphism.
\end{cor}

\begin{rmk}
For $p=1$ (i.e.\ $\Sigma=\varnothing$), \cref{thm:Tsuji} and \cref{cor:FKato} together imply that one can make a map of monoids $P_0\to Q_0$ saturated by an ``alteration'' (composition of a~Kummer \'etale map and an affine blow-up), performed ``locally'' with respect to a given map $P_0\to V$ into a~valuative monoid $V$. The following diagram illustrates both theorems.
\begin{equation} \label{eqn:Kato-Tsuji-diagram}
    \begin{tikzcd}[column sep=2.5cm]
          & Q_2 & Q_1\arrow[l] & Q_0 \arrow[l] \\
          & P_2 \arrow[ld,"\text{$V$ divisible}",dotted,swap]  \arrow[u,"\text{saturated}"] & P_1 \arrow[lld,"\text{$V$ valuative}" description,dotted] \arrow[l,"\text{Kummer \'etale}",swap] \arrow[u,"\text{integral}"] & P_0\arrow[l,"\text{affine blowup}",swap] \arrow[u] \arrow[llld] \\
          V
    \end{tikzcd}
\end{equation}
Here, the squares are pushouts in the category of fs monoids. The map $P_1 \to V$ exists (for a~suitable choice of $P_1$) by \cref{lem:monoid-VCP} since $V$ is valuative. The map $P_2\to V$ exists if $V$ is furthermore divisible, as then the Kummer \'etale map $V\to (V\oplus_{P_1}P_2)^\sat$ admits a section.

Indeed, if $P$ is a saturated monoid with $P^\gp$ divisible, and $P\to Q$ is a Kummer \'etale map, then we can write $Q^\gp = P^\gp\oplus M$ with $M$ a finite group, and then $Q = P\oplus M$ (see \cref{lem:ket-monoid-conditions}), so $P\to Q$ admits a section.
\end{rmk}

\begin{cor}[``Reduced fibre theorem''] 
\label{cor:RFT}
    Let $V$ be a valuative monoid and let $\vtheta\colon V\to P$ be an sfp map. Suppose that either $p=1$ or $V\to P$ is \'etale. Then there exists a Kummer \'etale map $V\to W$ such that the saturated base change $V\to Q = (W\oplus_V P)^\sat$ is a saturated morphism of finite presentation.
\end{cor}

\begin{proof}
There exists a map of fs monoids $V_0\to P_0$ and a map $V_0\to V$ such that $P \simeq (V\oplus_{V_0} P_0)^\sat$. If $V \to P$ is \'etale, then we may further assume by \cref{prop:monoid-descend-prop-P} that $V_0 \to P_0$ is also \'etale. 

Next we use \cref{cor:FKato} as in 
\labelcref{eqn:Kato-Tsuji-diagram}
to find an affine blow up $V_0 \to V_1$ and a factorization $V_0 \to V_1 \to V$ such that the saturated base change $V_1 \to P_1 = (V_1\oplus_{V_0} P_0)^\sat$ is integral and $P \simeq (V \oplus_{V_1} P_1)^\sat$. By assumption either $p=1$ or $V_1 \to P_1$ is \'etale. Thus all ramification indices of $V_1 \to P_1$ are prime to $p$, either trivially or by \cref{lem:etale is divisor tame}.
Then \cref{cor:Tsuji-cor} applies to $V_1 \to P_1$ providing a Kummer \'etale $V_1 \to W_1$ 
such that $W_1\to Q_1 = (W_1\oplus_{V_1} P_1)^\sat$ is a~saturated morphism (still of fs monoids and thus of finite presentation by \cref{prop:fs version of sorite for sfp}). Let $W = (V\oplus_{V_1} W_1)^\sat$, so that $V \to W$ is Kummer \'etale. Then 
\[
    Q = (W\oplus_V P)^\sat 
    = (W \oplus_{V_1} P_1)^\sat= (W \oplus_{W_1} Q_1)^\sat = W \oplus_{W_1} Q_1. \qedhere  
\]
\end{proof}

The result below is an analogue of the following theorem of Grauert and Remmert \cite[Corollary~6.4.1/5]{BGR}: if $A$ is a reduced affinoid algebra over an algebraically closed non-archimedean field $K$, then the subring of powerbounded elements $A^\circ\subseteq A$ is topologically finitely presented over the valuation ring $K^+ = K^\circ$.

\begin{cor}[``Grauert--Remmert finiteness theorem''] \label{cor:GR-finiteness}
    Let $V$ be a divisible and valuative monoid and let $V\to P$ be an sfp map. Then $V\to P$ is saturated and $P$ is finitely presented over $V$.
\end{cor}

\begin{proof}
Temporarily take $p=1$ and apply \cref{cor:RFT}, then use the fact that since $V^\gp$ is divisible, every Kummer \'etale map $V\to W$ admits a section. Indeed, since $V^\gp$ is an injective object and $V^\gp\to W^\gp$ is injective (\cref{lem:injective-implies-gp-injective}), we may write $W^\gp = V^\gp \oplus (W^\gp/V^\gp)$, and then we have $W = V\oplus (W^\gp/V^\gp)$.
\end{proof}

With an argument similar to \cref{prop:typeV-criterion}, we deduce:

\begin{prop} 
\label{prop:typeVd-criterion}
    Let $P$ be a saturated monoid. The following are equivalent:
    \begin{enumerate}[(a)]
        \item $P$ is of type $\typeVd$,
        \item $P$ is finitely presented over a divisible valuative submonoid, 
        \item $P$ is finitely generated over a divisible valuative submonoid.  
    \end{enumerate}
\end{prop}

\begin{lem}
\label{lem:localizeVVdiv}
    Let $P$ be a saturated monoid and let $F\subseteq P$ be a face. If $P$ is of type $\typeV$ (resp.\ of type $\typeVd$), then so are $F$, $P_F$, and $P/F$.
\end{lem}

\begin{proof}
We will make use of the criterion of \cref{prop:typeV-criterion} (resp.\ \cref{prop:typeVd-criterion}). Fix an sfp injective map $\vtheta\colon V\to P$ with $V$ valuative (resp.\ valuative and divisible) and let $S\subseteq P$ be a~finite sat-generating set. 

We treat $F$ first. By \cref{lem:compare-faces} we see that $F$ is sft over the valuative (resp.\ valuative and divisible) submonoid $W = F\cap \vtheta(V)$. By \cref{prop:typeV-criterion} we deduce that $W\to F$ is sfp. 

For $P_F$, consider the commutative square 
\[ 
    \begin{tikzcd}
        V\arrow[r] \arrow[d] & P \arrow[d] \\
        V_G \arrow[r] & P_F
    \end{tikzcd}
\]
where $G\subseteq V$ is the preimage of $F$. Then $V_G$ is valuative (resp.\ valuative and divisible) and $V_G\to P_F$ is injective by \cref{lem:injective-implies-gp-injective}.  We claim that this map is sft, with a sat-generating set $T = S \cup -(S\cap F)$. Indeed, an element of $P_F$ is of the form $p - f$ ($p\in P$, $f\in F$), and we can write $np = v + \sum m(s)s$ and $n'f = v' + \sum m'(s)s$ for $v,v'\in V$ and $m,m'\colon S\to\bN$ (cf.\ the proof of \cref{lem:compare-faces}). The second equation shows that $v'\in G$ and $m'(s) = 0$ for $s\notin F$. This implies that $nn'(p-f)$ belongs to the submonoid of $P_F$ generated by $V_G$ and $T$, and hence $T$ is a sat-generating set.

For $P/F$, we consider the further square
\[ 
    \begin{tikzcd}
        V_G \ar[r] \ar[d] & P_F \ar[d] \\
        V/G \ar[r] & P/F,
    \end{tikzcd}
\]
Again, $V/G$ is valuative (resp.\ valuative and divisible) and $V/G\to P/F$ is injective. Moreover, the image of $S$ (or $T$) in $P/F$ is easily seen to be a sat-generating set. This finishes the proof. 
\end{proof}

\subsection{More on Kummer \'etale maps}
\label{ss:ket-monoids}

We gather additional facts about Kummer \'etale maps of saturated monoids, to be used in \cref{s:kummer-pi1}. Recall that for the definition of a Kummer \'etale map (\cref{def:smooth-etale-Ket-monoids}) we have fixed an implicit finite set of primes $\Sigma$ and set $p = \prod_{\ell\in\Sigma}\ell$.

\begin{lem} \label{lem:ket-monoid-conditions}
    Let $u\colon P\to Q$ be a map of saturated monoids. The following are equivalent: 
    \begin{enumerate}[(a)]
        \item $u$ is Kummer \'etale (\cref{def:smooth-etale-Ket-monoids}),
        \item $u$ is injective, the quotient $Q^\gp/P^\gp$ is finite of order prime to $p$ and $P = Q\cap P^\gp$, i.e. $u$ is exact (equivalently, if $Q$ equals the saturation of $P$ in $Q^\gp$).
    \end{enumerate}
\end{lem}

\begin{proof}
The only non-trivial statement is that a map satisfying (b) is sfp. Since $u$ is injective, by \cref{prop:sfp-conditions}, this is equivalent to sft. But $Q$ is sat-generated over $P$ by any subset $S$ of $Q$ whose image generates $Q^\gp/P^\gp$: indeed, if $Q'$ is the submonoid of $Q$ generated by $P$ and $S$, then $(Q')^\gp = Q^\gp$, and if $n$ is the order of $Q^\gp/P^\gp$, then for every $q\in Q$ we have $nq \in P^\gp\cap Q = P \subseteq Q'$.
\end{proof}

If $P\to Q$ is a homomorphism of monoids, we shall say that $Q$ is {\bf finite as a $P$-set} if there exists a finite set $S_0 \subseteq S$ such that $S = P + S_0$ (cf.\ \cite[\S I 1.2]{Ogus}).  As we have seen in \cref{ex:typeV-not-fg}, if $P\to Q$ is a Kummer \'etale map between monoids of type~$\typeV$, it is not in general true that $Q$ is finite as a $P$-set. In particular, in such cases the map $\Spec(\bZ[Q])\to \Spec(\bZ[P])$ will be an integral morphism which is not of finite type. However, this issue disappears for monoids of type $\typeVd$:

\begin{lem} \label{lem:ket-finite}
    Let  $P\to Q$ be a Kummer \'etale map between monoids of type $\typeVd$. Then $Q$ is finite as a $P$-set. 
\end{lem}

\begin{proof}
By \cref{prop:typeVd-criterion}, $P$ is finitely generated over a divisible valuative submonoid $V$. If $n$ denotes the exponent of $(Q^\gp/P^\gp)$, we have inclusions $P\subseteq Q\subseteq \frac{1}{n}P$ inside $P^\gp\otimes \bQ$. If $p_1, \ldots, p_r$ generate $P$ over $V$, then $p'_i = \frac{1}{n}p_i$ generate $\frac{1}{n}P$ over $\frac{1}{n}V$, which equals $V$ because $V^\gp$ is divisible and $V$ is saturated. In particular, $p'_1, \ldots, p'_r$ generate $\frac{1}{n}P$ over $P$. Let $S\subseteq \frac{1}{n}P$ be the finite set of all sums $\sum_{i=1}^r n_i p'_i$ with $n_i\in \{0, \ldots, n-1\}$. Then $S$ generates $\frac{1}{n}P$ as a $P$-set. We check that $S\cap Q$ generates $Q$ as a $P$-set: for $q\in Q$, we can write $q = p + s$ for some $s\in S$ and $p\in P$, and then $s = q - p \in S\cap Q^\gp = S\cap Q$. 
\end{proof}

It is well-known that Kummer \'etale maps are not flat in general (consider e.g.\ the quotient map $\bA^2_\bQ\to \bA^2_\bQ/\mu_2$). For Kummer \'etale maps between valuative monoids, this does not occur.

\begin{lem} \label{lem:kummer-of-valuative}
    Let $V$ be a valuative monoid and let $V\to W$ be a Kummer \'etale map. Then $W$ is valuative and is ind-free as a $V$-set. In particular, the homomorphism
    \[
        \bZ[V] \la \bZ[W]
    \]
    is flat.
\end{lem}

\begin{proof}
Let $x\in W^\gp$, and let $n\geq 1$ be such that $nx\in V^\gp$. Then either $nx$ or $-nx$ belongs to $V$ and hence to $W$. Since $W$ is saturated, one of $x$ or $-x$ belongs to $W$, showing that $W$ is valuative.

For the second assertion, for $t\in W^\gp/V^\gp$, let $W_t\subseteq W$ be the intersection of the corresponding coset with $W$. Then $W = \coprod_t W_t$ as $V$-sets. Moreover, each $W_t$ can be identified with a fractional ideal $J_t$ of $V$ (i.e., a sub-$V$-set of $V^\gp)$: pick $w\in W_t$ and consider the map $v\mapsto w+v\colon V^\gp\to W^\gp$. The preimage $J_t$ of $W_t$ maps bijectively onto $W_t$. We can write $J_t$ as the inductive limit of its finitely generated sub-$V$-sets. Since $V$ is valuative, every finitely generated fractional ideal is principal and free. Thus each $W_t$ is and hence $W$ is ind-free.
\end{proof}

\begin{lem}
    Let $P\to Q$ be a Kummer \'etale map between saturated monoids. Then, the induced map $\Spec(Q)\to \Spec(P)$ is a bijection.
\end{lem}

\begin{proof}
Let $F$ be a face of $P$ and let $G$ be its saturation in $Q$. It is easy to check that $G$ is a face of $Q$, and that this produces an inverse map $\Spec(P)\to \Spec(Q)$.
\end{proof}

The following lemma is well known for fs monoids. The argument works identically also for saturated monoids.

\begin{lem}[{Vidal, \cite[Lemme 3.3]{IllusieFKN}}]
\label{lem:Vidal lemma on fs diagonal}
    Let $u \colon P \to Q$ be a morphism of saturated monoids such that the cokernel of $u^\gp \colon P^\gp \to Q^\gp$ is torsion.
    Then the map
    \[
    \psi \colon (Q \oplus_P Q)^{\sat} \longrightarrow Q \oplus (Q^{\gp}/P^{\gp}), \qquad (a,b) \mapsto (a+b,b)
    \]
    is an isomorphism.
\end{lem}

Next, we discuss the ``Kummer \'etale fundamental group'' of a saturated monoid $P$. Recall that in \cref{def:smooth-etale-Ket-monoids} we fixed a set of primes $\Sigma$ and denoted their formal product by $p$. Let 
\[
    \bZ_{\Sigma} =  \bZ[{\textstyle \frac{1}{n}}\,:\, (n, p)=1] \subseteq \bQ.
\]

By \cref{lem:ket-monoid-conditions}, there is an equivalence between the category of Kummer \'etale maps $P\to Q$ and the category of group extensions 
\[ 
    \begin{tikzcd}
        0\ar[r] & P^\gp \ar[r] &  M\ar[r] &  N\ar[r] &  0    
    \end{tikzcd}
\]
where $N$ is finite of order prime to $p$. Let us call a monoid connected if its associated group is torsion free. If $P$ is connected, then the category of connected Kummer \'etale $Q$ over $P$ is identified with the poset of subgroups $M\subseteq P^\gp\otimes_\bZ \bZ_\Sigma$ containing $P^\gp$ as a subgroup of finite index.

For the definition below, we pick an abelian group $\mu$ isomorphic to $\bZ_\Sigma/\bZ$ and denote by $\mu_n$ or $\bZ/n\bZ(1)$ its $n$-torsion subgroup. Later, we will mainly use $\Sigma =\{p\}$, and $\mu$ will signify $\mu(k) = \varinjlim_n \mu_n(k)$ for an algebraically closed field $k$ of characteristic exponent $p$. We denote by 
\[
    \widehat{\bZ}'(1) = \Hom(\bZ_\Sigma/\bZ, \mu) = \varprojlim_{(n,p) = 1} \bZ/n\bZ(1) 
\]
its Tate module, which is a rank one free module over the prime-to-$p$ (or prime-to-$\Sigma$) completion of $\bZ$.  

\begin{defi} \label{def:pi1-of-monoid}
    Let $P$ be a saturated monoid. 
    We define the {\bf fundamental group} of $P$ as
    \[ 
        \pi_1(P) = \Hom(P^\gp, \widehat{\bZ}'(1)).
    \]
\end{defi}

We note that there is a natural identification 
\[
    \pi_1(P) = \Hom(P^\gp, \Hom(\bZ_{\Sigma}/\bZ, \mu)) =  \Hom(P^\gp \otimes (\bZ_{\Sigma}/\bZ), \mu) = \Hom((P_\infty^\gp\otimes \bZ_{\Sigma})/P^\gp, \mu)  
\]
In particular, $\pi_1(P)$ has the structure of a profinite group.

\begin{lem}
    Let $P$ be a monoid of type $\typeVd$. Then, the following hold. 
    \begin{enumerate}[(a)]
        \item The fundamental group $\pi_1(P)$ is finitely generated (for arbitrary $p$).
        \item For every $n\geq 1$, the group $P^\gp\otimes\bZ/n\bZ$ is finite. 
    \end{enumerate}
\end{lem}

\begin{proof}
By \cref{prop:typeVd-criterion} we may choose an injection $V \inj P$ with $V$ divisible valuative and $P$ finitely generated over $V$. As $V^\gp$ is divisible there is a non-canonical isomorphism $P^\gp \simeq V^\gp \oplus P^\gp/V^\gp$. The claims follow from $V^\gp/nV^\gp = 0$ for all $n$.
\end{proof}

\section{Log schemes beyond fs}
\label{s:log-schemes}

In this section we develop logarithmic geometry beyond the case of fs log schemes. The most important notion is that of an sfp morphism (\cref{defi:sfp log map}) which is modelled on the notion of an sfp map of monoids from \cref{ss:sfp-monoids}.

\subsection{Preliminaries on saturated quasi-coherent log schemes}
\label{ss:log-sch-prelims}

A log scheme is {\bf quasi-coherent} if it admits a chart \'etale locally \cite[\S 2]{Kato1989:LogarithmicStructures}. For a monoid $P$ we denote by $\bA_P$ the scheme $\Spec(\bZ[P])$ endowed with the log structure associated to the natural map $P \to \bZ[P]$. By \cite[Proposition~III~1.2.4]{Ogus}, a chart $P \to \cM_X(X)$ on the log scheme $X$ is the same as a~strict map of log schemes $X \to \bA_P$. 

If $f\colon X \to Y$ is a strict map and $Y$ is quasi-coherent, then $X$ is quasi-coherent as well, simply by pulling back charts. Alternatively, this endows $\underline{X}$ with the pullback log structure, which remains quasi-coherent.

A log scheme is \textbf{integral} (resp. \textbf{saturated}) if \'etale locally it admits a chart by an integral (resp. saturated) monoid. In particular integral (resp.\ saturated) log schemes are quasi-coherent by definition. The pullback log structure of an integral (resp.\ saturated) log structure is again integral (resp.\ saturated). 

\begin{rmk} \label{rmk:warning-sat-int-nonstandard}
We warn the reader that our convention is slightly non-standard. In the literature (e.g.\ in \cite{Ogus}), integral (resp.\ saturated) log schemes are defined as log schemes $X$ such that for every \'etale $U\to X$, the monoid $\cM_X(U)$ is integral (resp.\ saturated). It follows from \cref{lem:integral and saturated for qcoh log schemes} below that a log scheme $X$ is integral (resp.\ saturated) in our sense if and only if it is integral (resp.\ saturated) in the more standard sense and quasi-coherent. This discrepancy will not cause issues since all of the log schemes considered in this paper are quasi-coherent.
\end{rmk}

We spell out equivalent conditions for integral (resp.\ saturated) log schemes (well known for fs log schemes). 

\begin{lem} \label{lem:integral and saturated for qcoh log schemes}
    Let $X$ be a quasi-coherent log scheme. 
    \begin{enumerate}[(1)]
    \item 
    The following are equivalent:
    \begin{enumerate}[(a)]
        \item 
        \label{lemitem:integral log scheme 1}
        \'Etale locally on $X$, there exists a chart $P\to \cM_X$ with an integral (resp.\ saturated) monoid $P$, i.e.\ $X$ is integral (resp.\ saturated).
        \item 
        \label{lemitem:integral log scheme 2}
        For every \'etale $U\to X$, the monoid $\cM_X(U)$ is integral (resp.\ saturated).
        \item 
        \label{lemitem:integral log scheme 3}
        For every geometric point $\overline{x}$, the monoid $\cM_{X,\overline{x}}$ is integral (resp.\ saturated). 
        \item 
        \label{lemitem:integral log scheme 4}
        For every geometric point $\overline{x}$, the monoid $\overline{\cM}_{X,\overline{x}}$ is integral (resp.\ saturated) and $\cM_{X,\overline{x}}$ is $u$-integral (i.e. $\cO_{X,\overline{x}}^\times$ acts freely on $\cM_{X,\overline{x}}$).
    \end{enumerate}
    \item \label{lemitem:charts of the natural sort}
    Moreover, if $X$ is integral (resp.\ saturated) and $P\to \cM(X)$ is a chart with $P$ arbitrary, then the induced map $P^\inte\to \cM(X)$ (resp.\ $P^\sat\to \cM(X)$) is a chart as well.
    \end{enumerate}
\end{lem}

\begin{proof}
\labelcref{lemitem:integral log scheme 1} $\Rightarrow$ \labelcref{lemitem:integral log scheme 4}: The stalk $\cM_{X,\overline{x}}$ is the pushout of $P\leftarrow F_{\overline{x}}\to \cO_{X,\overline{x}}^\times$ where the face $F_{\overline{x}}$ is the preimage of $\cO_{X,\overline{x}}^\times$ under $P\to \cM_{X,\overline{x}}$. Since $P$ is integral and $F_{\overline{x}} \to P$ injective, the pushout $\cM_{X,\overline{x}}$ is $u$-integral. Moreover, $\overline{\cM}_{X,\overline{x}}$ equals the quotient $P/F_{\overline{x}}$  and thus is integral. 

\labelcref{lemitem:integral log scheme 4} $\Leftrightarrow$ \labelcref{lemitem:integral log scheme 3}: 
By \cite[Proposition~I~1.3.3]{Ogus} a monoid $Q$ is integral if and only if $Q$ is $u$-integral and $\overline{Q} = Q/Q^\times$ is integral. The assertion about ``saturated'' then follows easily.

\labelcref{lemitem:integral log scheme 3}  $\Leftrightarrow$ \labelcref{lemitem:integral log scheme 2}: This is part of \cite[Proposition~II~1.1.3]{Ogus}.

\labelcref{lemitem:integral log scheme 2} and 
\labelcref{lemitem:integral log scheme 3}  imply \labelcref{lemitem:integral log scheme 1}: 
Let $P\to \cM_X(U)$ be a local chart, with $P$ arbitrary. Since $\cM_X(U)$ is integral (resp.\ saturated) by \labelcref{lemitem:integral log scheme 2}, this map factors through $P^\inte$ (resp.\  $P^\sat$). We claim that the resulting map
\[ 
    P^\inte\to \cM_X(U)
    \qquad
    \text{(resp.\ $P^\sat\to \cM_X(U)$)}
\]
is another chart for $\cM_X|_U$ (which will also show (2)). Let $\cM_U^i$ (resp.\ $\cM_U^s$) denote the log structure induced by $P^\inte$ (resp.\ $P^\sat$) on $U$. Then there are natural maps $\cM_X|_U \to \cM_U^i \to \cM_U^s$, 
which we show are isomorphisms by computing the stalks at a geometric point $\bar u$ of $U$.  Integralisation is a~left adjoint, so commutes with pushouts. Therefore, and since $\cM_{X,\overline{u}}$ is integral by assumption~\labelcref{lemitem:integral log scheme 3}, we have
\[
    \cM_{X,\overline{u}}  
    = \cM_{X,\overline{u}}^\inte 
    = \big( P \oplus_{\alpha^{-1}(\cO_{X,\bar u}^\times)} \cO_{X,\bar u}^\times\big)^\inte
    \isomto 
    P^\inte \oplus_{(\alpha^\inte)^{-1}(\cO_{X,\bar u}^\times)} \cO_{X,\bar u}^\times 
    = \cM_{U,\overline{u}}^i 
\]
where $\alpha\colon P\to \cM_{X,\overline{u}}$ is the induced map.
Here we use that for monoids $A$, $B$ and $C$ we have $A^\inte \oplus_{B^\inte} C^\inte  = A^\inte  \oplus_B C^\inte $, which is true since $B \to B^\inte$ is surjective. This shows the claim in the integral case. 

Furthermore, saturation is a left adjoint, so commutes with pushouts. Therefore, since now $\cM_{X,\overline{u}}$ is saturated by assumption \labelcref{lemitem:integral log scheme 3}, we have
\[
    \cM_{X,\overline{u}} 
    = \cM_{U,\overline{u}}^{\sat}
    \isomto (\cM_{U,\overline{u}}^i)^{\sat} 
    = \big( P^\inte \oplus_{\alpha^{-1}(\cO_{X,\bar u}^\times)} \cO_{X,\bar u}^\times\big)^{\sat} 
    \isomto P^\sat \oplus_{\alpha^{-1}(\cO_{X,\bar u}^\times)} \cO_{X,\bar u}^\times 
    = \cM_{U,\overline{u}}^s .
\]
Here we use that for integral monoids $A$, $B$ and $C$ we have $A^\sat \oplus_{B^\sat} C^\sat  = A^\sat \oplus_B C^\sat$, which is true since $B \to B^\sat$ is an epimorphism in the category of integral monoids. This shows the claim in the saturated case. 

These arguments also show assertion (2). 
\end{proof}

\begin{rmk}[Chart lifting property] \label{rmk:chart-lifting} 
    It is not clear whether every map between quasi-coherent log schemes locally admits a chart. Moreover, we do not expect the class of maps which admit local charts to be closed under composition, and we do not know if for maps $Y_0\to X$ and $Y_1\to X$ admitting local charts the fibre product $Y_0\times_X Y_1$ is quasi-coherent. We will avoid such issues by considering sfp morphisms, see \cref{sec:sfp log maps}  below, because such morphisms are examples for the following notion.

    We say that a morphism $f\colon X \to S$ of quasi-coherent log schemes has the \textbf{chart lifting property} with respect to a chart $\alpha\colon P \to \cM_S(U)$ on an \'etale open $U \to S$ 
    if the following holds. There is an \'etale covering $\{V_i\to U\times_S X\}$, monoid homomorphisms $\vtheta_i\colon P\to Q_i$, and commutative squares
    \[ 
        \begin{tikzcd}
            V_i \ar[r] \ar[d] & \bA_{Q_i} \ar[d,"\vtheta_i"] \\
            U \ar[r,"\alpha",swap] & \bA_P
        \end{tikzcd}
    \]
    where the horizontal arrows are strict. So $U \times_S X$ locally has a chart for $f$ compatible with the given chart on $U$.     
    We say that a morphism $f\colon X \to S$ of quasi-coherent log schemes has the \textbf{chart lifting property} if it has the chart lifting property with respect to every \'etale local chart of $S$.

    By the final assertion \labelcref{lemitem:charts of the natural sort} of \cref{lem:integral and saturated for qcoh log schemes}, if $X$ and $S$ are integral (resp.\ saturated), then $X\to S$ has the chart lifting property if and only if the above condition holds with $P$ and $Q_i$ integral (resp.\ saturated).

    Moreover, one easily shows that the composition of two maps with the chart lifting property has the chart lifting property, and that the fibre product of two maps $Y\to S$ and $X\to S$ with the chart lifting property is quasi-coherent, with the two projections $X\times_S Y\to X$, $X\times_S Y\to Y$ also with the chart lifting property. Moreover, since the saturation (resp.\ integralisation) map has the chart lifting property  (essentially by construction, see \cref{rmk:saturation-integralization} below), the last assertion holds for fibre products in integral (resp.\ saturated) log schemes. Furthermore, strict maps have the chart lifting property, of course.
\end{rmk}

\begin{rmk}
\label{rmk:refer to Tsuji exmaple in appendix}
    It would be easy to construct charts for morphisms of log schemes if for affine log schemes $X$ admitting local charts, the global sections $P = \cM(X)$ would yield a chart $\id \colon P \to \cM(X)$. 
    \Cref{ex:tsuji} due to Takeshi Tsuji gives a counterexample against this optimistic constructions of charts.
\end{rmk}

The following lemma provides local charts for maps under some conditions. 
Recall that an integral quasi-coherent log scheme $X$ is \textbf{fine} if \'etale locally it admits a chart by a fine (finitely generated integral) monoid.

\begin{lem}[{Beilinson \cite[\S1.1]{Beilinson2013:CrystallinePeriodMap}}]
\label{lem:Beilinson chart lemma}
    Let $f\colon X \to S$ be a morphism from a quasi-coherent log scheme $X$ to a log scheme $S$. Then $f$ has the chart lifting property with respect to charts $P\to \cM(S)$ on $S$ by finitely generated monoids $P$.
    
    More precisely, if $X$ is quasi-coherent (resp.\ integral, resp.\ saturated) and $\mu\colon P \to \cM(S)$ is a chart by a finitely generated (resp.\ fine, resp.\ fs) monoid $P$, then \'etale locally on $X$ we can find a morphism $\vtheta \colon P \to Q$ of monoids (resp.\ integral monoids, resp.\ saturated monoids) and a~commutative diagram
        \[
         \begin{tikzcd}
            X  \ar[r]  \ar[d,"f",swap]  & \bA_Q   \ar[d,"\vtheta"]  \\
            S  \ar[r,"\mu",swap]              & \bA_P
        \end{tikzcd}
         \]
    where the horizontal maps are strict.
\end{lem}

\begin{proof}
\cite[\S1.1]{Beilinson2013:CrystallinePeriodMap} asserts chart lifting for $f \colon X \to S$ for fine charts on $S$ and with $X$ integral. The proof essentially carries over; we revisit Beilinson's proof, which in turn is based on a proof by Kato.

Let $\bar x \to X$ be a geometric point and replace $X$ with a neighbourhood of $\bar x$ such that $X$ has a~global chart $\nu_0\colon Q_0 \to \cM(X)$. Let $p_1,\ldots,p_n$ be a finite set of generators of $P$. After replacing $X$ with a neighbourhood of $\bar x$, we may assume that there exist $q_i \in Q_0$ and units $u_i \in \cO^\times(X)$ such that $f^\ast(\mu(p_i)) = \nu_0(q_i) \cdot u_i$ in $\cM(X)$. The map
\[
    \nu' \colon Q' = Q_0 \oplus \bZ^n \la \cM(X), \quad \nu'(q,a_1,\ldots,a_n) = \nu_0(q)\cdot \prod_{i=1}^n u_i^{a_i}
\]
is also a chart for $X$. We write $Q$ for the image of $\nu'$.  The map $f^\ast \circ \mu$ factors through a map $\vtheta \colon P \to Q$ because this can be checked on a generating set for $P$. 
The assertion now follows because by \cref{lem:image-also-a-chart} below, the inclusion $\nu \colon Q \to \Gamma(X,\cM_X)$ is also a chart for $X$. 

If $X$ and $S$ are integral (resp.\ saturated) and $P$ is fine (resp.\ fs),  then we may replace $\vtheta\colon P \to Q$ as constructed above by the integralisation $P \to Q^\inte$ (resp.\ saturation $P = P^\sat \to Q^\sat$). The chart $Q \to \cM(X)$ factors through a map $Q^\sat \to \cM(X)$ by \cref{lem:integral and saturated for qcoh log schemes}.
This is still a chart by the final assertion \labelcref{lemitem:charts of the natural sort} of \cref{lem:integral and saturated for qcoh log schemes}. 
\end{proof}

The proof used the following simple lemma.

\begin{lem} \label{lem:image-also-a-chart}
    Let $X$ be a log scheme, let $Q'$ be a monoid, and let $\nu'\colon Q'\to \cM(X)$ be a chart. Let $Q\subseteq \cM(X)$ be the image of this map. Then, the inclusion $\nu\colon Q\to \cM(X)$ is a chart as well.
\end{lem}

\begin{proof}
The maps of prelog structures
$\underline{Q}{}'_X \surj \underline{Q}{}_X \to \cM_X$
yield maps of associated log structures
\[
    (\underline{Q}'{}_X)^{\log} \la (\underline{Q}{}_X)^{\log} \la \cM_X
\]
whose composition is an isomorphism. Thus the first map is injective. Moreover, associated log structure is a left adjoint functor, so it preserves epimorphisms, and hence the first map is an epimorphism as well. But, in any category, a split monomorphism which is an epimorphism is an isomorphism. Therefore $(\underline{Q}{}_X)^{\log} \to \cM_X$ is an isomorphism, and  $\nu \colon Q \to \cM(X)$ is a chart. 
\end{proof}

On the existence of fibre products in quasi-coherent log schemes we have the following lemma.

\begin{lem}
\label{lem:existence of qcoh fibre products part 1}
Let $f\colon X \to S$ and $g\colon Y \to S$ be maps of quasi-coherent log schemes. The fibre product $X \times_S Y$ in the category of log schemes is quasi-coherent if one of the following holds.
\begin{enumerate}[(1)]
    \item 
    \label{lemitem:qcoh fibre product:strict}
    One of the maps $f$ or $g$ is strict. 
    \item   
    \label{lemitem:qcoh fibre product:charts}
    \'Etale locally on $X$, $Y$, and $S$ there exists a chart 
    $P \to \cM(S)$ and charts $P \to Q$ and $P \to R$ for $f$ and $g$, 
    i.e., $Q \to \cM(X)$ and $R \to \cM(Y)$ are charts and the obvious maps commute.
    \item 
    \label{lemitem:qcoh fibre product:over fine}
    The log schemes $X$ and $Y$ are integral and $S$ is a fine log scheme. 
\end{enumerate}
\end{lem}
\begin{proof}
    \labelcref{lemitem:qcoh fibre product:strict}:
    Let us assume $f$ is strict. Then the projection $X \times_S Y \to Y$ is also strict and charts for $Y$ pull back to charts for $X \times_S Y$.

    \labelcref{lemitem:qcoh fibre product:charts}: 
    The fibre product can be constructed \'etale locally and the property of being quasi-coherent is \'etale local as well. So we may assume that both maps have charts with respect to a~common chart of $S$ as in the statement. Then the map 
    \[
    X \times_S Y \la \bA_Q  \times_{\bA_P} \bA_R = \bA_{Q \oplus_P R}
    \]
    is strict showing that $X \times_S Y$ is quasi-coherent.

    \labelcref{lemitem:qcoh fibre product:over fine}:
    The assumptions on charts in \labelcref{lemitem:qcoh fibre product:charts} are met thanks to Beilinson's \cref{lem:Beilinson chart lemma}.  
\end{proof}

\begin{rmk} \label{rmk:saturation-integralization}
The inclusions of saturated log schemes into integral log schemes and further to quasi-coherent log schemes admit right adjoints. These are the \textbf{integralisation} of $X$ denoted by $X^\inte$ and the \textbf{saturation} denoted by $X^\sat$. 

The construction of saturation uses three steps (and works verbatim for integralisation). First, if $X \to Y$ is strict and $Y^\sat$ exists, then $X^\sat  = X \times_Y Y^\sat$ with the fibre product as log schemes and the projection $X^\sat \to Y^\sat$ is strict. It therefore suffices to construct the saturation \'etale locally on $X$, because the local saturations will then automatically form a descent datum and glue to a~saturation globally (gluing on the level of the underlying schemes is permissible since the morphisms to $X$ are affine). We may thus assume that we have a global chart, and in fact that $X = \bA_P$ with the usual log structure induced by $P$. In this case $X^\sat$ is simply $\bA_{P^\sat}$. 

We note here that the maps $X^\sat\to X^\inte$ and $X^\inte \to X$ as well as their composition have the chart lifting property, essentially by construction.
\end{rmk}

We can now state three cases where saturated fibre products exist for saturated log schemes.

\begin{lem}
\label{lem:saturated fibre product}
    In the category of integral (resp.\ saturated) log schemes the fibre product, denoted by  
    $X \times_S^\inte Y$ (resp.\ $X \times_S^{\sat} Y$), exists if one of the following holds:
    \begin{enumerate}[(1)]
    \item 
    One of the maps $f\colon X \to S$ or $g\colon Y \to S$ is strict: then the fibre product $X \times_S Y$ as log schemes is already integral (resp.\ saturated).
    \item 
    \label{lemitem:saturated fibre product:charts}
    \'Etale locally on $X$, $Y$, and $S$ there exist charts as in 
    \crefpart{lem:existence of qcoh fibre products part 1}{lemitem:qcoh fibre product:charts}.
    \item 
    The log scheme $S$ is a fine (resp.~fs) log scheme. 
\end{enumerate}

In all these cases the integral (resp.\ saturated) fibre product has the form
\begin{align*}
    X \times_S^\inte Y 
    & = (X \times_S Y)^\inte, \\
    X \times_S^{\sat} Y 
    & =     (X \times_S Y)^{\sat}.
\end{align*}
\end{lem}
\begin{proof}
\Cref{lem:existence of qcoh fibre products part 1} shows that in each of the cases $X \times_S Y$ is quasi-coherent. 
\end{proof}
    
We next prove some scheme-theoretic properties of integralisation and saturation. 

\begin{prop}
\label{prop:geometry sat map}
    Let $X$ be a quasi-coherent log scheme. 
    \begin{enumerate}[(1)]
        \item 
        \label{propitem:int map closed imm}
        The morphism of schemes underlying the integralisation map $X^\inte\to X$ is a closed immersion.
        \item 
        \label{propitem:sat map integral}
        The morphism of schemes underlying the saturation map $X^\sat\to X$ is integral.\footnote{It is a bit unfortunate that the term ``integral'' appears in this article in several entirely different meanings: integral scheme, integral morphism of schemes, integral monoid/log scheme, integral morphism of monoids/log schemes. Although the intended meaning should always be clear from context, we nevertheless advise the reader to pay close attention.}
        \item 
        \label{propitem:sat map surj}
        If moreover $X$ is an integral log scheme, then $\nu$ is surjective.
    \end{enumerate}
\end{prop}

\begin{proof}
The scheme-theoretic properties in these assertions are all stable under base change and \'etale local on the target. By construction of integralisation and saturation, it is therefore enough to consider $X = \bA_P$ for a monoid $P$, i.e.\ the morphisms
\[ 
    \bA_{P^\sat} \la \bA_{P^\inte} \la \bA_P
\]
Since for every monoid $P$, the map $P\to P^\inte$ is surjective, the map $\bA_{P^\inte}\to \bA_P$ is a closed immersion, showing \labelcref{propitem:int map closed imm}. For \labelcref{propitem:sat map integral}, we note that for every $q\in P^\sat$ we have $nq = \iota(p)$ for some $p\in P$ and some $n \in \bN$, where $\iota\colon P\to P^\sat$ is the canonical map. This means that the corresponding element $q$ of $\bZ[P^\sat]$ satisfies the integral equation $T^n = \iota(p)$. 

To show \labelcref{propitem:sat map surj}, it is enough to prove that every homomorphism $\vtheta\colon P\to k$ to an algebraically closed field extends to $P^\sat$. The face $F = \vtheta^{-1}(k^\times)$ of $P$ has a unique extension to a face $F' \subseteq P^\sat$, see \cref{rmk:spec of saturation map}. Moreover, since $P$ is integral, the map $F\hookrightarrow F'$ is injecitve, and hence by \cref{lem:injective-implies-gp-injective} also $F^\gp\hookrightarrow (F')^\gp$ is injective. The map $F\to k^\times$ extends uniquely to $F^\gp$. Since $k^\times$ is divisible and hence an injective object in the category of abelian groups, there exists a $\vtheta'\colon (F')^\gp\to k^\times$ with $\vtheta'|_F=\vtheta|_F$. We then extend $\vtheta'|_{F'}$ to $\vtheta'\colon P^\sat\to k$ by sending $P^\sat\setminus F'$ to zero.
\end{proof}

\subsection{Prelog rings}

In this subsection we discuss prelog rings and the associated log schemes. 

A {\bf prelog ring} is a monoid homomorphism $P\to A$ from a commutative monoid $P$ (written additively) to the underlying multiplicative monoid of a commutative ring $A$ (equivalently a ring homomorphism $\bZ[P]\to A$). 
Prelog rings form a category in the obvious way. This category has filtered colimits and is generated under filtered colimits by its compact objects, which are precisely the prelog rings $P\to A$ where $P$ is finitely presented and $A$ is of finite type over $\bZ$ (use \cref{prop:fp-conditions} which implies that the compact monoids are the finitely presented ones).  

A prelog ring $P\to A$ is {\bf saturated} if $P$ is saturated. The inclusion of saturated prelog rings into prelog rings admits a left adjoint, which is given by
\[ 
    (P\to A)^\sat = \big(P^\sat \to A\otimes_{\bZ[P]} \bZ[P^\sat]\big).
\]
Indeed, given a map $(P\to A)\to (Q\to B)$ with $Q$ saturated, we get a unique factorisation $P\to P^\sat\to Q$, and then a unique $A\otimes_{\bZ[P]} \bZ[P^\sat]\to B$. The category of saturated prelog rings has small colimits, with pushouts given by
\[ 
    (P\to A)\otimes_{(P_0\to A_0)} (Q_0\to B_0) = ((P\oplus_{P_0} Q_0)^\sat \to A\otimes_{A_0} B_0,
\]
and is generated under filtered colimits by its compact objects which are precisely the $(P\to A)$ with $P$ fs and $A$ of finite type over $\bZ$ (see \cref{cor:prelog-comp-obj} below).

\begin{prop} 
\label{prop:compact-saturated-prelog-ring}
    Let $(P\to A)\to (Q\to B)$ be a morphism of saturated prelog rings. Then $(Q\to B)$ is a compact object in the category of saturated prelog rings over $(P\to A)$ if and only if
    \begin{enumerate}[(i)]
        \item 
        \label{lemitem:cspl1}
        the homomorphism of monoids $P\to Q$ is sfp, and
        \item 
        \label{lemitem:cspl2}
        the induced homomorphism of rings $A\otimes_{\bZ[P]}\bZ[Q]\to B$ is of finite presentation.
    \end{enumerate}
\end{prop}

\begin{defi} \label{def:sfp-prelog-ring-map}
    A morphism of saturated prelog rings $(P\to A)\to (Q\to B)$ is {\bf sfp} if the equivalent conditions of \cref{prop:compact-saturated-prelog-ring} are satisfied.
\end{defi}

\begin{proof}[Proof of \cref{prop:compact-saturated-prelog-ring}]
The proof is entirely straightforward; we include it for completeness. Suppose \labelcref{lemitem:cspl1} and \labelcref{lemitem:cspl2} hold, let $(M\to R) = \varinjlim (M_i\to R_i)$ be a filtered colimit of saturated prelog rings over $(P\to A)$ and let $(Q\to B)\to (M\to R)$ be a morphism over $(P\to A)$. 

Since $P\to Q$ is sfp by \labelcref{lemitem:cspl1}, $Q$ is a compact object in the category of saturated monoids over $P$ (\cref{prop:sfp-conditions}), and hence the morphism $Q\to M$ factors as $Q\to M_{i_0}\to M$ for some index $i_0$. 
Then $(R_i)_{i\geq i_0}$ becomes a system of $A\otimes_{\bZ[P]}\bZ[Q]$-algebras. 

For any prelog ring $(M'\to R')$ over $(P\to A)$, extending a given homomorphism $Q\to M'$ over $P$ to a morphism of prelog rings $(Q\to B)\to (M'\to R')$ over $(P\to A)$ is the equivalent to finding a homomorphism $B\to R'$ of $A\otimes_{\bZ[P]}\bZ[Q]$-algebras. By condition \labelcref{lemitem:cspl2}, $B$ is a compact object in the category of $A\otimes_{\bZ[P]}\bZ[Q]$-algebras, and hence the $A\otimes_{\bZ[P]}\bZ[Q]$-algebra homomorphism $B\to R$ above factors through $B\to R_i$ for some index $i\geq i_0$, providing a factorisation 
\[
    (Q\to B)\la (M_i\to R_i)\la (M\to R).
\]
Thus $(Q\to B)$ is compact.

\smallskip

For the converse direction we suppose that $(Q\to B)$ is compact. To show \labelcref{lemitem:cspl1}, let $Q \to M$ be a morphism over $P$ with $M = \varinjlim M_i$ a filtered colimit of saturated monoids over $P$. Then $(M\to 0)=\varinjlim (M_i\to 0)$ as saturated prelog rings over $(P \to A)$. Therefore the map 
\[
    (Q\to B) \la (M\to 0)
\]
factors through $(M_i\to 0)$ for some $i$, and then $Q\to M$ factors as $Q\to M_i\to M$ over $P$, so that $Q$ is compact. By 
\cref{prop:sfp-conditions} property \labelcref{lemitem:cspl1} follows.

To show \labelcref{lemitem:cspl2}, let $R = \varinjlim R_i$ be a filtered colimit of $A\otimes_{\bZ[P]}\bZ[Q]$-algebras, and let $B\to R$ be an $A\otimes_{\bZ[P]}\bZ[Q]$-algebra homomorphism. Then $(Q\to R)=\varinjlim (Q\to R_i)$ in saturated prelog rings over $(P\to A)$, and thus the map $(Q\to B)\to (Q\to R)$ factors through $(Q\to R_i)$ for some $i$, so that $B\to R$ factors through $R_i$. 
\end{proof}

The following corollary is immediate.

\begin{cor} \label{cor:prelog-comp-obj}
    A saturated prelog ring $(P \to A)$ is a compact object in the category of saturated prelog rings if and only if  $P$ is an fs monoid and $A$ is of finite type over $\bZ$.
\end{cor}

\begin{cor}
    \label{cor:compose sfp monoid mor}
    The composition of sfp morphisms of prelog rings is again sfp.
\end{cor}
\begin{proof}
    This is formal using the characterisation as compact objects. 
\end{proof}

\begin{prop}
\label{prop:colimit descent for sfp prelog maps}
    Let $P \to A = \varinjlim_i (P_i \to A_i)$ for a filtered direct system of prelog rings, and let 
    $(P \to A) \to (Q \to B)$ be an sfp morphism of prelog rings. 
    
    Then, for large enough $i$, there is an sfp morphism $(P_i \to A_i) \to (Q_i \to B_i)$  and a cocartesian diagram 
    \[
    \begin{tikzcd}
        (Q \to B)  
        & \ar[l]  (Q_i \to B_i)   \\
        (P \to A) \ar[u] 
        & \ar[l] (P_i \to A_i) \ar[u]
        \end{tikzcd}    
    \]
    in saturated prelog rings.
\end{prop}
\begin{proof}
    By \cref{prop:sfp-conditions} there exists a pushout square 
    \[
     \begin{tikzcd}
         Q          & Q'   \ar[l]  \\
         P  \ar[u]  & P'   \ar[l]  \ar[u]
     \end{tikzcd}
    \]
    in the category of saturated monoids where $P'\to Q'$ is a morphism of fs monoids. Since $P'$ is a compact object in saturated monoids by \cref{cor:fs_compact}, the map $P' \to P = \varinjlim P_i$ factors as $P' \to P_i$ for $i$ large enough. We fix an initial choice $i_0$ and restrict the filtered system to all indices $i \geq i_0$. Thus the system $(P_i)$ restricts to a system of monoids over $P'$. For all such $i$ we set $Q_i = (P_i \oplus_{P'} Q')^\sat$ so that $Q = \varinjlim_{i \geq i_0} Q_i$. The maps $P_i \to Q_i$ are sfp by \cref{lem:sfp stable under bc} \cref{prop:fs version of sorite for sfp}.

    By assumption, the map 
    \[
        \varinjlim_{i \geq i_0} A_i \otimes_{\bZ[P_i]} \bZ[Q_i] = A \otimes_{\bZ[P]} \bZ[Q] \to B
    \]
    is of finite presentation. Hence, for again large enough $i$, the $A \otimes_{\bZ[P]} \bZ[Q]$-algebra $B$ descends to an $A_i \otimes_{\bZ[P_i]} \bZ[Q_i]$-algebra $B_i$ of finite presentation. This exhibits the desired sfp morphism $(P_i \to A_i) \to (Q_i \to B_i)$ of prelog rings.
\end{proof}

For a prelog ring $(P\to A)$, we define 
\[
    \Spec(P\to A)
\]
to be the scheme $X= \Spec(A)$ endowed with the log structure induced by $P\to A = \Gamma(X, \cO_X)$. For example, we have $\bA_P = \Spec(P \to \bZ[P])$.
If $Y$ is a log scheme, then $(\cM_Y(Y)\to \cO_Y(Y))$ is a prelog ring. These constructions define functors between the category of prelog rings and the opposite of the category of log schemes. 

\begin{lem}
\label{lem:maps to affine charted}
    We have an adjunction 
    \begin{equation}
        \label{eq:Spec prelog adjunction}
        \Hom(Y, \Spec(P\to A)) \simeq \Hom((P\to A), (\cM_Y(Y)\to \cO_Y(Y))).
    \end{equation}
\end{lem}
\begin{proof}
    By \cite[Proposition~III~1.2.9]{Ogus}, we have
    \[ 
        \Hom(Y, \bA_P) = \Hom(P, \cM_Y(Y)).
    \]
    Moreover, for every strict map of log schemes $X'\to X$ we have 
    \[
        X' \simeq X\times_{\underline{X}} \underline{X}'.
    \]
    Applying this to the map $\Spec(P\to A)\to \bA_P$ we obtain
    \begin{align*} 
        \Hom(Y, \Spec(P\to A)) &= \Hom(Y, \bA_P)\times_{\Hom(Y,\Spec(\bZ[P]))} \Hom(Y, \Spec(A)) \\
        &= \Hom(P, \cM_Y(Y))\times_{\Hom(P, \cO_Y(Y))} \Hom(A, \cO_Y(Y)) \\
        &=  \Hom((P\to A), (\cM_Y(Y)\to \cO_Y(Y))).\qedhere
    \end{align*}
\end{proof}
    
In particular, if $(P\to A)=\varinjlim\, (P_i\to A_i)$ where $(P_i\to A_i)$ is a filtered colimit of prelog rings, then we have 
\[
    \Spec(P\to A) = \varprojlim \Spec(P_i\to A_i)
\]
is a cofiltered limit. If we choose $P_i \to A_i$, as we always can, to be compact in prelog rings, then we get the claim of the following lemma.

\begin{lem}  
\label{lem:prelogring_colimit_fs}
    Let $X$ be an affine saturated log scheme with a global chart. Then $X$ is the cofiltered inverse limit $\varprojlim \Spec(P_i\to A_i)$ of affine fs log schemes of finite type over $\bZ$.
\end{lem}

\subsection{Maps of saturated finite presentation}
\label{sec:sfp log maps}

We work in the category of saturated log schemes. 

\begin{defi}
\label{defi:sfp log map}
    Let $f\colon Y \to X$ be a morphism of saturated log schemes.
    \begin{enumerate}[(i)]
        \item 
        \label{defitem:sfp1}
        The map $f$ is called \textbf{locally of finite presentation up to saturation} or \textbf{locally sfp} for short if \'etale locally on~$X$ and~$Y$ it is isomorphic to the spectrum of an sfp map $\vtheta \colon (P \to A) \to (Q \to B)$ of prelog rings.
         \[
            \begin{tikzcd}
                Y    \ar[d,swap,"f"]  
                & \ar[l,swap,"\sim"] \Spec(Q \to B)  \ar[d,"\vtheta"]  \\
                X               
                & \ar[l,swap,"\sim"]  \Spec(P \to A)
            \end{tikzcd}
        \]
       \item 
        The map $f$ is called \textbf{of finite presentation up to saturation} or \textbf{sfp} for short if it is locally sfp and qcqs.
        \item We denote the category of sfp morphisms with target $X$ by $\cat{Sfp}_X$. 
    \end{enumerate}
\end{defi}

\begin{rmks}
\label{rmk:maps between fs are sfp}
\begin{enumerate}
    \item 
    In terms of charts,
    \crefpart{defi:sfp log map}{defitem:sfp1} concretely means that $f$ is locally sfp if
    \'etale locally on~$X$ and~$Y$  it admits a chart (the horizontal maps are strict)
         \[
            \begin{tikzcd}
                Y  \ar[r]  \ar[d,swap,"f"]  & \bA_Q   \ar[d,"\vtheta"]  \\
                X  \ar[r]              & \bA_P
            \end{tikzcd}
        \]
        such that $\vtheta: P \to Q$ is an sfp monoid homomorphism and the natural map
        \[
         Y \longrightarrow X \times_{\bA_P} \bA_Q
        \]        
        is (strict and) of finite presentation. Note that the saturated fibre product $X \times^\sat_{\bA_P} \bA_Q$ exists by \cref{lem:saturated fibre product} and agrees with $X \times_{\bA_P} \bA_Q$.    
    \item 
    For an sfp morphism $(P \to A) \to (Q \to B)$ of prelog rings the induced map 
    \[
        \Spec(Q \to B) \la \Spec(P \to A)
    \]
    is an sfp morphism.
    \item 
    A strict morphism $Y\to X$ is locally sfp if and only if it is locally of finite presentation as a morphism of schemes (this is not immediate; see \cref{lem:strict-sfp-is-fp}). 
    \item 
    If~$X$ is an fs log scheme, the sfp morphisms $Y \to X$ coincide with the morphisms of finite presentation of fs log schemes. This follows at once from 
    \cref{prop:fs version of sorite for sfp}.
    \item 
    \label{rmkitem:sfp etale local}
    The notion of a locally sfp map is \'etale local on source and target. More precisely, if $f\colon Y \to X$ is locally sfp and 
        \[
           \begin{tikzcd}
                Y     \ar[d,"f",swap]   
                & V    \ar[d,"f_V"] \ar[l] \\
                X     
                & U \ar[l]  
           \end{tikzcd}
       \] 
    a commutative diagram (not necessarily cartesian) with horizontal maps strict and \'etale, then $f_V$ is also locally sfp. And if 
    \[ 
           \begin{tikzcd}
                Y     \ar[d,"f",swap]   
                & V_i    \ar[d,"f_i"] \ar[l] \\
                X     
                & U_i \ar[l]  
           \end{tikzcd}
    \]
    are commutative squares as above where $\{U_i\to X\}$ and $\{V_i\to Y\}$ form strict \'etale covers, and the maps $f_i$ are locally sfp, then so is $f$.
    \item 
    \label{rmkitem:sfp under strict bc}
    The base change of a (locally) sfp morphism by a strict morphism is again (locally) sfp.
\end{enumerate}
\end{rmks}

\begin{prop} \label{prop:sfp morphisms}
    Let $f\colon Y \to X$ be a morphism of saturated quasi-coherent log schemes.
    Then the following are equivalent.
    \begin{enumerate}[(i)]
        \item 
        \label{propitem:sfp-log-schemes} 
        The morphism~$f$ is locally sfp.
        \item 
        \label{propitem:sfp-chart-lifting}
        For every strict \'etale map $U= \Spec(P\to A)\to X$, the map $V= Y\times_X U \to U$ locally on  $V$ is of the form $\Spec(Q\to B) \to \Spec(P\to A)$ for an sfp map of prelog rings 
        \[
            (P\to A)\la (Q\to B).
        \]
        
        \item 
        \label{propitem:sfp-local-approx} 
        For every strict \'etale map $U= \Spec(P\to A)\to X$ and every filtered colimit of saturated prelog rings $(P\to A)=\varinjlim (P_i\to A_i)$ the following holds. 

        The map $V= Y\times_X U \to U$ locally on $V$ is, for large enough $i$, the saturated base change of $\Spec(Q_i \to B_i) \to \Spec(P_i \to A_i)$ for an sfp map $(P_i \to A_i) \to (Q_i \to B_i)$ of saturated prelog rings. 
        
        Diagram for visual aid:
        \[  
        \begin{tikzcd}
        Y \arrow[d,"f",swap] & V \arrow[l] \arrow[d] \arrow[r] & \Spec(Q_i \to B_i) \arrow[d] \\
        X & U = \Spec(P\to A) \arrow[l] \arrow[r] & \Spec(P_i \to A_i)
        \end{tikzcd}
        \]
        
        \item 
        \label{propitem:sfp-bc-over-Z} 
        \'Etale locally on~$X$ and~$Y$ the morphism~$f$ fits into a cartesian diagram
        \[
           \begin{tikzcd}
                        Y     \ar[d,"f",swap]  \ar[r] & Y_0    \ar[d,"f_0"]  \\
                        X     \ar[r]             & X_0
           \end{tikzcd}
       \]        
       in the category of saturated log schemes where the right hand vertical map is a morphism of fs log schemes of finite type over $\bZ$.
       \item 
       \label{propitem:sfpbc} 
       \'Etale locally on~$X$ and~$Y$ the morphism~$f$ fits into a cartesian diagram as above where the right hand vertical map is a morphism of finite presentation of fs log schemes.
    \end{enumerate}
\end{prop}

\begin{proof}
We shall prove 
\[  \text{
\labelcref{propitem:sfp-log-schemes}
$\Rightarrow$
\labelcref{propitem:sfpbc}
$\Rightarrow$
\labelcref{propitem:sfp-chart-lifting}
$\Rightarrow$
\labelcref{propitem:sfp-log-schemes}
\quad and \quad 
\labelcref{propitem:sfp-chart-lifting}
$\Rightarrow$
\labelcref{propitem:sfp-local-approx}
$\Rightarrow$
\labelcref{propitem:sfp-bc-over-Z}
$\Rightarrow$
\labelcref{propitem:sfpbc}.}
\]

\labelcref{propitem:sfp-log-schemes} $\Rightarrow$ \labelcref{propitem:sfpbc}: 
We may work \'etale locally on $X$ and $Y$ and assume that $Y \to X$ is given by an sfp morphism $(P \to A) \to (Q \to B)$ of prelog rings. Choose a presentation 
\[
    (P \to A) = \varinjlim_i\, (P_i \to A_i)
\]
as a filtered colimit where the $P_i$ are fs monoids. The proof now follows by applying \cref{prop:colimit descent for sfp prelog maps} because the monoid $Q_i$ is fs for a prelog ring $(Q_i \to B_i)$ that is sfp over a prelog ring $(P_i \to A_i)$ with $P_i$ an fs monoid, see 
\cref{prop:fs version of sorite for sfp}.

\labelcref{propitem:sfpbc} $\Rightarrow$ \labelcref{propitem:sfp-chart-lifting}:
If assertion \labelcref{propitem:sfp-chart-lifting} holds for all $U'_\alpha = \Spec(P' \to A') \to U \to X$ with an \'etale covering $\{U'_\alpha \to U\}$ induced by maps of prelog rings $(P \to A) \to (P'_\alpha \to A'_\alpha)$, then it also holds for $U \to X$. We may therefore work \'etale locally on $X$ and $Y$, assume that $X = U = \Spec(P \to A)$ and that there is a cartesian diagram as in \labelcref{propitem:sfpbc}. Moreover, by further localising, because maps between fs log schemes have the chart lifting property for fs charts, we may assume that $Y_0 \to X_0$ is the map associated to a map $(P_0 \to A_0) \to (Q_0 \to B_0)$ of prelog rings with $P_0$ and $Q_0$ being fs monoids. 

We write $P = \varinjlim_i P_i$ as a filtered colimit of fs monoids $P_i$, and we set $X_i = \Spec(P_i \to A)$. Then $X = \varprojlim_i X_i$ and $\cM(X) = \varinjlim_i \cM_{X_i}(X_i)$ since we evaluate on an affine scheme $\Spec(A)$. Since $P_0$ is a compact object in saturated monoids, the map $X \to X_0$ given on log structures by $P_0 \to \cM(X)$ factors, for $i$ large enough, via a morphism $P_0 \to \cM_{X_i}(X_i)$ and leading to a~factorisation $X_i \to X_0$. Here we use that $X \to X_i$ is an isomorphism on underlying schemes.

We denote by $Y_i \to X_i$ the saturated base change of $Y_0 \to X_0$. Note that these are fs log schemes and the map $Y_i \to X_i$ is of finite presentation.
We therefore have chart lifting for $Y_i \to X_i$ and fs charts. Working locally on $Y_i$ we may thus assume that $Y_i = \Spec(Q_i \to B)$ for a~map $(P_i \to A) \to (Q_i \to B)$ of prelog rings. Then $Y$ is the saturated base change of a~diagram of affine charted log schemes with charted morphisms. Hence $Y$ itself is of the form $\Spec(Q \to B)$ and $Y \to X$ comes from a map of prelog rings $(P \to A) \to (Q \to B)$, where $P \to Q$ is the saturated base change of $P_i \to Q_i$, hence is sfp, because $P_i$ and $Q_i$ are fs. 

\labelcref{propitem:sfp-chart-lifting} $\Rightarrow$ \labelcref{propitem:sfp-log-schemes}: This is obvious from the definition of locally sfp morphisms of saturated log schemes.

\labelcref{propitem:sfp-chart-lifting} $\Rightarrow$ \labelcref{propitem:sfp-local-approx}:
By \labelcref{propitem:sfp-chart-lifting} we may assume that the map $Y \to X$ is induced by an sfp morphism $(P\to A)\to (Q\to B)$ of prelog rings. The claim in this affine charted case then was proven already in \cref{prop:colimit descent for sfp prelog maps}.

\labelcref{propitem:sfp-local-approx} $\Rightarrow$ \labelcref{propitem:sfp-bc-over-Z}: 
This follows since \'etale locally $X$ is of the form $\Spec(P\to A)$, and $(P\to A)$ can be written as the direct limit of prelog rings $(P_i\to A_i)$ with $P_i$ fs and $A_i$ of finite type over $\bZ$. In this case, the maps $\Spec(Q_i \to B_i)\to \Spec(P_i\to A_i)$ are morphisms of fs log schemes of finite type over $\bZ$.

\labelcref{propitem:sfp-bc-over-Z} $\Rightarrow$ \labelcref{propitem:sfpbc}:
This is obvious since every morphism between schemes of finite type over $\bZ$ is of finite presentation.
\end{proof}

In particular, characterization \labelcref{propitem:sfp-chart-lifting} of locally sfp maps implies:

\begin{cor} \label{cor:sfp-chart-lifting}
    Every locally sfp map of saturated log schemes has the chart lifting property (\cref{rmk:chart-lifting}).
\end{cor}

\begin{prop}
\label{prop: sorite for sfp}
    Let $f\colon Y\to X$ be a locally sfp morphism of saturated log schemes.
    \begin{enumerate}[(a)]
        \item 
        \label{propitem:sfp base change}
        Let $h\colon X'\to X$ be a morphism from a saturated log scheme $X'$ ($h$ is not necessarily locally charted). Then the saturated pullback $Y' = X'\times^\sat_X Y$ exists and the morphism $Y'\to X'$ is locally sfp. 
        \item 
        \label{propitem:sfp composition}
        Let $g\colon Z \to Y$ be another locally sfp morphism of saturated log schemes. Then the composition $Z \to Y \to X$ is locally sfp. 
        \item 
        \label{propitem:sfp 2 out of 3}
        Let $g\colon Z\to Y$ be a morphism from a saturated log scheme $Z$ such that the composition $Z\to Y\to X$ is locally sfp. Then $Z\to Y$ is locally sfp. 
    \end{enumerate}
\end{prop}

\begin{proof}
\labelcref{propitem:sfp base change}:
Working \'etale locally on $X$, $Y$ and $X'$, we may assume that there exists a commutative square as in 
\crefpart{prop:sfp morphisms}{propitem:sfpbc}: $Y \to X$ is the saturated base change of a map of finite presentation $Y_0 \to X_0$ of fs log schemes along $X \to X_0$. So $Y' \to X'$ is the saturated base change of $Y_0 \to X_0$ along the composition $X' \to X_0$ which exists by 
\crefpart{lem:saturated fibre product}{lemitem:saturated fibre product:charts}.
Being a~saturated base change of $Y_0 \to X_0$ shows by \cref{prop:sfp morphisms} that $Y' \to X'$ is locally sfp.

\labelcref{propitem:sfp composition}: 
The assertion is local on $X$, $Y$ and $Z$. 
We may therefore assume that $Y \to X$ is isomorphic to the spectrum of an sfp map $(P \to A) \to (Q \to B)$ of saturated prelog rings. By \crefpart{prop:sfp morphisms}{propitem:sfp-chart-lifting}, locally on $Z$ the map $Z \to Y = \Spec(Q \to B)$ is induced by an sfp morphism $(Q \to B) \to (R \to C)$. Thus the composition $Z \to X$ is induced by the composition 
\[
(P \to A) \la (Q \to B)\la (R \to C)
\]
which is again an sfp morphism of prelog rings, see \cref{cor:compose sfp monoid mor}. 
This shows that $Z \to X$ is locally sfp.

\labelcref{propitem:sfp 2 out of 3}:
The map $g$ factors as $Z \to Z \times^\sat_X Y \to Y$ which is the composition of the graph map and the base change of $Z \to X$. The graph map is the base change of the relative diagonal $\Delta_{Y/X} \colon Y \to Y \times^\sat_X Y$. So by \labelcref{propitem:sfp base change} and \labelcref{propitem:sfp composition} the claim follows from the special case of $\Delta_{Y/X}$ as a~map between saturated log schemes that are locally sfp over $Y$.

To study the diagonal we may work \'etale locally on $Y$ and $X$. By \cref{prop:sfp morphisms} we may assume that $Y \to X$ is the saturated base change of a map $Y_0 \to X_0$ of finite presentation between fs log schemes. Here the diagonal $\Delta_0 \colon Y_0 \to Y_0 \times^\sat_{X_0} Y_0$ is of finite presentation between fs log schemes and thus sfp. The diagonal $\Delta_{Y/X}$ is the saturated base change of $\Delta_0$ and thus by \labelcref{propitem:sfp base change} also sfp.
\end{proof}

\subsection{Approximation of sfp maps}
\label{ss:approx}

In this subsection, we provide approximation results for sfp morphisms, in the spirit of \cite[\S 8]{EGAIV3}. This is a delicate task, since sfp maps are in general not of finite presentation as maps of schemes, and because we use saturated base change, which does not always agree with the scheme-theoretic base change. 

We first deal with the case of an affine charted base. Let $(P_i\to A_i)_{i\in I}$ be a directed system of prelog rings with $P_i$ fs and $A_i$ of finite type over $\bZ$ and let $(P\to A)=\varinjlim (P_i\to A_i)$. Set $X = \Spec(P\to A)$ and $X_i = \Spec(P_i\to A_i)$. 

Recall that $\cat{Sfp}_{X}$ denotes the category of sfp morphisms $Y\to X$. Note that, since each $X_i$ is fs and of finite type over $\bZ$, the category $\cat{Sfp}_{X_i}$ consists of fs log schemes $Y_i\to X_i$ whose underlying morphism of schemes is of finite type. For $j\geq i$, saturated base change along $X_{j}\to X_i$ induces a functor $\cat{Sfp}_{X_i}\to \cat{Sfp}_{X_{j}}$, turning $\{\cat{Sfp}_{X_i}\}_{i\in I}$ into a directed system of categories.

\begin{prop} 
\label{prop:fs-approximation}
     Saturated base change along $X\to X_i$ induces an equivalence of categories 
     \[
        \varinjlim \cat{Sfp}_{X_i}\isomto \cat{Sfp}_{X}.
    \]
\end{prop}

We start with some easy category theory. Consider a category $\mathcal{C}$ admitting small colimits and a map of inductive systems $\{P_i\to Q_i\}$ indexed by a directed set $I$. We say that such a map is {\bf cocartesian} if for every $j>i$ the natural square 
\[ 
    \begin{tikzcd}
        Q_j & Q_i \ar[l] \\
        P_j \ar[u] & P_i \ar[l] \ar[u]
    \end{tikzcd}
\]
is cocartesian. Suppose  $\{P_i\to Q_i\}$ is cocartesian. If $P = \varinjlim P_i$ and $Q = \varinjlim Q_i$ then for every $i\in I$ the square
\[
    \begin{tikzcd}
        Q & Q_i \ar[l] \\
        P \ar[u] & P_i \ar[l] \ar[u]
    \end{tikzcd}
\]
is cocartesian as well. This means that for every morphism $P \to R$ and every $i$ we have
\[ 
    \Hom_P(Q, R) = \Hom_{P_i}(Q_i, R).
\]

\begin{lem} \label{lem:abstract-colim-bullshit}
    Let $I$ be a directed set, $\mathcal{C}$ a category with small colimits, and let $\{P_i\to Q_i\}_{i\in I}$ be a cocartesian map of inductive systems such that $Q_0$ is a compact object of $\mathcal{C}_{/P_0}$ for some $0\in I$.  Let $\{P_i\to R_i\}_{i\in I}$ be an arbitrary map of inductive systems. Set $P=\varinjlim P_i$, $Q = \varinjlim Q_i$, and $R = \varinjlim R_i$. Then 
    \[
        \varinjlim \Hom_{P_i}(Q_i, R_i) \isomto \Hom_P(Q, R).
    \]
\end{lem}

\begin{proof}
We compute
\[
\Hom_P(Q, R) = \Hom_{P_0}(Q_0, R) = \Hom_{P_0}(Q_0, \varinjlim R_i)
\]
which, because $Q_0$ is compact and $\{P_i\to Q_i\}_{i\in I}$ is cocartesian,  equals
\[
    \varinjlim \Hom_{P_0}(Q_0, R_i) = \varinjlim \Hom_{P_i}(Q_i, R_i). \qedhere
\]
\end{proof}

\Cref{lem:abstract-colim-bullshit} will be used to obtain fully faithfulness in the proof below.

\begin{proof}[Proof of \cref{prop:fs-approximation}]
We first prove establish fully faithfulness, and then essential surjectivity.

\medskip

\noindent {\bf Proof of fully faithfulness.} Let $0 \in I$ be an index and let $Y_0$ and $Z_0$ be objects of $\cat{Sfp}_{X_0}$. We must show that
\begin{equation} \label{eqn:approx-ff} 
    \varinjlim_{i\geq 0} \Hom_{X_i}(Z_i, Y_i) \isomto \Hom_X(Z, Y),
\end{equation}
where $Z_i$, $Y_i$ (resp.\ $Z$, $Y$) are defined for all $i \geq 0$ by saturated base change of $Y_0$ and $Z_0$ to $X_i$ (resp.\ to $X$).

Let us call $Y_0$ {\bf special} if it is of the form  $Y_0 = \Spec(Q_0\to B_0)$ where $(Q_0\to B_0)$ is a compact object in the category of saturated prelog rings over $(P_0\to A_0)$. Recall from \cref{prop:compact-saturated-prelog-ring} that this means that $P_0\to Q_0$ is sfp and $A_0\otimes_{\bZ[P_0]}\bZ[Q_0]\to B_0$ is of finite presentation, thus every sfp $Y_0\to X_0$ is locally of this form. 

\medskip

\noindent \emph{Step 1: \labelcref{eqn:approx-ff} holds assuming $Y_0$ is special.}

\smallskip

If $Y_0 = \Spec(Q_0\to B_0)$ is special, then compactness of $(Q_0\to B_0)$ combined with \cref{lem:abstract-colim-bullshit} and \cref{lem:maps to affine charted} together imply that (in the cofinal subsystem $i \geq 0$)
\begin{align*} 
    \Hom_X(Z, Y) &= \Hom_{(P\to A)}\big((Q\to B), (\mathcal{M}_Z(Z)\to \mathcal{O}_Z(Z))\big) \\
    &= \Hom_{\varinjlim(P_i\to A_i)}\big(\varinjlim(Q_i\to B_i), \varinjlim(\mathcal{M}_{Z_i}(Z_i)\to \mathcal{O}_{Z_i}(Z_i))\big) \\
    &= \varinjlim \Hom_{(P_i\to A_i)}\big((Q_i\to B_i), (\mathcal{M}_{Z_i}(Z_i)\to \mathcal{O}_{Z_i}(Z_i))\big) \\
    &= \varinjlim \Hom_{X_i}(Z_i, Y_i). 
\end{align*}
Here, we used the fact that, since $Z$ is qcqs, we have
\[
    (\mathcal{M}_Z(Z)\to \mathcal{O}_Z(Z)) = \varinjlim\, (\mathcal{M}_{Z_i}(Z_i)\to \mathcal{O}_{Z_i}(Z_i)),
\]
(see \cref{lem:M-of-lim} below). 

\medskip

\noindent \emph{Step 2: \labelcref{eqn:approx-ff} holds assuming $Y_0$ is a disjoint union of special objects.}

\smallskip

Next, suppose that $Y_0$ is the disjoint union of a (finite) number of special objects $\{Y_0^\alpha\}_{\alpha\in A}$. We treat the index set $A$ as a discrete space. Then
\[ 
    \Hom_{X_i}(Z_i, Y_i) = \coprod_{a\colon |Z_i|\to A} \prod_{\alpha} \Hom_{X_i}(a^{-1}(\alpha), Y^\alpha_i),
\]
and the same with $\Hom_X(Z, Y)$. To pass to the limit observe that since $|Z|=\varprojlim |Z_i|$ (inverse limit of spectral spaces with quasi-compact transition maps), every locally constant function $|Z|\to A$ factors through one of the projections $\pi_j\colon |Z|\to |Z_j|$ (see~\stacks[Lemma]{0A2Y}).  Therefore
\begin{align*}
    \Hom_X(Z, Y) &= \coprod_{a\colon |Z|\to A} \prod_{\alpha} \Hom_{X}(a^{-1}(\alpha), Y^\alpha)  \\
    &= \varinjlim_j \coprod_{a\colon |Z_j|\to A} \prod_{\alpha} \Hom_X(\pi_j^{-1} a^{-1}(\alpha), Y^\alpha) \\
    &= \varinjlim_j \coprod_{a\colon |Z_j|\to A} \prod_{\alpha} \varinjlim_{i\geq j} \Hom_{X_i}(\pi_{ij}^{-1} a^{-1}(\alpha), Y_i^\alpha) \\
    &= \varinjlim_i \coprod_{a\colon |Z_i|\to A} \prod_{\alpha} \Hom_{X_i}(a^{-1}(\alpha), Y^\alpha_i) \\
    &= \varinjlim_i \Hom_{X_i}(Z_i, Y_i).
\end{align*}
Here, for $i\geq j$ we denoted by $\pi_{ij}$ the map $X_i\to X_j$. In the fourth equality, we commute the filtered colimit $\varinjlim_{i\geq j}$ past the finite limit $\prod_\alpha$, then past the colimit $\coprod_a$, and finally $\varinjlim_j \varinjlim_{i\geq j}$ is the same as $\varinjlim_i$.

\medskip

\noindent \emph{Step 3: \labelcref{eqn:approx-ff} holds in general.}

\smallskip

Now, consider the general case. We pick a strict \'etale surjection $U_0\to Y_0$ with $U_0$ a finite disjoint union of special objects, and a strict \'etale surjection $U'_0\to U_0\times_{Y_0} U_0$ with $U'_0$ a finite disjoint union of special objects. Define $U_i = U_0\times_{Y_0} Y_i$ and $U = U_0\times_{Y_0} Y$, and analogously with $U'_i$ and $U'$. Then $Y_i$ is the coequalizer of $U'_i\rightrightarrows U_i$ (in sheaves on the big etale site of $\underline{X}$) where the two maps are strict \'etale. Let $f\colon Z\to Y$ be a morphism over $X$. We will show that there exists a morphism $f_i\colon Z_i\to Y_i$ inducing $f$, and that any two such $f_i$ become equal after base change to $X_j$ for $j\gg i$. Let $V = Z\times_Y U$ and $V' = Z\times_Y U'$. Then $Z$ is the coequalizer of the two strict \'etale maps $V'\rightrightarrows V$. Since $Z = \varprojlim Z_i$ as schemes, we may assume that there exists a coequalizer diagram $V'_0\rightrightarrows V_0\to Z_0$ of strict \'etale maps whose base change is $V'\rightrightarrows V\to Z$. Since $U_0$ is the disjoint union of special objects, the result of the previous paragraph applied to $V\to U$ produces an (essentially unique) map $g_i\colon V_i\to U_i$ inducing $V\to U$. Similarly, we obtain a map $g'_i\colon V'_i\to U'_i$. Diagrams for visual aid:
\[ 
    \begin{tikzcd}
        V'_i \ar[r,shift right=0.5ex,"t",swap] \ar[r,shift left=0.5ex,"s"] \ar[d,"g'_i",swap] & V_i\ar[d,"g_i"] \ar[r] & Z_i \ar[d,dotted,"f_i?"] \\
        U'_i \ar[r,shift right=0.5ex,"t",swap] \ar[r,shift left=0.5ex,"s"] & U_i  \ar[r] & Y_i
    \end{tikzcd}
    \qquad
    \begin{tikzcd}
        V' \ar[r,shift right=0.5ex,"t",swap] \ar[r,shift left=0.5ex,"s"] \ar[d,"g'",swap] & V\ar[d,"g"] \ar[r] & Z\ar[d,"f"] \\
        U' \ar[r,shift right=0.5ex,"t",swap] \ar[r,shift left=0.5ex,"s"] & U  \ar[r] & Y.
    \end{tikzcd}
\]
Looking at the four compositions in the left square
\[ 
    g_i s, \quad s g'_i, \quad g_i t, \quad t g'_i \quad \colon  \quad V'_i\longrightarrow U_i,
\]
since $gs = sg'$ and $gt = tg'$, we must have $g_i s = sg'_i$ and $g_it = tg'_i$ for $i\gg 0$. We have thus obtained a natural transformation between the two $(\bullet\rightrightarrows\bullet)$-shaped diagrams, yielding a map $f_i\colon Z_i\to Y_i$ between their coequalizers. It is clear from the construction that $f$ is the base change of $f_i$ to $X$. Moreover, retracing the argument, we see that any two solutions to this problem become equal for $i\gg 0$. This finishes the proof of fully faithfulness.

\medskip

\noindent {\bf Proof of essential surjectivity.} Let $Y\to X$ be an sfp map. If $Y$ is special, i.e.\  of the form $\Spec(Q\to B)$ for an sfp map $(P\to A)\to (Q\to B)$, then we have already established in 
\cref{prop:colimit descent for sfp prelog maps} that $Y$ is in the essential image. Therefore finite disjoint unions of special objects are in the image as well. Moreover, if $U$ is in the essential image and $V\to U$ is a strict morphism of finite presentation, then $V$ is in the essential image as well, simply by the corresponding fact for schemes \cite[Th\'eor\`eme 8.5.2]{EGAIV3} (and the fact that strict finitely presented maps are sfp). 

With these preliminary observations, let $U\to Y$ be a strict \'etale surjection with $U$ a finite disjoint union of special objects, and let $R = U\times_Y U$. Then $U$ is in the essential image and as the first projection $R\to U$ is strict \'etale, so is $R$. Consider the corresponding \'etale groupoid $(U, R, s, t, c)$ or rather $(U, R, s, t, c, e, i)$ with \'etale maps $s,t\colon R\to U$ (source and target), $c\colon R\times_{t,U,s} R\to R$ (composition), $e\colon U\to R$ (identity), and $i\colon R\to R$ (inverse). All of these are \'etale morphisms over $U$, satisfying certain identities (see \stacks[Section]{0230}) taking place in maps between some fibre products of $R$ over $U$. Let $U_i$ and $R_i$ be sfp over $X_i$ with saturated base change $U$ and $R$. By fully faithfulness, possibly after increasing $i$ there exist morphisms $s_i, t_i\colon R_i\to U_i$ etc. (Note that since $R_i\to U_i$ is strict \'etale, $R_i\times_{U_i} R_i$ is sfp over $X_i$, with saturated base change to $X$ isomorphic to $R\times_U R$.) Moreover, the axioms of a groupoid will be satisfied for $i\gg 0$. 

Let $Y_i$ be the algebraic stack over $X_i$ defined by the \'etale groupoid $(U_i, R_i, \ldots)$. These form an inverse system with affine transition maps and inverse limit $Y$. Moreover, $Y$ and all $Y_i$ are qcqs. Since $Y$ is a scheme, by \cite[Theorem~C(iii)]{RydhApproximation} the $Y_i$ are schemes for $i\gg 0$.

Coming back from our brief detour through groupoids and stacks, we establish that for $i\gg 0$ the coequalizer $Y_i$ of $R_i\rightrightarrows U_i$ exists as a scheme. Moreover, $U_i\to Y_i$ is an \'etale surjection and $R_i = U_i\times_{Y_i} U_i$. Since log structures form an \'etale stack, this enables us to construct the log structure on $Y_i$ for which $U_i\to Y_i$ is strict. Then, since the formation of \'etale coequalizers commutes with base change \stacks[Lemma]{03I4}, we deduce that $Y$ is the saturated base change of $Y_i$. 
\end{proof}

\begin{lem} \label{lem:M-of-lim}
    Let $(X_i)$ be an inverse system of log schemes with affine transition maps and limit $X$. Suppose that $X_i$ are qcqs. Then
    \[ 
        \varinjlim \cM(X_i) \isomto \cM(X).
    \]
\end{lem}

\begin{proof}
We have $\underline{X} = \varprojlim \underline{X}_i$, and the same with underlying topological spaces. In particular, $\underline{X}$ is qcqs. We claim that the map of \'etale sheaves on $\underline{X}$
\[ 
    \varinjlim \pi_i^{-1}(\cM_{X_i}) \la \cM_X
\]
is an isomorphism. To this end, first note that since  $\cO_X^\times =\varprojlim \pi_i^{-1}(\cO_{X_i}^\times)$, the source (call it $\cM$) of the map in question is a log structure. Moreover, the log scheme $(\underline{X}, \cM)$ has the universal property of the inverse limit $\varprojlim X_i$, and hence $\cM\simeq\cM_X$.

Since $X$ is qcqs, in view of \stacks[Lemma]{09YN} (see also \cite[Exp.\ VI, \S 5]{SGA4_2}) for any system of sheaves of sets $\cF_i$ on $X_i$ together with maps $\pi_{ij}^{-1}(\cF_i)\to \cF_j$ for $j\geq i$, satisfying the natural transitivity property, we have
\[ 
    \varinjlim \Gamma(X_i, \cF_i) = \Gamma(X, \varinjlim \pi^{-1}_i \cF_i).
\]
The result follows by taking $\cF_i = \cM_{X_i}$.  
\end{proof}

We can strengthen \cref{prop:fs-approximation} as follows. 

\begin{defi} \label{def:affine-charted-system}
We call an inverse system of saturated log schemes $\{X_i\}_{i\in I}$ {\bf affine charted} if there exists a set $A$, direct systems of saturated prelog rings 
\[
    \{(P_{\alpha,i}\to A_{\alpha,i})\}_{i\in I^{\rm op}}
    \quad \text{for $\alpha\in A$}
\]
and strict \'etale maps forming a morphism of inverse systems
\[ 
    f_{\alpha,i}\colon U_{\alpha,i} = \Spec(P_{\alpha,i}\to A_{\alpha,i}) \la X_i
\]
and such that for every index $i$, the family $\{f_{\alpha,i}\colon U_{\alpha,i} \to X_i\}_{\alpha\in A}$ is a covering family for the \'etale topology on $\underline{X}_i$.
\end{defi}

Note that in particular the transition maps $X_j\to X_i$ in an affine charted inverse system are affine. Since affine morphisms are in particular qcqs, if $X_i$ is qcqs for some index $i$, then $X_j$ ($j\geq i$) and $X$ are qcqs as well. The inverse limit $\varprojlim X_i$ exists and is a saturated log scheme, and 
\[ 
    U_\alpha = \Spec(P_\alpha\to A_\alpha), \quad (P_\alpha\to A_\alpha) = \varinjlim_{i\in I^{\rm op}}\, (P_{\alpha,i}\to A_{\alpha,i})
\]
forms a strict \'etale cover of $X$. Moreover, we have
\[ 
    \cM_X = \varinjlim_{i\in I^{\rm op}} \pi^{-1}_i(\cM_{X_i})
\]
where $\pi_i\colon X\to X_i$ is the projection. 

Our main approximation result is the following. 

\begin{thm} \label{thm:sfp-approx}
    Let $\{X_i\}_{i\in I}$ be an affine charted system of qcqs saturated log schemes with inverse limit $X$. Then, saturated base change induces an equivalence of categories
    \[ 
        \varinjlim \cat{Sfp}_{X_i} \isomto \cat{Sfp}_{X}
    \]
\end{thm}

\begin{proof}
Since the underlying scheme of $X$ is the inverse limit of the underlying schemes of $X_i$, if $U\to X$ is an \'etale map with $U$ qcqs, then it descends to $U_i\to X_i$. Then the association $U\mapsto \varinjlim \cat{Sfp}_{/U_i}$ defines a stack on the category of qcqs (strict) \'etale maps $U\to X$. Similarly, $U\mapsto \cat{Sfp}_{/U}$ is a stack, and our functor is the global sections of a morphism of stacks. This implies that we can work locally, and hence we may assume that $X = \Spec(P\to A)$ is the inverse limit of $X_i = \Spec(P_i\to A_i)$. 

Now, \cref{prop:fs-approximation} cannot be applied yet since $P_i$ might not be fs and $A_i$ may not be of finite type over $\bZ$. However, let $J$ be the poset of pairs $j = (i(j), P'_j\to A'_j)$ where $i=i(j)\in I$, where $P'_j\subseteq P_{i(j)}$ is a finitely generated (and hence fs) submonoid, and where $A'_j\subseteq A_{i(j)}$ is a~finitely generated $\bZ$-algebra. We have $j'=(i(j'), P'_{j'}\to A'_{j'})\geq (i(j), P'_j\to A'_j) = j$ if $i(j')\geq i(j)$, $P_{i(j)}\to P_{i(j')}$ sends $P'_{j}$ inside $P'_{j'}$, and $A_{i(j)}\to A_{i(j')}$ sends $A'_j$ inside $A'_{j'}$. Then the projection $\pi\colon J\to I$, $\pi(j) = i(j)$ is monotone and cofinal, thus $X = \varprojlim_J X'_j$ where $X'_j = \Spec(P'_j\to A'_j)$. Moreover, for every $i\in I$ we have $X_i=\varprojlim_{J_i} X'_j$ where $J_i = \pi^{-1}(i) \subseteq J$. We then have, by \cref{prop:fs-approximation} applied both to $X=\varinjlim_J X_i$ and $X_i = \varinjlim_{J_i} X_j$
\[ 
    \cat{Sfp}_{X} = \varinjlim_J \cat{Sfp}_{X'_j} = \varinjlim_I \left(\varinjlim_{J_i} \cat{Sfp}_{X'_j} \right)= \varinjlim_I \cat{Sfp}_{X_i}. \qedhere
\]
\end{proof}

The following innocent-looking result is surprisingly not so easy to prove.

\begin{lem} \label{lem:strict-sfp-is-fp}
    Let $f\colon Y\to X$ be a strict map of saturated log schemes. Then $f$ is (locally) sfp if and only if the underlying morphism of schemes $\underline{f}\colon \underline{Y}\to\underline{X}$ is (locally) of finite presentation.
\end{lem}

\begin{proof}
The ``if'' direction is clear. For the converse, suppose that $f$ is sfp. Working strict \'etale locally, see \stacks[Lemma]{05B0}, we may assume by \cref{prop:sfp morphisms} that we are in the following situation: there is a directed system of prelog rings $(P_i\to A_i)$ with $P_i$ fs and $A_i$ of finite type over $\bZ$ and direct limit $(P\to A)$, and sfp maps $(P_i\to A_i)\to (Q_i\to B_i)$ such that 
\[ 
    (Q_j\to B_j) = (Q_i\to B_i) \otimes^\sat_{(P_i\to A_i)} (P_j\to B_j)\qquad \text{for $j\geq i$}
\]
(where $\otimes^\sat$ denotes pushout in the category of saturated prelog rings), with direct limit $(Q\to B)$ such that setting $X_i=\Spec(P_i\to A_i)$ and $Y_i=\Spec(Q_i\to B_i)$ we have
\[ 
    (Y\to X) \quad\simeq\quad (\Spec(Q\to B) 
    \to 
    \Spec(P\to A)) \quad\simeq\quad \varprojlim (Y_i\to X_i).
\]

Since $Y \to X$ is strict, the map of prelog rings $(P \to B) \to (Q \to B)$ induces an isomorphism 
\[
    Y = \Spec(Q \to B) \isomto \Spec(P \to B) = \varprojlim \Spec(P \to A \otimes_{A_i} B_i) = \varprojlim Y'_i.
\]
of saturated  log schemes over $X$, where $Y'_i = \Spec(P\to A\otimes_{A_i} B_i)$.  
Now, since $(Q \to B)$ is a~compact object in the category of saturated prelog rings over $(P\to A)$, we have
\begin{align*} 
    \Hom_X(\varprojlim Y'_i, Y) 
    &= \Hom_{(P\to A)}((Q\to B), \varinjlim (\cM(Y'_i)\to \cO(Y'_i))) \\
    &= \varinjlim \Hom_{(P\to A)}((Q\to B), (\cM(Y'_i)\to \cO(Y'_i))) \\
    &= \varinjlim \Hom_X(Y'_i, Y).
\end{align*}
Therefore the inverse isomorphism $\varprojlim Y'_i\isomto Y$ factors through some $Y'_i$. This means that $Y$ is a retract of $Y'_i$ for some $i$. Thus $\underline{Y}=\Spec(B)$ is a retract of $\underline{Y}'_i=\Spec(A\otimes_{A_i}B_i)$, and since $A\to A\otimes_{A_i}B_i$ is of finite presentation, so is $A\to B$.
\end{proof}

\begin{cor} \label{cor:descent-of-strictness}
    Let $\{X_i\}_{i\in I}$ be an affine charted system of qcqs saturated log schemes with inverse limit $X$. Let $Z_0\to Y_0$ be a morphism between sfp log schemes over $X_0$. Then, the saturated base change $Z\to Y$ to $X$ is strict if and only if the saturated base change $Z_i\to Y_i$ to $X_i$ is strict for $i\gg 0$. 
\end{cor}

\begin{proof}
If $Z\to Y$ is strict, then by \cref{lem:strict-sfp-is-fp} the underlying map of schemes $\underline{Z}\to \underline{Y}$ is of finite presentation. Since $\underline{Y} = \varinjlim \underline{Y}{}_i$ (with affine transition maps $Y_j\to Y_i$), by \cite[Th\'eor\`eme~8.5.2]{EGAIV3} there exists a morphism of finite presentation $\underline{Z}{}'_i\to \underline{Y}{}_i$ for some index $i$ whose base change to $\underline{Y}$ is $\underline{Z}\to\underline{Y}$. Let $Z'_i$ be $\underline{Z}{}'_i$ with log structure making $Z{}'_i\to \underline{Y}{}_i$ strict. Then $Z'_i\to Y_i$ is sfp and its (saturated or not) base change to $X_i$ is $Z\to Y$. Therefore (by the fully faithfulness part of \cref{thm:sfp-approx}), after increasing the index $i$, there exists an isomorphism $Z'_i\simeq Z_i$ over $Y_i$. Since $Z'_i\to Y_i$ is strict, so is $Z_i\to Y_i$. The other direction is clear.
\end{proof}

\subsubsection*{No absolute fs approximation}

A nice complement to the noetherian approximation results of \cite[\S 8]{EGAIV3} is the absolute approximation theorem of Thomason--Trobaugh \cite[Theorem~C.9]{ThomasonTrobaugh}: every qcqs scheme can be written as the inverse limit of schemes of finite type over $\bZ$ with affine transition maps. Our relative approximation result for sfp maps (\cref{thm:sfp-approx}) as well as the fact that every affine charted saturated log scheme $\Spec(P\to A)$ can be written as the inverse limit $\Spec(P_i\to A_i)$ where $P_i$ are fs and $A_i$ are of finite type over $\bZ$ may suggest that a similar result to the Thomason--Trobaugh approximation theorem could hold for log schemes. For example, one could hope that every saturated qcqs log scheme can be written as the limit of an affine charted inverse system of fs log schemes of finite type over $\bZ$. Unfortunately, this is not true, as the following example shows.

\begin{ex} \label{ex:no-absolute-approx}
    We start with any noetherian scheme $X$ covered by two connected opens $U_+, U_-$ whose intersection $U = U_+\cap U_-$ has two connected components $U_0, U_1$. For the simplest example, let $k$ be a connected noetherian ring in which $2$ is invertible and let 
    \[
        X = \Spec(A), \qquad A=k[X,Y]/((Y-X^2+1)(Y+X^2-1))
    \]
    be the union of two parabolas intersecting in two points $x_\pm = (\pm 1, 0)$. Let 
    \[
        U_\pm = X\setminus \{x_\pm\} = \Spec(A[1/(X \mp 1)]
    \]
    and let $U = U_+\cap U_-$, which is isomorphic to the disjoint union of two copies $U_0$, $U_1$ of $\bA^1_k\setminus \{\pm 1\}$. 
    
    Consider the saturated log structure on $X$ constructed in the following way. Let 
    \[
        P = \bigoplus_{n\in \bZ}\bN T^n
    \]
    and let $T\colon P\to P$ be the automorphism induced by $T^n\mapsto T^{n+1}$. We give $U_\pm$ the log structure charted by the map $P\to \cO(U_\pm)$ sending $P\setminus\{0\}$ to zero. We glue these log structures on $U_\pm$ to a log structure on $X$ using the identity $P\to P$ on $U_0$ and the automorphism $T$ on $U_1$. 

    Let us describe the monodromy of the locally constant sheaf $\cF = \ov{\cM}^\gp_X$. The nerve of the cover $X = U_+\cup U_-$ by two connected opens with intersection $U=U_0\sqcup U_1$ with two connected components is homotopy equivalent to $S^1$ and hence induces a surjection 
    \[
        \pi_1^{\rm SGA3}(X,\ov{x})\surj \pi_1(S^1)
    \]
    where $\pi_1^{\rm SGA3}(X,\ov{x})$ is \emph{groupe fondamental \'elargi} introduced in \cite[Exp.\ X, \S 6]{SGA3:2} (see also \cite[Lemma~7.4.3]{BhattScholze}). It has the property that the category of locally constant \'etale sheaves of sets on $X$ is equivalent to the category of sets endowed with a continuous action of $\pi_1^{\rm SGA3}(X,\ov{x})$.  Since $\cF$ is constant on $U_0$ and $U_1$, the monodromy action of $\pi_1^{\rm SGA3}(X,\ov{x})$ on $\cF_{\bar{x}}$ factors through $\pi_1(S^1) \simeq \bZ$ with a generator $T \in \pi_1(S^1)$ acting by the automorphism $T$ of $P^\gp = \bigoplus_{n\in\bZ} \bZ T^n = \bZ[T, T^{-1}]$ according to the glueing instruction for $\cF$. This endows the stalk $\cF_{\bar{x}}$ with the structure of a free module of rank one under the group algebra $\bZ[\pi_1(S^1)] = \bZ[T,T^{-1}]$.
\end{ex}

\begin{prop}
    There does not exist an (not necessarily charted!) inverse system $\{X_i\}_{i\in I}$ with affine transition maps of fs log schemes with inverse limit $X$.
\end{prop}

\begin{proof}
Suppose such a system exists, and let $\cM_i$ be the pull-back log structure along the projection $X\to X_i$, which is again fs. Then $\cM_X = \varinjlim \cM_i$ and so $X = \varinjlim (X, \cM_i)$. We may therefore assume that the underlying scheme of $X_i$ is $X$. Let $\cF= \overline{\cM}^\gp_X$, which is a (Zariski) locally constant sheaf with fibre $\bZ[T,T^{-1}]$, and let $\cF_i = \overline{\cM}^\gp_i$. Then $\cF = \varinjlim \cF_i$. We will make use of the following facts:

\begin{lem} \label{lem:constr-sheaves}
    Let $X$ be a noetherian scheme. We work with \'etale sheaves of $\bZ$-modules on~$X$. 
    \begin{enumerate}[(a)]
        \item \label{lemitem:constr-surj} If $\cF\to \cF'$ is a surjection and $\cF$ is constructible, then so is $\cF'$.
        \item \label{lemitem:constr-lc} A constructible sheaf $\cF$ on $X$ is locally constant if and only if for every specialization between geometric points $\ov{x}\leadsto \ov{y}$ of $X$, the cospecialization map $\cF_{\ov{y}}\to \cF_{\ov{x}}$ is an isomorphism. Every locally constant sheaf (not necessarily constructible) has this property.
        \item \label{lemitem:constr-mbargp} If $\cM$ is an fs log structure on $X$, then the sheaf $\cF = \overline{\cM}^\gp$ is constructible.
        \item \label{lemitem:constr-mbar-surj-cosp} Moreover, for $\cF$ as in \labelcref{lemitem:constr-mbargp}, for every specialization $\ov{x}\leadsto \ov{y}$ between geometric points of $X$, the cospecialization map $\cF_{\bar{y}}\to \cF_{\bar{x}}$ is surjective.
    \end{enumerate}
\end{lem}

\begin{proof}
Assertion \labelcref{lemitem:constr-surj} is \stacks[Proposition]{09BH}. \labelcref{lemitem:constr-lc} is \cite[Exp.\ IX, Proposition~2.11]{SGA4_3},
and 
\labelcref{lemitem:constr-mbargp} and \labelcref{lemitem:constr-mbar-surj-cosp} follow from \cite[Theorem II 2.5.4]{Ogus}. 
\end{proof}

Let $\cF'_i$ be the image of $\cF_i\to \cF$. Then $\varinjlim \cF'_i\to \cF$ is an isomorphism (being both injective and surjective). We claim that each $\cF'_i$ is a locally constant constructible sheaf of $\bZ$-modules. Indeed, it is constructible by \crefpart{lem:constr-sheaves}{lemitem:constr-mbargp} and \labelcref{lemitem:constr-surj}. Moreover, for every specialization $\ov{x}\leadsto \ov{y}$ we have a commutative diagram of cospecialization maps
\[ 
    \begin{tikzcd}
        (\cF_i)_{\bar{y}} \ar[r,two heads] \ar[d,two heads] & (\cF'_i)_{\bar{y}} \ar[r,hookrightarrow] \ar[d] & \cF_{\bar{y}} \ar[d,"\rotatebox{90}{$\sim$}"] \\
        (\cF_i)_{\bar{x}} \ar[r,two heads] & (\cF'_i)_{\bar{x}} \ar[r,hookrightarrow] & \cF_{\bar{x}} 
    \end{tikzcd}
\]
Here, the left arrow is surjective by \crefpart{lem:constr-sheaves}{lemitem:constr-mbar-surj-cosp} and the right map is an isomorphism by \crefpart{lem:constr-sheaves}{lemitem:constr-lc}. It follows that the middle vertical arrow is an isomorphism, so $\cF'_i$ is locally constant by \crefpart{lem:constr-sheaves}{lemitem:constr-lc}. 

We claim that there does not exist an inverse system of \'etale local systems of finite free \mbox{$\bZ$-modules} $\{\cF'_i\}$ on $X$ with $\cF = \varinjlim \cF'_i$, obtaining a contradiction. As in the construction above, we may assume that $\cF'_i$ is a subsheaf of $\cF$. Therefore the monodromy representation of $\pi_1^{\rm SGA3}(X,\ov{x})$ on the fibre $\cF'_{i,\bar{x}}$ factors through $\pi_1(S^1)$ and thus yields a $\bZ[T,T^{-1}]$-submodule of $\cF_{\bar x} \simeq \bZ[T,T^{-1}]$. As $\cF'_{i,\bar{x}}$ is of finite rank as a $\bZ$-module, this can only be $0$, a contradiction. 
\end{proof}

On the positive side, let us call a saturated log scheme \emph{approximable} if it can be written as the limit of an affine charted inverse system of fs log schemes of finite type over $\bZ$. Our \cref{thm:sfp-approx} implies that if $X$ is approximable and $Y\to X$ is sfp, then $Y$ is approximable as well.

\subsection{Smooth, \'etale, and Kummer \'etale maps of log schemes}
\label{sec:log maps:sm et ket}

We shall now extend the familiar notions of smooth, \'etale, and Kummer \'etale maps to arbitrary saturated log schemes. 

\begin{defi}[{\cite[1.5, 1.6]{IllusieFKN} in the fs case}]
\label{defi:smooth etale ket}
    A morphism of saturated log schemes $Y\to X$ is {\bf smooth} (resp.\ {\bf \'etale}, resp.\ {\bf Kummer \'etale}) if \'etale locally on source and target it admits a~chart by a smooth (resp.\ \'etale, resp.\ Kummer \'etale) morphism of monoids $P\to Q$ (with $\Sigma$ the set of primes non-invertible on $X$) such that the induced strict map 
    \[
        Y\la X\times_{\bA_P}\bA_Q
    \]
    is smooth (resp.\ \'etale, resp. \'etale). In particular, such a morphism is locally sfp.
\end{defi}

\begin{rmk}
    The above definitions are very similar to Kato's criteria for smooth and \'etale morphisms \cite[Theorem~3.5]{Kato1989:LogarithmicStructures}. \cref{prop:fs approximation for smooth etale and ket morphisms} below justifies that these notions are natural. We have refrained from discussing log differentials in this paper. However, we expect that smooth and \'etale morphisms can be characterized (among locally sfp morphisms) by a suitable infinitesimal lifting property. 
\end{rmk}

\begin{prop}
\label{prop:fs approximation for smooth etale and ket morphisms}
    Let $f\colon Y\to X$ be a morphism between saturated log schemes. The following conditions are equivalent.
    \begin{enumerate}[(a)]
        \item \label{propitem:sm-et-ket} 
            The morphism $f$ is smooth (resp.\ \'etale, resp.\ Kummer \'etale).
        \item \label{propitem:sm-et-ket-fs-bc}
            Locally on $X$ and $Y$ there exists a smooth (resp.\ \'etale, resp.\ Kummer \'etale) morphism of fs log schemes $Y_0\to X_0$ and a map $X\to X_0$ such that $Y$ is isomorphic to the saturated pullback of $Y_0\to X_0$ along $X\to X_0$.
    \end{enumerate}
\end{prop}

\begin{proof} 
    This is an analogue of \cref{prop:sfp morphisms} for smooth (resp.\ \'etale, resp.\ Kummer \'etale) morphisms. The proof follows along the same lines. Suppose that $Y\to X$ is smooth (resp.\ \'etale, resp.\ Kummer \'etale).
    \'Etale locally  $Y \to X$ has a chart with monoid morphism $P \to Q$ which by \cref{prop:properties of sfp that descend to fs} is the saturated pushout of a smooth (resp.\ \'etale, resp.\ Kummer \'etale) morphism $P_0 \to Q_0$ between fs monoids. 
    We put $P_0$ as an initial monoid of a filtered colimit $P = \varinjlim_i P_i$ with fs monoids $P_i$  and with corresponding saturated pushout $Q_i = (P_i \oplus_{P_0} Q_0)^\sat$. This allows us to write $X\times_{\bA_P} \bA_Q$ as a filtered colimit of $X_i = X_0 \times_{\bA_{P_i}} \bA_{Q_i}$ so that $Y$ descends, for $i$ large enough,  to $Y_i \to X_i \times_{\bA_{P_i}} \bA_{Q_i}$ with log structure induced by $Q_i$. 
    \Cref{lem:bc of smooth etale Ket monoid maps} shows that $P_i \to Q_i$ is again a smooth (resp.\ \'etale, resp.\ Kummer \'etale) morphism between fs monoids. 
    This shows that $Y \to X$ is the saturated base change of a map between fs log schemes with the same property. 

    The converse direction is easier than in the proof of \cref{prop:sfp morphisms} because we already know that the base change of sfp morphisms is sfp. The claim follows from that and \cref{lem:bc of smooth etale Ket monoid maps} for stability of the monoid notions under saturated pushout. 
\end{proof}

\begin{lem}
\label{lem:Delta of ket is strict open}
    Let $f\colon Y \to X$ be an \'etale morphism between saturated log schemes. Then the relative diagonal 
    \[
        \Delta_{Y/X} \colon Y \la Y \times^\sat_X Y
    \]
    is a strict open immersion.
\end{lem}

\begin{proof}
The relative diagonal of a composition $Z \to Y \to X$ is the composition of the relative diagonal and a base change of the relative diagonal: 
\[
\begin{tikzcd}
Z \ar[r,"\Delta_{Z/Y}"] \ar[dr,"\Delta_{Z/X}",swap] & Z \times^\sat_Y Z  \ar[d] \ar[r] & Y \ar[d,"\Delta_{Y/X}"] \\
& Z \times^\sat_X Z \ar[r] & Y \times^\sat_X Y.
\end{tikzcd}
\]
Since strict open immersions are stable under base change and composition and, moreover, the assertion is strict \'etale local on $X$ and $Y$, we may assume that $Y \to X$ is either (1) strict and \'etale or (2) of the form $\Spec R[Q] \to \Spec R[P]$ for an \'etale map of saturated monoids $P \to Q$ (with respect to the set of primes invertible in $R$). In case (1) we note that the diagonal of a~strict \'etale map is strict and an open immersion. In case (2) the diagonal is obtained from the map of monoids 
\[
    (Q \oplus_P Q)^\sat \isomto Q^\gp/P^\gp \oplus Q  \xrightarrow{\pr_2} Q
\]
combining the isomorphism of \cref{lem:Vidal lemma on fs diagonal} with the second projection. The claim follows as the group scheme $\Spec R[Q^\gp/P^\gp]$ is finite \'etale over $\Spec R$. 
\end{proof}
    
\begin{prop}
\label{prop:bc and composition and 2 out of 3}
    Let $Y\to X$ be a smooth (resp.\ \'etale, resp.\ Kummer \'etale)  morphism between saturated log schemes.
    \begin{enumerate}[(a)]
        \item 
        \label{propitem:sm ket base change}
        Let $h\colon X'\to X$ be a morphism from a saturated log scheme $X'$ ($h$ is not necessarily locally charted). Then the saturated pullback $Y' = X'\times^\sat_X Y$ exists and the morphism $Y'\to X'$ is smooth (resp.\ \'etale, resp.\ Kummer \'etale). 
        \item
         \label{propitem:sm ket composition}
        Let $g\colon Z\to Y$ be a morphism from a saturated log scheme $Z$. If $Z\to Y$ is smooth (resp.\ \'etale, resp.\ Kummer \'etale) then the composition $Z\to Y\to X$ is smooth (resp.\ \'etale, resp.\ Kummer \'etale).  
        \item 
        \label{propitem:ket 2 out of 3}
        Let $g\colon Z\to Y$ be a morphism from a saturated log scheme $Z$ such that the composition $Z\to Y\to X$ is 
        \'etale (resp.\ Kummer \'etale).  Then $g\colon Z\to Y$ is \'etale (resp.\ Kummer \'etale).         
    \end{enumerate}
\end{prop}
\begin{proof}
    The proof of \labelcref{propitem:sm ket base change} and \labelcref{propitem:sm ket composition} are parallel to the proof of \cref{prop: sorite for sfp} with \cref{prop:sfp morphisms} replaced by \cref{prop:fs approximation for smooth etale and ket morphisms}.

    As in \cref{prop: sorite for sfp}, assertion \labelcref{propitem:ket 2 out of 3} follows from the fact that the relative diagonal $\Delta_{Y/X}$ is an open immersion, which was proven in \cref{lem:Delta of ket is strict open}.
\end{proof}

Finally, we extend the approximation results of \cref{ss:approx} to smooth, \'etale, and Kummer \'etale morphisms.

\begin{prop} \label{prop:descend-smoothness-etc}
    Let $\{X_i\}_{i\in I}$ be an affine charted inverse system with limit $X$ and let $Y_0\to X_0$ be an sfp morphism for $0\in I$ the smallest element, and define $Y_i\to X_i$ and $Y\to X$ by saturated base change. Suppose that $X_0$ is quasi-compact. Then, the map $Y\to X$ is smooth/\'etale/Kummer \'etale if and only if $Y_i\to X_i$ is smooth/\'etale/Kummer \'etale for $i\gg 0$. 
\end{prop}

\begin{proof}
The ``if'' part holds since smooth, \'etale, and Kummer \'etale maps are stable under base change (\cref{prop:bc and composition and 2 out of 3}). To show the ``only if'' part, suppose that $Y\to X$ is smooth, \'etale, or Kummer \'etale. Since $X_0$ is quasi-compact and the properties are strict \'etale local, we may assume that $X_i = \Spec(P_i\to A_i)$ and $X = \Spec(P\to A)$. Since the assertion is also local on $Y$, we may assume that there exists a smooth/Kummer \'etale morphism of monoids $P\to Q$ (with respect to primes invertible on $X$) such that $Y\to X$ factors through a strict \'etale morphism $Y\to Y' = \Spec(Q\to A\otimes_{\bZ[P]}\bZ[Q])$. By \cref{cor:monoid-sm-et-Ket-approx}, there exists a smooth/\'etale/Kummer \'etale homomorphism of monoids $P_i\to Q_i$ inducing $P\to Q$ via saturated base change. Thus $Y'$ is the saturated base change of $Y'_i = \Spec(Q_i\to A_i\otimes_{\bZ[P_i]}\bZ[Q_i])$. Increasing $i$, by the corresponding fact about \'etale maps of schemes \cite[Exp VII \S 5]{SGA4_2} we may assume that $Y\to Y'$ is the base change of an \'etale morphism $Y''_i\to Y'_i$. Then $Y'_i\to X_i$ is smooth/\'etale/Kummer \'etale and hence so is $Y''_i\to X_i$. Since $Y''_i$ and $Y_i$ have the same saturated base change to $X$, namely $Y$, increasing $i$ we may assume that $Y''_i = Y_i$. So $Y_i\to X_i$ is smooth/\'etale/Kummer \'etale. 
\end{proof}

\begin{cor}\label{cor:SmEtKet-approx}
    Let $\{X_i\}_{i\in I}$ be an affine charted inverse system of qcqs saturated log schemes with limit $X$. Then saturated base change induces equivalences of categories
    \begin{align*}
        \varinjlim \cat{Sm}^\qcqs_{X_i} &\isomto \cat{\bf Sm}^\qcqs_X \\  
        \varinjlim \cat{Et}^\qcqs_{X_i} &\isomto \cat{\bf Et}^\qcqs_X \\  
        \varinjlim \cat{KEt}^\qcqs_{X_i} &\isomto \cat{\bf KEt}^\qcqs_X,
    \end{align*}
    where we denote the category of smooth/\'etale/Kummer \'etale maps with target $X$ by $\cat{Sm}_X$, $\cat{Et}_X$, and $\cat{KEt}_X$ and the superscript signifies the full subcategories of qcqs objects.
\end{cor}

\begin{proof}
Combine \cref{prop:descend-smoothness-etc} with \cref{thm:sfp-approx}.
\end{proof}

This approximation result combined with Kato's criterion \cite[Theorem~3.5]{Kato1989:LogarithmicStructures} and its variant \cref{lem:Kato criterion for Ket} below allows us to obtain the following ``chart lifting property'' for smooth, \'etale, and Kummer \'etale morphisms.

\begin{prop}
\label{prop:chart lifting for smooth etale and ket morphisms}
    Let $f\colon Y\to X$ be a morphism between saturated log schemes. The following conditions are equivalent.
    \begin{enumerate}[(a)]
        \item \label{propitem:map is sm-et-ket} 
            The morphism $f$ is smooth (resp.\ \'etale, resp.\ Kummer \'etale).

        \item \label{propitem:sm-et-ket-chart-lifting}
            For every strict \'etale $U\to X$ and every chart $P\to \cM(U)$ with a saturated monoid $P$, the base change map $Y_U\to U$ locally on source and target admits a chart by a smooth (resp.\ \'etale, resp.\ Kummer \'etale) morphism of monoids $P\to Q$ (with $\Sigma$ the set of primes non-invertible on the respective target) such that the induced strict map $Y_U\la U\times_{\bA_P}\bA_Q$ is \'etale. 
    \end{enumerate}
\end{prop}

\begin{proof}
    The implication \labelcref{propitem:sm-et-ket-chart-lifting}$\Rightarrow$\labelcref{propitem:map is sm-et-ket} is  clear. For the converse \labelcref{propitem:map is sm-et-ket}$\Rightarrow$\labelcref{propitem:sm-et-ket-chart-lifting}
    we may assume that $U = X$ is of the form $\Spec(P\to A)$ and write $(P\to A)=\varinjlim\,(P_i\to A_i)$ with $P_i$ fs and $A_i$ of finite type over $\bZ$. By \cref{cor:SmEtKet-approx}, there exists, for $i$ large enough, a smooth (resp.\ \'etale, resp.\ Kummer \'etale) map $Y_i\to X_i = \Spec(P_i\to A_i)$ whose saturated base change is $Y\to X$. By Kato's criterion \cite[Theorem~3.5]{Kato1989:LogarithmicStructures} and the analogue for Kummer \'etale in \cref{lem:Kato criterion for Ket} below, locally on $Y_i$ there exists a chart for $Y_i\to X_i$ given by a smooth (resp.\ \'etale, resp.\ Kummer \'etale) map of monoids $P_i\to Q_i$ (with respect to primes non-invertible on $X_i$) such that the induced strict map $Y_i\to X_i\times_{\bA_{P_i}}\bA_{Q_i}$ is \'etale. Taking saturated base change to $X$ and $P$ we obtain the required local chart for $Y \to X$.
\end{proof}

Above we used the following fact about Kummer \'etale maps between fs log schemes.

\begin{lem}
\label{lem:Kato criterion for Ket}
    Let $f\colon Y\to X$ be a Kummer \'etale map of fs log schemes and let $\alpha\colon P\to \cM(X)$ be an fs chart. Then locally on $Y$ the chart $P$ lifts to a chart for $f$ given by a Kummer \'etale map of monoids $P\to Q$ such that $Y\to X\times_{\bA_P}\bA_Q$ is (strict) \'etale. 
\end{lem}

\begin{proof}
We will work locally around a geometric point $\ov{x}$ of $X$. Let $F\subseteq P$ be the preimage of $\smash{\cO_{X,\bar{x}}^\times}$ and let $P' = P/F \simeq \ov{\cM}_{X,\bar{x}}$. Then $P_F\to \cM(X)$ is also a chart in a neighbourhood of $\ov{x}$. Moreover, if we solve the problem for this chart, then we solve it also for $P$. Indeed, if $P_F \to Q''$ is a Kummer \'etale map, we set $Q$ to be the saturation of $P$ in $(Q'')^\gp$. Then $P\to Q$ is Kummer \'etale. Moreover, the maps $\bA_{Q''}\to \bA_Q$ and $\bA_{P_F}\to \bA_P$ are strict open immersions. We may thus compose $Y\to X\times_{\bA_{P_F}}\bA_{Q''}$ with the base change of $\bA_{Q''}\to \bA_Q$. 

We may now suppose that $P = P_F$, i.e.\ that $F = F^\gp$. In this case, since $P'$ is sharp and hence $(P')^\gp$ is free, we may write $P = P' \oplus F^\gp$. Then the composition $P'\to P \to \cM(X)$ is also a chart, which is moreover neat at the point $\ov{x}$. 

By \cite[Proposition~3.4.1]{Stix2002:Thesis}, the assertion holds for the neat chart $P'\to \cM(X)$. More precisely, locally around $\ov{x}$ and locally on $Y$ there exists a Kummer \'etale map of monoids $P'\to Q'$ such that $f$ factors through a strict \'etale map to $X\times_{\bA_{P'}}\bA_{Q'}$. We may therefore assume without loss of generality that $Y = X\times_{\bA_{P'}}\bA_{Q'}$.

Let $Q = Q' \oplus F^\gp$. Then $P\to Q$ is Kummer \'etale and it remains to prove that $X\times_{\bA_P}\bA_{Q}$ and $X\times_{\bA_{P'}}\bA_{Q'}$ are isomorphic over $X$. The chart $X\to \bA_P$ factors through the chart $X\to \bA_{P'}$ and the strict map $\bA_{P'}\to\bA_P$ induced by the projection $P = P'\oplus F^\gp \to P'$. Since $Q' = P'\oplus_{P} Q$, we have
\[ 
    Y \,\simeq\, X\times_{\bA_{P'}}\bA_{Q'} 
      \,\simeq\, X\times_{\bA_{P'}}(\bA_{P'}\times_{\bA_P} \bA_Q)
      \,\simeq\, X\times_{\bA_P} \bA_Q. \qedhere
\]
\end{proof}


\subsection{Appendix: Tsuji's example}

In \cref{rmk:refer to Tsuji exmaple in appendix}
we promised a counterexample against the naive approach to constructing charts by global sections of $\cM$, a potential source of charts in verifying the chart lifting property of a morphism, see \cref{rmk:chart-lifting}. However, this does not work in general.

The following construction, due to Takeshi Tsuji, gives an example of an affine fs log scheme $C$ (a punctured elliptic curve) for which the identity map $\cM(C)\to \cM(C)$ is not a chart. 

\begin{ex}[Tsuji]
\label{ex:tsuji}
    Let $k$ be an algebraically closed field, let $E$ be an elliptic curve defined over $k$, and let $x_0\in E(k)$ be a point of infinite order. Let $x_1, x_2, x_3$ be the three distinct points $[-3](x_0)$, $[2](x_0)$, and $x_0$, respectively, where $[n]$ ($n\in\bZ$) denotes the multiplication by $n$ on $E$. Put $C = E\setminus\{0\}$, $U = C\setminus\{x_1, x_2, x_3\}$, and let $j$ denote the open immersion $U\hookrightarrow C$. The scheme $C$ is affine. We endow it with the compactifying log structure $\cM = \cO_C \cap j_*\cO_U^\times$ induced by the open subset $U$. We set $M = \cM(C)$ to be its global sections and let $\cM'$ be the log structure induced by the map $M\isomto \cM(C) \to \cO(C)$. We will show that the stalk $\overline{\cM}{}'_{x_1}$ has rank two while $\overline{\cM}_{x_1} \simeq \bN$. This implies that $\cM$ and $\cM'$ are not isomorphic as sheaves of monoids on $C$. 

    We first compute $M = \cM(C)$. Letting $v_i$ ($i=1,2,3$) denote the discrete valuation of the function field $k(C)$ of $C$ defined by the point $x_i\in C$, we have 
    \[ 
        \cM(C) = \{ g\in \cO(U)^\times \,:\, v_i(x)\geq 0 \text{ for } i=1,2,3\}. 
    \]
    Consider the map $(v_1, v_2, v_3)\colon \cM(C)\to \bN^3$. For $h_1, h_2 \in \cO(U)^\times$ such that $v_i(h_1) = v_i(h_2)$ for all $i=1,2,3$, we have $h_1h_2^{-1} \in \cO(E)^\times = k^\times$. Moreover, for $(n_1, n_2, n_3)\in \bZ^3$, there exists an $f\in \cO(U)^\times$ such that ${\rm div}(f) = n_1 x_1 + n_2 x_2 + n_3 x_3$ if and only if
    \begin{equation} \label{eqn:Tsuji1}
        -3n_1 + 2n_2 + n_3 = 0.
    \end{equation}
    (In particular, $x_1+x_2+x_3$ is a principal divisor, equal to ${\rm div}(g)$ for some function $g\in \cO(C)$, and $n\mapsto g^n \colon \bN\to \cM$ is a chart for $\cM$.) Therefore, by \labelcref{eqn:Tsuji1}, the map $(v_1, v_2, v_3)\colon \cM(C)\to \bN^3$ induces an isomorphism
    \begin{equation} \label{eqn:Tsuji2}
        M/k^\times = \cM(C)/k^\times \isomto \{ (n_1, n_2, n_3)\in\bN^3 \,:\, -3n_1+2n_2+n_3=0\}.
    \end{equation}
    
    To compute $\overline{\cM}{}'_{x_1}$, we note that we have
    \[
        M / \alpha_{\ov{x}_1}^{-1}(\cO^\times_{C,x_1}) \isomto \cM'_{x_1} / \cO^\times_{C,x_1} = \overline{\cM}{}'_{x_1}.
    \]
    For $h\in M$, we have $\alpha_{x_1}(h) \in \cO^\times_{C, x_1}$ if and only if $v_1(h) = 0$. By \labelcref{eqn:Tsuji2}, the latter implies $2v_2(h) + v_3(h) = 0$, $v_2(h) \geq 0$, $v_3(h)\geq 0$, and therefore $v_2(h) = v_3(h) = 0$. It follows that $h\in k^\times$, and we deduce that
    \[ 
        \overline{\cM}{}'_{x_1} \simeq M / \alpha_{x_1}^{-1}(\cO^\times_{C,x_1})  = M/k^\times,
    \]
    which is the monoid \labelcref{eqn:Tsuji2} with groupification $\bZ^2$. On the other hand, we have $\overline{\cM}_{x_1}\simeq \bN$. Thus $\cM$ and $\cM'$ are not isomorphic.
\end{ex}

\section{The Kummer \'etale site and fundamental group}
\label{s:kummer-pi1}

In this section, we use the theory of sfp morphisms developed in \cref{s:log-schemes} to define the Kummer \'etale site (\cref{sec:ket site}) and the Kummer \'etale fundamental group (\cref{ss:finite-ket,ss:ket-pi1}) of a saturated log scheme $X$. In the final \cref{ss:kummer-val-rings}, we exemplify the latter in the case of spectra of valuation rings.  

\subsection{The Kummer \'etale site}
\label{sec:ket site}

On our way to the Kummer \'etale fundamental group in \cref{ss:ket-pi1}, we first define the Kummer \'etale site. See \cite{IllusieFKN} for an overview in the fs case.

\begin{defi}
\label{defi:ket site}
    The \textbf{Kummer \'etale site} $\KEtsite{X}$ of a saturated log scheme $X$ consists of the underlying category of Kummer \'etale morphisms $U \to X$. The topology on $\KEtsite{X}$ is defined by saying that a family of maps $\{f_i\colon U_i\to U\}_{i\in I}$ is a covering family if $|U|=\bigcup_{i\in I} f_i(|U_i|)$. 
\end{defi}

We note that for every object $U \to X$ of $\KEtsite{X}$, the Kummer \'etale site $\KEtsite{U}$ is identified with the slice category $(\KEtsite{X})_{/U}$. This follows from \crefpart{prop:bc and composition and 2 out of 3}{propitem:ket 2 out of 3}.

The axioms of a site for $\KEtsite{X}$ follow from \crefpart{prop:bc and composition and 2 out of 3}{propitem:sm ket base change} and \labelcref{propitem:sm ket composition} once we show that the saturated base change of a jointly surjective family is still jointly surjective. For fs log schemes this follows from Nakayama's four point lemma \cite[Proposition 2.2.2]{NakayamaLogEtCoh}. Here we follow Ogus' version for the base change of an exact morphism \cite[III \S2.2]{Ogus} of fine integral log schemes \cite[Proposition III~2.2.3]{Ogus}. 

For the statement below, recall \cite[III \S 2.2]{Ogus} that a morphism $f\colon Y\to X$ of integral log schemes is {\bf exact} if for every geometric point $\ov{y}\to Y$, the homomorphism of monoids 
\[
    f^*\colon \cM_{X,f(\ov y)}\la \cM_{Y,\ov y}
    \quad 
    \text{(or, equivalently, $\ov{\cM}_{X,f(\ov y)}\la \ov{\cM}_{Y,\ov y}$)}
\]
is exact (\crefpart{defi:exactintsatvert monoids}{defitem:exact monoid map}). We note the following properties of exact morphisms:
\begin{enumerate}[(1)]
    \item \label{item:exact-pullback-comp} 
        It follows from Remark~\labelcref{rmks:sorite exact int sat vert}\labelcref{rmkitem:sorite exact} that exact morphisms of saturated log schemes are stable under composition and saturated pull-back (whenever it exists, cf.\ \cref{lem:saturated fibre product}), and in particular stable under strict pull-back.
    \item \label{item:exact-AQ-AP}
        If $P\to Q$ is an exact homomorphism of integral monoids, then the morphism $\bA_Q\to\bA_P$ is exact. 
    \item \label{item:exact-chart-exact}
        It follows from \labelcref{item:exact-pullback-comp} and \labelcref{item:exact-AQ-AP} that if $f\colon Y\to X$ is a morphism of integral log schemes which locally admits a chart by an exact homomorphism of integral monoids, then $f$ is exact.
    \item 
        Most importantly, observation \labelcref{item:exact-chart-exact} implies that Kummer \'etale maps of saturated log schemes are exact.
\end{enumerate}  

\begin{lem}[Four point lemma for exact morphisms]
\label{lem:4ptLemma}
    Consider a cartesian diagram in the category of saturated log schemes
    \[
        \begin{tikzcd}
            U' \ar[r,"h'"] \ar[d,"f'",swap] & U \ar[d,"f"] \\
            X' \ar[r,"h",swap] & X,
        \end{tikzcd}
    \]
    i.e.\ $U' = X'\times^\sat_X U$, where the morphism $f$ is exact and locally sfp.    
    Then, for every pair of points $u \in U$ and $x' \in X'$ with $f(u) = h(x')$ there exists a point $u' \in U'$ with $f'(u') = x'$ and $h'(u') = u$.
\end{lem}

\begin{proof}
    As saturated base change by a strict map commutes with forgetting the log structure, we may reduce to the case where the underlying schemes of $U$, $X$ and $X'$ all equal $\Spec(k)$ for an algebraically closed field $k$, and $f$ and $h$ are the identity on the underlying scheme. We need to show that the saturated fibre product is non-empty. 

    We write $P = \cM_X(X)$, $P'=\cM_{X'}(X')$ and $Q = \cM_U(U)$, and $Q' = P' \oplus_P Q$ for the pushout in monoids. 
    Since by \cref{prop:geometry sat map} the saturation map of an integral log scheme is surjective, we must show that the integral base change
    \[
    X' \times^\inte_X U = \Spec(Q' \to k)^\inte
    \]
    is non-empty. This amounts to showing that the monoid homomorphism $Q' \to k$ can be extended to a monoid homomorphism $Q'^\inte \to k$. 
    Since $P\to Q$ is exact and $P\to P'$ is local (since $P^\times = (P')^\times = k^\times$), by \cite[Proposition I 4.2.5]{Ogus} we have
    \[
    (Q'^\inte)^\times = P'^\times \oplus_{P^\times} Q^\times = k^\times,
    \]
    and the map $Q'^\inte \to k$ extends this by sending  $Q'^\inte \setminus (Q'^\inte)^\times$ to $0$. 
\end{proof}

Since Kummer \'etale morphisms are exact, we deduce the following result.

\begin{cor} \label{cor:4ptKummer}
    Let $f\colon Y\to X$ and $g\colon X'\to X$ be morphisms of saturated log schemes such that $f$ is Kummer \'etale. Let $f'\colon Y'\to X'$ be the saturated base change of $f\colon Y\to X$. Then 
    \[
        f'(Y') = g^{-1}(f(X)).
    \]
\end{cor}

The fact that a quasi-compact (qcqs) log scheme is also quasi-compact (qcqs) in the Kummer \'etale topology, used later in \cref{prop:limit-Ket-topos}\labelcref{propitem:limit-Ket-coherent}, requires the following lemma, the fs version of which is \cite[Corollary~3.7]{IllusieFKN}. 

\begin{lem}
\label{lem:ket is open}
    Let $f \colon Y\to X$ be a Kummer \'etale map between saturated log schemes. Then the underlying map of schemes $\underline{Y}\to\underline{X}$ is open.
\end{lem}

\begin{proof}
Since the (strict) inclusion $U \subseteq Y$ of an open is Kummer \'etale, it suffices to show that the image $f(Y)$ is open in $X$. This assertion is local in $Y$ and $X$. We may therefore by \cref{prop:fs approximation for smooth etale and ket morphisms} assume that $Y \to X$ is the saturated base change along $h\colon X \to X_0$ of a~Kummer \'etale map $f_0 \colon Y_0 \to X_0$ between fs log schemes. 
By \cite[Corollary~3.7]{IllusieFKN} the image $f_0(Y_0)$ is open, and thus by \cref{cor:4ptKummer} also $f(Y) = h^{-1}(f_0(Y_0))$ is open.
\end{proof}

\Cref{lem:ket is open} combined with \cref{cor:4ptKummer} together imply the following analogue of \cite[Th\'eor\`eme~8.10.5(vi)]{EGAIV3} for Kummer \'etale maps. 

\begin{cor} 
\label{cor:descent-of-surjectivity}
    Let $\{X_i\}_{i\in I}$ be an affine charted inverse system of qcqs saturated log schemes with inverse limit $X$. Let $Y_0\to X_0$ be a qcqs Kummer \'etale map for some $0\in I$, and define $Y_i\to X_i$ ($i\geq 0$) and $Y\to X$ by saturated base change. Then $Y\to X$ is surjective if and only if $Y_i\to X_i$ is surjective for $i\gg 0$.
\end{cor}

\begin{proof}
The ``if'' direction follows from \cref{cor:4ptKummer}. For the ``only if'' direction, suppose that $Y\to X$ is surjective. Let $U_i\subseteq X_i$ be the image of $Y_i\to X_i$. It is a quasi-compact open subset of $X_i$ by \cref{lem:ket is open}. Moreover, again by \cref{cor:4ptKummer}, we have $U_j = U_i\times_{X_i} X_j$ for $j\geq i$, and $X = U_i\times_{X_i} X$. Using \cite[Th\'eor\`eme 8.3.11]{EGAIV3} (which states that the set of quasi-compact opens of $X$ is the colimit of the sets of quasi-compact opens of the $X_i$) we conclude that $U_i = X_i$ for $i\gg 0$.
\end{proof}

\begin{prop} 
\label{prop:Ket-subcanonical}
    The Kummer \'etale site $\KEtsite{X}$ is subcanonical (representable presheaves are sheaves).
\end{prop}

\begin{proof}
    Let $Y\to X$ be a Kummer \'etale map. We shall prove that the presheaf $h_Y = \Hom(-,Y)$ on $\KEtsite{X}$ is a sheaf. By subcanonicality of the usual big \'etale site of schemes, $h_Y$ is a sheaf for the topology on $\KEtsite{X}$ whose coverings are jointly surjective families of strict \'etale maps. In particular, for every family of objects $\{V_i\}_{i\in I}$ we have $h_Y(\coprod V_i) = \prod h_Y(V_i)$. Arguing as in the proof of \stacks[Lemma]{022H}, in order to prove that $h_Y$ is a sheaf it is enough to check the sheaf condition for singleton covering families $\{V\to U\}$ where $U$ and $V$ are affine. We set $W = V\times_U V$ (in this proof, all fibre products are taken in respective sites, i.e.\ signify saturated fibre products). We need to show that the diagram below is an equalizer
    \begin{equation} \label{eqn:subcanonical-equalizer}
        \begin{tikzcd}
            \Hom(U, Y) \ar[r] & \Hom(V, Y) \ar[r,shift left=0.5ex]  \ar[r,shift right=0.5ex] & \Hom(W, Y)
        \end{tikzcd}
    \end{equation}
    We do this in several steps.

    \medskip
    
    \noindent \emph{Step 1: \labelcref{eqn:subcanonical-equalizer} is an equalizer if $X$, $Y$, $U$, and $V$ are qcqs, and in addition if $X$ is the limit of an affine charted inverse system $\{X_i\}_{i\in I}$ with $X_i$ fs and qcqs.} 

    \smallskip
    
    By \cref{cor:SmEtKet-approx}, the schemes and maps in question are base-changes of analogous ones in some $X_{i_0,\ket}$ and we can write  the diagram above as
    \[
        \begin{tikzcd}
            {\displaystyle \varinjlim_i} \Hom(U_i, Y_i) \ar[r] & 
            {\displaystyle \varinjlim_i} \Hom(V_i, Y_i) \ar[r,shift left=0.5ex]  \ar[r,shift right=0.5ex] &
            {\displaystyle \varinjlim_i} \Hom(W_i, Y_i).
        \end{tikzcd}
    \]
    where $W_i = V_i\times_{U_i}V_i$.
    By \cref{cor:descent-of-surjectivity}, we can assume $V_i \to U_i$ is surjective (not just Kummer \'etale), so a cover. We know subcanonicality in the fs case (cf.~\cite[Theorem 3.1]{Kato2021:LogarithmicStructuresII}), so $\Hom(U_i, Y_i)$ is the equalizer of $\Hom(V_i, Y_i) \rightrightarrows  \Hom(W_i, Y_i)$. Now, as filtered colimits commute with equalizers, we finish the proof of this step.
    
    \medskip
    
    \noindent \emph{Step 2: \labelcref{eqn:subcanonical-equalizer} is an equalizer if $X$, $Y$, $U$, and $V$ are qcqs.}

    \smallskip
    
    We prove the same assertion as in Step~1, but without assuming that $X$ can be presented as a~limit. Fix a strict \'etale cover $\{X_\alpha\}$ of $X$ such that each $X_\alpha$ satisfies the assumptions of Step~1 (this is possible by \cref{lem:prelogring_colimit_fs}). 

    Denote the fibre products $X_\alpha\times_X X_\beta$ by $X_{\alpha\beta}$ and similarly for the base-changes of $U, V$ and $Y$. Further cover the fibre products $X_{\alpha\beta}$ by $X_\gamma$ that satisfy the assumptions of Step~1.
    As log schemes (Kummer \'etale) over $X$ form a prestack in the strict \'etale topology, we first get that  $\prod_{\alpha,\beta} \Hom(U_{\alpha\beta},Y_{\alpha,\beta}) \hookrightarrow \prod_\gamma \Hom(U_{\gamma},Y_{\gamma})$ and then that
    \[ 
        \begin{tikzcd}
            \Hom(U,Y) \ar[r] & 
            \Hom(U_\alpha,Y_\alpha)  \ar[r,shift left=0.5ex]  \ar[r,shift right=0.5ex] &
            \prod_\gamma \Hom(U_{\gamma},Y_{\gamma}).
        \end{tikzcd}
    \]
    is an equalizer. Similar equations hold for $\Hom(V,Y)$ and $\Hom(W,Y)$.
    
    Consider the diagram
    \[ 
        \begin{tikzcd}
            \Hom(U, Y) \ar[r] \ar[d] & \Hom(V, Y)\ar[r,shift left=0.5ex] \ar[r,shift right=0.5ex] \ar[d] &  \Hom(V\times_U V, Y) \ar[d] \\
            \prod \Hom(U_\alpha, Y_\alpha) \ar[r] \ar[d,shift left=0.5ex] \ar[d,shift right=0.5ex] & \prod \Hom(V_\alpha, Y_\alpha)\ar[r,shift left=0.5ex] \ar[r,shift right=0.5ex] \ar[d,shift left=0.5ex] \ar[d,shift right=0.5ex] &  \prod \Hom(W_\alpha, Y_\alpha) \ar[d,shift left=0.5ex] \ar[d,shift right=0.5ex] \\
            \prod \Hom(U_\gamma, Y_\gamma) \ar[r] & \prod \Hom(V_\gamma, Y_\gamma) \ar[r,shift left=0.5ex] \ar[r,shift right=0.5ex] & \prod \Hom(W_\gamma, Y_\gamma)
        \end{tikzcd}
    \]
    where each object of the top row is the equalizer in its column and the leftmost elements of the bottom two rows are equalizers. Commuting limits with limits, we see that that the top left element is the equalizer of its row. This finishes the proof of this step.
    
    \medskip
    
    \noindent \emph{Step 3: \labelcref{eqn:subcanonical-equalizer} is an equalizer in general.}

    \smallskip
    
    We now tackle the general case of an arbitrary saturated log scheme $X$ and $Y \in \KEtsite{X}$. By the previous Step, we can assume that $X$ is as in Step 1. If $Y \to X \in \KEtsite{X}$ is qcqs, by previous steps, it follows that $h_Y$ is a sheaf on $\KEtsite{X}$. We now extend it to general $Y$.
    
    We prove injectivity of $h_Y(U)\to h_Y(V)$. Let $f, g\colon U\to Y$ be two maps equalized by the map $\pi\colon V\to U$. To show $f=g$ it is enough to do so locally around every point $u\in U$. Moreover, since $|V|\to |U|$ is surjective, we have $|f|=|g|\colon |U|\to |Y|$. Let $Y_0\subseteq Y$ be an affine neighbourhood of $f(u)=g(u)$. We may replace $U$ with $f^{-1}(Y_0) = g^{-1}(Y_0)$ and then the already established sheaf property of $h_{Y_0}$ implies the injectivity.
    
    To prove surjectivity, let $f'\colon V\to Y$ be a map equalizing $\pi_1, \pi_2\colon W\to V$. We first note that $|f'|\colon |V|\to |Y|$ factors uniquely through a continuous map $\varphi\colon |U|\to |Y|$. This follows from the fact that $|W|\to |V|\times_{|U|}|V|$ is surjective and that $|V|\to |U|$ is a topological quotient map (as it is a surjective open map). Let $\{Y_\lambda\}_{\lambda\in A}$ be an affine open cover of $Y$ and let 
    \[
        U_\lambda = \varphi^{-1}(Y_\lambda), \quad
        V_\lambda = \pi^{-1} \varphi^{-1}(Y_\lambda) = (f')^{-1}(Y_\lambda), \quad \text{and} \quad
        W_\lambda = \pi_1^{-1}(V_\lambda) = V_{\lambda}\times_{U_\lambda} V_\lambda
    \]
    be their preimages in $U$, $V$, and $W$, respectively.
    
    Consider the Zariski sheaf $\cF$ (resp.\ $\cF'$) on $U$ (resp.\ $V$) defined by $h_Y$. We have an injection of sheaves $\cF\to \pi_*(\cF')$ and a section $f'$ of the latter which we want to lift to a section of the former. By the Zariski sheaf property, it is enough to do so locally on every $U_\lambda$. 
    
    Note that $Y_\lambda \to X$ is qcqs (as both are qcqs), so $h_{Y_\lambda}$ is a sheaf on $\KEtsite{X}$ by previous steps. Since the restriction of $f'$ to $V_\lambda = \pi^{-1}(U_\lambda)$ lies in $h_{Y_\lambda}(V_\lambda)\subseteq h_Y(V_\lambda) = (\pi_* \cF')(U_\lambda)$ and $h_{Y_\lambda}$ is a~sheaf, we get that the diagram 
    \[
        \begin{tikzcd}
            h_{Y_\lambda}(U_\lambda) \ar[r] & h_{Y_\lambda}(V_\lambda) \ar[r,shift left=0.5ex] \ar[r,shift right=0.5ex] & \prod h_{Y_\lambda}(W_\lambda) 
        \end{tikzcd}
    \]
    is an equalizer. Thus, we obtain a unique $f_\lambda\colon U_\lambda\to Y_\lambda\subseteq Y$ lifting $f'|_{U_\lambda}$. This finishes the proof.
\end{proof}

We have the following version of topological invariance \cite{Vidal2001:MorphismesLogEtales}. Since the universal homeomorphism is supposed to be strict, our result is slightly weaker in the fs case, as Vidal allows certain non-strict universal homeomorphisms such as the Frobenius map.

\begin{prop}[Topological invariance]
\label{prop:top invariance ket site}
    Let $i\colon X_0\to X$ be a strict morphism of saturated log schemes such that $\underline{X}_0\to\underline{X}$ is a universal homeomorphism. Then the induced functor $\KEtsite{X}\to X_{0,\ket}$ is an equivalence.
\end{prop}
\begin{proof}
    We first show that the functor is fully faithful. The graph $Z(f)\subseteq U\times_X^\sat V$ of a map $f\colon U \to V$ in $\KEtsite{X}$ is the saturated base change of the relative diagonal $\Delta_{V/X}\colon V \to V \times^\sat_X V$. It follows from  \cref{lem:Delta of ket is strict open} that $Z(f)$ is a strict open subset of $W = U \times^\sat_X V$ such that the restriction of the first projection $\pr \colon Z(f) \to U$ is an isomorphism. Conversely, any such open in $W$ is a~graph. We denote the base change by $i$ with an index $0$. The map $\Hom_X(U,V) \to \Hom_{X_0}(U_0,V_0)$ has thus been identified with the map 
    \[
    \{Z \subseteq W \ : \ \text{open with } \pr \colon Z \isomto U\} 
    \la 
    \{Z_0 \subseteq W_0 \ : \ \text{open with } \pr_0 \colon Z_0 \isomto U_0\}
    \]
    induced by base change. This is clearly injective as the base change $i_W \colon W_0 \to W$ is a universal homeomorphism. It is also surjective, because any $Z_0$ determines an open $Z \subseteq W$ such that the induced map $\pr \colon Z \to U$ sits in a commutative square
    \[
    \begin{tikzcd}
    Z_0   \ar[r,"i_Z"] \ar[d,"\pr_0",swap]  
    & Z \ar[d,"\pr"]  \\
    U_0   \ar[r,"i_U",swap] 
    & U,
    \end{tikzcd}    
    \]
    where all maps except possibly for $\pr$ are universal homeomorphisms. Thus $\pr\colon Z \to U$ is Kummer \'etale and a universal homeomorphism. More precisely, all maps in the diagram induce isomorphisms of \'etale sites, and three of the maps are strict. Hence also the fourth map $\pr \colon Z \to U$ is strict, and hence strict \'etale. Being strict \'etale and a universal homeomorphism, we conclude that $Z \to U$ is an isomorphism. 

    To complete the proof we must show that every Kummer \'etale $U_0 \to X_0$ comes by base change from a Kummer \'etale $U \to X$. Since we already know the fully faithful part, this amounts to a~strict \'etale local assertion on $X$ and a Zariski local assertion on $U_0$. We may thus assume that $X$ is affine with a global chart, and $U_0$ is affine. Since universal homeomorphisms are affine \stacks[Lemma]{04DE}, then also $X_0$ is affine.

    We claim that the assertion is in fact strict \'etale local on $U_0$. Suppose that $V_0\to U_0$ is a strict \'etale surjection and that $V_0 = V\times_X X_0$ is in the essential image. Let $W_0 = V_0\times_{U_0} V_0$, then either projection $W_0\to V_0$ is strict \'etale, and $V_0\to V$ is a universal homeomorphism, therefore there exists a strict \'etale $W\to V$ with $W_0 = W\times_{V} V_0 = W\times_X X_0$. In particular, $W$ is in the essential image. By fully faithfulness, both morphisms $W_0 \to V_0$ extend uniquely to morphisms $W\to V$, and $W\to V\times_X V$ is an \'etale equivalence relation since so is $W_0\to V_0\times_{X_0} V_0$. Let $U$ be the algebraic space $V/W$ over $X$, then its base change to $X_0$ is $U_0$, which is a scheme. Therefore $U$ is a scheme by \cref{lem:alg-sp-univ-homeo}. Thus $U_0 = U\times_X X_0$ with the induced log structure is in the essential image.  

    The paragraph above allows us to work strict \'etale locally on $U_0$. We may therefore assume that there exists a Kummer \'etale homomorphism of monoids $P\to Q$ and a strict \'etale map 
    \[
        U_0 \la Y_0 = X_0\times_{\bA_P}\bA_Q.
    \]
    Since $Y_0 = X_0\times_{\bA_P}\bA_Q\to X\times_{\bA_P}\bA_Q =Y$ is a universal homeomorphism, there exists a strict \'etale map $U\to Y$ with $U_0 = U\times_Y Y_0 = U\times_X X_0$, and we are done. 
\end{proof}

The easy proof above relied on the following annoying lemma.

\begin{lem} 
\label{lem:alg-sp-univ-homeo}
    Let $X_0\to X$ be a universal homeomorphism of schemes and let $U\to X$ be an algebraic space. If the base change $U_0$ of $U$ to $X_0$ is a scheme, then $U$ is a scheme. 
\end{lem}

\begin{proof}
Since this is local on $X$, we may assume that $X = \Spec(A)$ is affine. Since universal homeomorphisms are integral and in particular affine, we have that $X_0 = \Spec(B)$ is affine as well. 

Let $\{V_{0,\alpha}\}$ be an affine open cover of $U_0$. Since $U_0\to U$ is a universal homeomorphism of algebraic spaces (meaning that its pullback to any scheme over $U$ is a (universal) homeomorphism of schemes), it induces an equivalence between the \'etale site of $U$ and the \'etale site of $U_0$. Therefore there exists an \'etale cover $\{V_\alpha\to U\}$ such that $V_{0,\alpha} = V_\alpha\times_U U_0$ for all $\alpha$. 

We claim that 
\begin{enumerate}[(1)]
    \item \label{item:uhclaim1} the maps $V_\alpha\to U$ are open immersions of algebraic spaces,
    \item \label{item:uhclaim2} $V_\alpha$ are affine schemes.
\end{enumerate}
Since the maps $V_\alpha\to U$ are jointly surjective, this will imply that $U$ is a scheme.

Claim \labelcref{item:uhclaim1} can be checked upon pull-back to any scheme $Y$ over $U$. Since $Y_0 = Y\times_U U_0\to Y$ is a universal homeomorphism, the geometric fibres of $V_\alpha\times_U Y\to Y$ are the same as the geometric fibres of $V_{0,\alpha}\times_{U_0} Y_0\to Y_0$, which means that they are of cardinality $\leq 1$. By \cite[17.2.6]{EGAIV4} this implies that $V_\alpha\times_U Y\to Y$ is a monomorphism, and since it is also \'etale, it is an open immersion by \stacks[Theorem]{025G}.

Showing \labelcref{item:uhclaim2} means proving the assertion of the lemma with the additional assumption that $U_0$ is affine. To this end, we will reduce the proof to the situation where we can apply \stacks[Lemma]{07VP}, which states that if $Y'\to Y$ is a finite surjective morphism from an affine scheme $Y'$ to a Noetherian algebraic space $Y$, then $Y$ is an affine scheme as well. 

First, we reduce to the case $X_0\to X$ finitely presented. By \stacks[Lemma]{0EUJ} we may write $B = \varinjlim B_\lambda$ where $B_\lambda$ are finitely presented over $A$ and $\Spec(B_\lambda)\to \Spec(A)$ is a universal homeomorphism. Then $U_0 = \varprojlim\, (U\times_X \Spec(B_\lambda))$, and since $U_0$ is an affine scheme, by \cite[Theorem~C]{RydhApproximation} we have that $U\times_X \Spec(B_\lambda)$ is an affine scheme for $\lambda\gg 0$. Therefore we may replace $X_0\to X$ with $\Spec(B_\lambda)\to X$, and hence assume that $X_0\to X$ is finitely presented (and hence finite). 

Since $U$ is homeomorphic to $U_0$, it is qcqs, and hence by \cite[Theorem~D]{RydhApproximation} we may write $U = \varprojlim U^i$ where $U^i$ are algebraic spaces of finite type over $\bZ$ and the transition maps are affine. Since $U\to U^i$ is affine, it suffices to show that $U^i$ is an affine scheme for $i\gg 0$. 

Since $X_0\to X$ and hence $U_0\to U$ is of finite presentation, by \cite[Proposition B.2]{RydhApproximation} for $i\gg 0$ there exists a finitely presented morphism $U^i_0\to U^i$ whose base change to $U$ is $U_0\to U$. Since $U_0\to U$ is finite and surjective, by \cite[Proposition B.3]{RydhApproximation} by increasing $i$ we may ensure that $U^i_0\to U^i$ is finite and surjective for $i\gg 0$. Since $U_0 = \varprojlim U_0^i$ is an affine scheme, arguing as before we see that increasing $i$ further we may ensure that $U_0^i$ is an affine scheme. We can now apply \stacks[Lemma]{07VP} as promised.
\end{proof}

\begin{rmk}
    It seems we are not far off from answering the following question. Let $X_0\to X$ be a separated universal homeomorphism of algebraic spaces. If $X_0$ is a scheme, must $X$ be a scheme as well? What is missing is a global version of \stacks[Lemma]{0EUJ} over a qcqs algebraic space. Note that the result is well-known if $X_0\to X$ is a closed immersion (see \cite[Theorem~2.2.5]{ConradModuli} or \stacks[Lemma]{0BPW}). 
\end{rmk}

\begin{prop} 
\label{prop:limit-Ket-topos}
    Let $X$ be a qcqs saturated log scheme. Then, 
    \begin{enumerate}[(a)]
        \item \label{propitem:limit-Ket-coherent} $\cat{Sh}(\KEtsite{X})$ is a coherent topos  \cite[Exp.\ VI, \S 2]{SGA4_2},
        \item \label{propitem:limit-Ket-limit} if $X = \varprojlim X_i$ is the limit of an affine charted inverse system of qcqs saturated log schemes $\{X_i\}_{i\in I}$, then
        \[ 
            \cat{Sh}(\KEtsite{X}) = \varprojlim \cat{Sh}(X_{i,\ket})
        \]
        (see \cite[Exp.\ VI, \S 8]{SGA4_2} for inverse limits of topoi and \cite[Exp.\ VII, \S 5]{SGA4_2} for the analogous fact for \'etale sites of schemes).
    \end{enumerate}
\end{prop}

\begin{proof}
\labelcref{propitem:limit-Ket-coherent} 
This follows in the same way as for the \'etale site of a scheme (\cite[Exp.\ VII, \S 5]{SGA4_2}) using the fact that Kummer \'etale maps are open (\cref{lem:ket is open}), so that an object whose underlying scheme is qcqs is coherent. 

\labelcref{propitem:limit-Ket-limit} 
As in \cite[Exp.\ VII, \S 5]{SGA4_2}, we replace the Kummer \'etale site with the restricted Kummer \'etale site $\KEtsite{X}^\qcqs$ consisting of sfp Kummer \'etale maps $Y\to X$, with finite covering families. By \cref{cor:SmEtKet-approx}, we see that $\KEtsite{X}^\qcqs = \varinjlim X_{i,\ket}^\qcqs$ as categories. Moreover, by \cref{cor:descent-of-surjectivity}, a~family of morphisms $\{Y_\alpha\to Y\}$ in $\KEtsite{X}^\qcqs$ is a covering family if and only if it is the base change of a covering family $\{Y_{\alpha,i}\to Y_i\}$ in $X_{i,\ket}^\qcqs$. This implies that $X^\qcqs_\ket$ is equivalent to the inductive limit of the sites $X^\qcqs_{i,\ket}$ \cite[Exp.\ VI, 8.2.5]{SGA4_2}, which implies the assertion about topoi by \cite[Exp.\ VI, 8.2.3]{SGA4_2}.
\end{proof}

\subsection{Finite Kummer \'etale maps}
\label{ss:finite-ket}

Finite Kummer \'etale maps are the covering spaces used for defining the Kummer \'etale fundamental group. Despite their name, they are only finite up to saturation --- see \cref{rmk:fKet-not-finite} below.

\begin{defi}
\label{defi:fet}
    Let $X$ be a saturated log scheme.
    \begin{enumerate}[(1)]
        \item A morphism $Y\to X$ is \textbf{finite Kummer \'etale} if it is Kummer \'etale (\cref{defi:smooth etale ket}) and if the underlying morphism of schemes is integral.
        \item The category of finite Kummer \'etale log schemes over $X$ is denoted by $\FEt_X$.
    \end{enumerate}
\end{defi}

\begin{rmk} \label{rmk:fKet-not-finite}
The morphism of schemes underlying a finite Kummer \'etale map need not be finite. We chose to use ``finite Kummer \'etale'' for compatibility with established terminology for fs log schemes. Moreover, we are going to show that finite Kummer \'etale covers capture the monodromy of finite locally constant sheaves, hence the terminology \textit{finite} resonates with that.

We construct an example of a non-finite but finite Kummer \'etale map from  \crefpart{ex:typeV-not-fg}{exitem:typeV-not-fg2}. With $V = (\bZ \oplus \bZ)^+$ the non-negative elements in lexicographic ordering, the map $V \to \frac{1}{2}V$ is a~chart for a finite Kummer \'etale map $\Spec R[\frac{1}{2}V] \to \Spec R[V]$ as long as $2$ is invertible in $R$. But this map is not finite, since otherwise $\frac{1}{2}V$ would be finitely generated over $V$. Other examples will be described in Examples~\labelcref{ex:type3,ex:type5,ex:perfectoid} when we discuss tame extensions of henselian valued fields in relation to finite Kummer \'etale covers of their valuation rings. 
\end{rmk}

\begin{lem}
    The underlying scheme morphism of a finite Kummer \'etale map is both open and closed.
\end{lem}

\begin{proof}
It is closed, because it is integral, and it is open by \cref{lem:ket is open}.
\end{proof}

The following proposition gathers some basic properties of finite Kummer \'etale maps. 

\begin{prop} \label{prop:fKet-properties}
    Let $X$ be a saturated log scheme. 
    \begin{enumerate}[(a)]
        \item \label{propitem:fKet-properties-pullback} Finite Kummer \'etale maps are stable under saturated pullback and composition. 
        \item \label{propitem:fKet-properties-2of3} Every morphism between objects of $\FEt_X$ is finite Kummer \'etale.
        \item \label{propitem:fKet-properties-stack} $U\mapsto \FEt_U$ forms a stack for the strict \'etale topology on $X$.\footnote{We will see in \cref{cor:fKet-Ket-stack} that it is even a stack in the Kummer \'etale topology.}
        \item \label{propitem:fKet-properties-uh} If $X_0\to X$ is a strict universal homeomorphism, then the induced functor 
        \[
            \FEt_X\la \FEt_{X_0}
        \]
        is an equivalence. 
    \end{enumerate}
\end{prop}

\begin{proof}
\labelcref{propitem:fKet-properties-pullback} and \labelcref{propitem:fKet-properties-2of3}: 
We have established these for Kummer \'etale maps in \cref{prop:bc and composition and 2 out of 3}, and they hold for integral maps of schemes. To account for the discrepancy between pullback of schemes and saturated pullback, we use the fact that the saturation map is integral (\cref{prop:geometry sat map}). 

\labelcref{propitem:fKet-properties-stack}:
Integral morphisms of schemes $\underline{Y}\to \underline{X}$ satisfy \'etale descent on $\underline{X}$. Further, log structures on the scheme $\underline{Y}$ over $\underline{X}$ together with an upgrade of $\underline{Y}\to\underline{X}$ to a morphism of log schemes $Y\to X$ satisfy \'etale descent on $\underline{Y}$ (and hence also on $\underline{X}$) basically by definition. Finally, checking that a morphism $Y\to X$ is Kummer \'etale can be done \'etale locally on $\underline{X}$. 

\labelcref{propitem:fKet-properties-uh}:
Using \cref{prop:top invariance ket site} it is enough to show that the morphism of schemes $Y\to X$ is integral if and only if its restriction $Y_0\to X_0$ is integral. The ``only if'' is clear since integral morphisms of schemes are stable under base change. For the ``if'' part, we use the following characterisation \stacks[Lemma]{01WM}: a morphism of schemes is integral if and only if it is affine and universally closed. Suppose $Y_0\to X_0$ is integral. We may assume $X$ is affine, and therefore both $X_0$ and $Y_0$ are affine. Since $Y_0\to Y$ is a universal homeomorphism, also $Y$ is affine by \stacks[Lemma]{01ZT}, showing that $Y\to X$ is an affine morphism. The fact that $Y\to X$ is universally closed follows trivially from $X_0\to X$ being a universal homeomorphism. 
\end{proof}

In order to prove that finite Kummer \'etale maps can be approximated by \emph{finite} Kummer \'etale maps, we will need the following scheme-theoretic lemma.

\begin{lem} \label{lem:limit_quasifinite_integral}
    Let $X = \varprojlim_{i \in I} X_i$ and $Y = \varprojlim_{i \in I} Y_i$ be cofiltered limits of qcqs schemes with affine transition maps.
    Suppose that we have compatible quasi-finite, separated morphisms
    \[
     f_i \colon Y_i \longrightarrow X_i.
    \]
    with integral diagonal
    \[
     \Delta_{ij} \colon Y_j \longrightarrow Y_i \times_{X_i} X_j
    \]
    for $j \geq i$ in~$I$.
    If the induced morphism $f\colon Y \to X$ is integral, then there exists an $i \in I$ such that $f_i \colon Y_i \to X_i$ is finite.
\end{lem}

\begin{proof}
    For each $i \in I$ we denote by~$\bar{Y}_i$ the normalization of~$X_i$ in~$Y_i$.
    We obtain compatible commutative diagrams
    \[
     \begin{tikzcd}
         Y_i  \ar[rr,"\iota_i"] \ar[dr,"f_i"']   &  & \bar{Y}_i  \ar[dl,"\bar{f}_i"]  \\
                                                                            & X_i.
     \end{tikzcd}
    \]
    Now Zariski's main theorem in the version \stacks[Lemma]{02LR} states that~$\iota_i$ is a quasi-compact open immersion with dense image and~$\bar{f}_i$ is integral.
    Taking the base change to~$X_j$ for $j \geq i$ in~$I$ yields a diagram
    \[
     \begin{tikzcd}
         Y_i \times_{X_i} X_j \ar[rr] \ar[dr]      &        & \bar{Y}_i \times_{X_i} X_j   \ar[dl]    \\
                                          & X_j.
     \end{tikzcd}
    \]
    The horizontal map remains a quasi-compact open immersion and the right hand vertical map is integral.
    Using \stacks[Lemma]{035I} and \stacks[Lemma]{035J}, we obtain a factorisation 
    \begin{equation} \label{eqn:factorisation_ij}
     \begin{tikzcd}
         Y_j        \ar[d,"\Delta_{ij}"]  \ar[rr,open] \ar[ddr,bend right=90,looseness=1.5,"f_j"']              &        & \bar{Y}_j  \ar[d,"\bar{\Delta}_{ij}"']  \ar[ddl,bend left=90,looseness=1.5,"\bar{f}_j"]    \\
         Y_i \times_{X_i} X_j \ar[rr,open] \ar[dr]      &        & \bar{Y}_i \times_{X_i} X_j   \ar[dl]    \\
                                          & X_j.
     \end{tikzcd}
    \end{equation}
    and in particular a morphism $\bar{Y}_j \to \bar{Y}_i$.
    
    We claim that the resulting morphism
    \[
     Y  \longrightarrow \varprojlim_{i \in I} \bar{Y}_i
    \]
    is an isomorphism.
    On the one hand, $Y \to \overline{Y} = \varprojlim_{i \in I} \overline{Y}_i$ is a cofiltered limit of open immersions with dense image. 
    So~$Y$ is pro-open in $\overline{Y}$ with dense image. To see the latter assertion, let $W\subseteq \overline{Y}$ be a non-empty qc open. Such a $W$ is the preimage of a qc open $W_i\subseteq \bar Y_i$. We define $W_j = W_i\times_{\bar Y_i} \bar Y_j$ for $j\geq i$. Then $W_j\cap Y_j$ is constructible and hence compact in the constructible topology. It is moreover non-empty (as $Y_i$ is dense in $\bar Y_i$), and hence $W\cap Y = \varprojlim (W_j\cap Y_j)$ is non-empty. Thus $Y$ intersects every qc open of $\overline{Y}$ and therefore is dense. 
    
    On the other hand, $Y$ being integral over~$X$ implies that~$Y$ is also integral over $\varprojlim_{i \in I} \bar{Y}_i$.
    Combining both pieces of information we conclude that $Y \cong \varprojlim_{i \in I} \bar{Y}_i$. 
    
    Let $Z_i$ be the complement of~$Y_i$ in~$\bar{Y_i}$.
    It is closed and thus compact in the constructible topology of~$\bar{Y}_i$.
    We want to show that for $j \to i$ in~$I$ the transition map $\bar{Y}_j \to \bar{Y}_i$ maps~$Z_j$ to~$Z_i$.
    Let us go back to diagram~\labelcref{eqn:factorisation_ij} and look at the upper commutative square.
    The preimage
    \[
     \bar{\Delta}_{ij}^{-1}(Y_i \times_{X_i} X_j)
    \]
    is an open of~$\bar{Y}_j$ containing~$Y_j$ as a dense open subset which is moreover integral over $Y_i \times_{X_i} X_j$.
    But by assumption~$Y_j$ is also integral over $Y_i \times_{X_i} X_j$ and thus
    \[
     Y_j = \bar{\Delta}_{ij}^{-1}(Y_i \times_{X_i} X_j).
    \]
    This implies that~$Z_j$ is the preimage of the complement of $Y_i \times_{X_i} X_j$ in $\bar{Y}_i \times_{X_i} X_j$, which in turn is the preimage of~$Z_i$ under the projection to $\bar{Y}_j$.
    In particular,~$Z_j$ maps to~$Z_i$.
    
    By assumption the limit $\varprojlim_{i \in I} Z_i$ is empty.
    But $Z_i$ being compact (in the constructible topology) there has to be $i \in I$ with $Z_i = \varnothing$.
    In other words this means that $Y_i = \bar{Y}_i$ is integral over~$X_i$.
    But~$Y_i$ is also quasi-finite, hence of finite type, over~$X_i$.
    Therefore, $Y_i \to X_i$ is finite.
\end{proof}

Thanks to the above lemma, we have the following approximation result, which will allow us to easily extend well-known facts about finite Kummer \'etale maps between fs log schemes to our setting.

\begin{prop} 
\label{prop:fKet-approx}
    Let $(X_i)_{i \in I}$ be an affine charted inverse system of qcqs saturated log schemes with limit $X$. Then the saturated base change functors induce an equivalence of categories
    \[ 
        \varinjlim \FEt_{X_i} \isomto \FEt_X.
    \]
\end{prop}

\begin{proof}
By \cref{cor:SmEtKet-approx}, the functor is fully faithful. To show essential surjectivity, let $Y\to X$ be a finite Kummer \'etale map. Again by \cref{cor:SmEtKet-approx}, there exists a Kummer \'etale map $Y_0\to X_0$ for some $0\in I$ whose saturated base change is $Y\to X$. It remains to show that for $i\gg 0$ the maps $Y_i\to X_i$ obtained by saturated pullback along $X_i\to X_0$ are integral and hence finite Kummer \'etale. This is local on $X$, and hence we may assume (by definition of an affine charted system) that $X=\Spec(P\to A)$ and $X_i = \Spec(P_i\to A_i)$ for a direct system of saturated prelog rings $(P_i\to A_i)$ with colimit $(P\to A)$. 

Suppose first that all $X_i$ are fs and of finite type over $\bZ$. Then the maps $Y_i\to X_i$ are quasi-finite, being finitely presented and \'etale locally isomorphic to the pull-back of a standard Kummer \'etale map $\bA_Q\to \bA_P$, which is finite. Moreover, since $Y$ is affine, \cite[Proposition~C.6]{ThomasonTrobaugh} (with $\Lambda=\bZ$) implies that $Y_i$ is affine $i\gg 0$. In particular, the map $Y_i\to X_i$ is separated. Moreover, this map factors into
\[
    Y_i \longrightarrow Y_0 \times_{X_0} X_i \longrightarrow X_i,
\]
where the first morphism is the saturation of the ordinary fibre product. As saturations are integral by \crefpart{prop:geometry sat map}{propitem:sat map integral}, we are in the position to apply \cref{lem:limit_quasifinite_integral}. This gives us precisely the statement we want to prove.

For the general case, we argue as in the end of the proof of \cref{thm:sfp-approx}. 
\end{proof}

The next lemma shows that finite Kummer \'etale covers are \'etale locally of a standard form.

\begin{lem} 
\label{lem:fKet-local-structure}
    Let $Y\to X$ be a finite Kummer \'etale map and let $P\to \cM(X)$ be a chart by a~saturated monoid $P$. Then \'etale locally on $X$ there exists a finite collection of Kummer \'etale maps $P\to Q_j$, for $j=1,\ldots, r$, and an isomorphism over $X$ 
    \[
        Y \simeq \coprod_{i=1}^r X\times_{\bA_P}\bA_{Q_j}.
    \]
\end{lem}

\begin{proof}
Suppose first that $X=\Spec(P\xrightarrow{\alpha} A)$ is Noetherian and strictly local and that $P$ fs. Let $x\in X$ be the closed point. Let $F = \alpha^{-1}(A^\times)$ and $P' = P/F = \overline{\cM}_{X,x}$. By \cite[Proposition II 2.3.7]{Ogus}, $X$ admits a neat chart $P'\to A$. By \cite[Proposition~3.1.10]{Stix2002:Thesis}, the assertion holds for $X$ equipped with this neat chart. More precisely, if $Y\to X$ is a finite Kummer \'etale map with $Y$ non-empty and connected, and $y\in Y$ is the (unique) point above $x$, then the homomorphism $P'\to Q' = \overline{\cM}_{Y,y}$ is Kummer \'etale and 
\[ 
    Y \simeq X\times_{\bA_{P'}} \bA_{Q'}.
\]
It is therefore enough to express $X\times_{\bA_{P'}} \bA_{Q'}$ as $X\times_{\bA_P}\bA_Q$ for a Kummer \'etale map $P\to Q$. Since $P'$ is fs and sharp, the group $(P')^\gp$ is free, and hence we may write $P^\gp = (P')^\gp\oplus F^\gp$. Let $Q$ be the saturation of $P$ in $(Q')^\gp\oplus F^\gp$. Let $Z = X\times_{\bA_P}\bA_Q$. We check easily that there exists a unique point $z\in Z$ above $x$ (in particular, since $A$ is henselian, $Z$ is connected), and that $\overline{\cM}_{Z,z} \simeq Q'$. Thus, again by \cite[Proposition~3.1.10]{Stix2002:Thesis}, we have $Z\simeq Y$ over $X$. 

Next, suppose that $X$ is Noetherian and that $P$ is fs. Let $\ov{x}\to X$ be a geometric point and let $A = \cO_{X,(\ov x)}$ be the corresponding strictly henselian local ring. By the above paragraph, we obtain a finite number of Kummer \'etale homomorphisms $P\to Q_j$ and an isomorphism $\varphi\colon Y_A \simeq \coprod_j \Spec(A)\times_{\bA_P}\bA_{Q_j}$. Expressing $\Spec(A)$ as the inverse limit of affine (strict) \'etale neighbourhoods of $\ov{x}\to X$, by \cref{prop:fKet-approx} we can spread out $\varphi$ to one of these neighbourhoods. This shows the assertion for $X$.

Finally, we treat the general case. Working strict \'etale locally, we may assume that $X$ has the form $\Spec(P\to A)$ and write $(P\to A) = \varinjlim (P_i\to A_i)$ where $P_i$ are fs and $A_i$ are of finite type over $\bZ$. Let $Y\to X$ be a finite Kummer \'etale map. By \cref{prop:fKet-approx}, there exists an index $i$ and a finite Kummer \'etale map $Y_i\to X_i = \Spec(P_i\to A_i)$ whose saturated base change is $Y\to X$. By the previous paragraph, working locally on $X_i$ we find a finite number of Kummer \'etale homomorphisms $P_i\to Q_{ij}$ ($j=1, \ldots, r$) and an isomorphism $Y_i\simeq \coprod X_i\times_{\bA_{P_i}}\bA_{Q_{ij}}$. Saturated base change to $X$ then yields the result with $Q_j = (Q_{ij}\oplus_{P_i} P)^\sat$.  
\end{proof}

\subsection{The Kummer \'etale fundamental group}
\label{ss:ket-pi1}

Before defining the Kummer \'etale fundamental group, we first need to discuss Kummer \'etale coverings of strictly local log points. 

\begin{prop} \label{prop:local log pt}
    Let $X = \Spec(P\xrightarrow{\alpha} A)$ where $P$ is a saturated monoid and $A$ is strictly local whose residue field $k$ has exponential characteristic $p$. Let $F = \alpha^{-1}(A^\times)\subseteq P$, and let
    \[
        P' = P/F = \overline{\cM}_{X,\bar{x}}
    \]
    where $\ov x$ is the closed point of $X$. Then, there is an equivalence of categories (constructed explicitly in \eqref{eqn:strlocpt-fibre})
    \[
        \FEt_X \isomto \Gsets{\pi_1(P')}
    \]
    (see \cref{def:pi1-of-monoid} for the meaning of $\pi_1(P')$).
\end{prop}

The proof will be slightly more involved than in the fs case \cite[\S 3]{Stix2002:Thesis} since we cannot assume the existence of a chart by the monoid $P'$ (a ``neat chart'') and therefore reduce to the case $P = \overline{\cM}_{X,\bar{x}}$. The issue is that $X_\infty$ defined in the proof might be disconnected if $F\neq 0$, in which case its automorphism group over $X$ will be too large. 

\begin{proof}
Let $P_\infty$ be the saturation of $P$ in $P_\infty^\gp = P^\gp\otimes \bZ_{(p)}$. We set 
\[ 
    X_\infty =  X \times_{\bA_P} \bA_{P_\infty} = \Spec(P_\infty \to A\otimes_{\bZ[P]}\bZ[P_\infty]).
\]
We pick a section $\sigma\colon \Spec(k)\to X_\infty$ and let $Y$ be the connected component of $X_\infty$ containing the image of $\sigma$. 

Concretely, $\sigma^*\colon A\otimes_{\bZ[P]}\bZ[P_\infty]\to k$ corresponds to a homomorphism $\beta\colon P_\infty\to k$ making the following diagram commute
\[ 
    \begin{tikzcd}
        P\ar[r] \ar[d,"\alpha",swap] & P_\infty \ar[d,"\beta"] \\
        A \ar[r,"\bar{x}^\ast",swap] & k.
    \end{tikzcd}
\]
By definition of $P$, this datum can be interpreted as a compatible choice of prime-to-$p$ roots of $\bar{x}^\ast \alpha(q)$ for $q\in P$. 
Let $F_\infty\subseteq P_\infty$ be the saturation of $F$ in $P_\infty$. Since $A$ is henselian, for every $q \in F_\infty$ there exists a unique $\chi(q)\in A^\times$ lifting $\beta(q)\in k^\times$ such that $\chi(q)^m = \alpha(mq)$ for every $m\geq 1$ such that $mq\in F$. In other words, there exists a unique $\chi\colon F_\infty\to A^\times$ making the diagram below commute
\[ 
    \begin{tikzcd}
        F\ar[r] \ar[d,"\alpha|_F",swap] & F_\infty \ar[d,"\beta|_{F_\infty}"] \ar[dl,"\chi",swap]\\
        A^\times \ar[r] & k^\times.
    \end{tikzcd}
\]

The group $\pi_1(P)$ acts on the log scheme $X_\infty$. Namely, $\gamma\in \pi_1(P) = \Hom(P_\infty^\gp/P^\gp, A^\times)$ sends $a\otimes q \in A\otimes_{\bZ[P]}\bZ[P_\infty]$ to 
\[
\gamma\cdot (a \otimes q) = \gamma(q)a\otimes q
\ \in \cO(X_\infty)
\]
and  $q \in P_\infty$ to 
\[
\gamma\cdot q = \gamma(q)q \ \in \cM(X_\infty).
\]
The projection $X_\infty\to X$ is $\pi_1(P)$-invariant.

We claim that the subgroup $\pi_1(P')\subseteq \pi_1(P)$ stabilizes the point $y = \sigma(x) \in Y$ (in fact, $\pi_1(P')$ is the kernel of $\pi_1(P)\to \Aut(\pi_0(X_\infty))$). This follows from the commutativity of 
\[
    \begin{tikzcd}
        A \otimes_{\bZ[P]} \bZ[P_\infty] \ar[dr,"\sigma^\ast",swap] \ar[rr,"\gamma\in\pi_1(P/F)"]
        & & 
        A \otimes_{\bZ[P]} \bZ[P_\infty] \ar[dl,"\sigma^\ast"] \\
        & k
    \end{tikzcd}
\]
Indeed, $1\otimes q\mapsto \gamma(q))1\otimes q$; if $q\notin F$, both map to zero in $k$, and if $q \in F$ then both map to one since $\gamma = 1$ on $F$. It follows that $\pi_1(P')$ acts on $Y$ and the map $Y\to X$ is $\pi_1(P')$-invariant.

The action of $\pi_1(P')$ on $Y$ allows us to define the desired functor
\begin{equation} \label{eqn:strlocpt-fibre}
    F = \Hom_X(Y, -) \colon \FEt_X \la \Gsets{\pi_1(P')}.
\end{equation}
If $P$ is fs and $X$ is Noetherian, then $F$ agrees with the fibre functor constructed in \cite[\S 3]{Stix2002:Thesis}. In particular, $F$ is an equivalence under these assumptions. We shall use approximation to reduce to this case.

To this end, let us write $(P\to A) = \varinjlim \,(P_i\to A_i)$ where $P_i$ are fs and where $A_i$ are strict henselisations of finite type $\bZ$-algebras, and such that $A_i\to A$ and $A_i\to A_j$ are local. For each index $i$, we follow the above recipe for constructing $P'_i$, $P_{i\infty}$, $X_{i\infty}$ etc. The maps $P_i\to P_j\to P$ ($j\geq i$) induce maps $P_{i\infty}\to P_{j\infty}\to P_\infty$ and $X_\infty\to X_{j\infty}\to X_{i\infty}$. We let $Y_i$ be the connected component of $X_{i\infty}$ to which $Y$ maps. We obtain the induced maps $Y\to Y_j\to Y_i$. Define the functors
\[
    F_i = \Hom_{X_i}(Y_i, -)\colon \FEt_{X_i}\la \Gsets{\pi_1(P'_i)}
\]
in the same way as $F$.

We claim that we have a commutative square of categories and functors
\[ 
    \begin{tikzcd}
        \FEt_X \ar[r,"F"] & \Gsets{\pi_1(P')} \\
        \varinjlim \, \FEt_{X_i} \ar[u] \ar[r,"(F_i)",swap] & \varinjlim \Gsets{\pi_1(P'_i)}. \ar[u]
    \end{tikzcd}
\]
Note that the left arrow is an equivalence by \cref{prop:fKet-approx}. Moreover, since $P' = \varinjlim P'_i$, we have $\pi_1(P') \simeq \varprojlim \pi_1(P'_i)$ and hence the right arrow is an equivalence. Finally, we have already established that the functors $F_i$ are equivalences. Therefore the existence of the above diagram implies the assertion. 

Consider the transition maps $X_j\to X_i$ ($j\geq i$). Since $Y_j$ maps to $Y_i$, we have a natural map $Y_j\to (Y_i\times^\sat_{X_i} X_j)$. Therefore, if $\{Z_i\}$ is an object of the category $\varinjlim \, \FEt_{X_i}$, then a morphism $Y_i\to Z_i$ over $X_i$ induces a morphism 
\[
    Y_j\la Y_i\times^\sat_{X_i} X_j\la Z_i\times^\sat_{X_i}X_j= Z_j.
\]
Moreover, $Y_j\to Y_i$ is equivariant with respect to the map $\pi_1(P'_j)\to \pi_1(P'_i)$. This shows that the functors $F_i$ are compatible with base change along $X_j\to X_i$ in the obvious way, and hence together induce a functor $(F_i)$ in the diagram. An analogous argument for the maps $X\to X_i$ shows that the square of categories naturally commutes.
\end{proof}

By a {\bf geometric log point} (or log geometric point) we shall mean a log scheme of the form $\overline{x}=\Spec(P\to k)$ where $k$ is a separably closed field and where $P^\gp$ is $n$-divisible for every $n$ invertible in $k$.  By \cref{prop:local log pt}, the category $\FEt_{\overline{x}}$ is canonically equivalent to the category of finite sets. Similarly, the category $\cat{Sh}(\overline{x}_\ket)$ of Kummer \'etale sheaves is equivalent to the category of sets. A geometric log point of a log scheme $X$ is a map $ \overline{x}\to X$ from a geometric log point. Pull-back along such a map produces a point of the Kummer \'etale topos of $X$. It is easy to see that every non-empty saturated log scheme admits a geometric log point, and in fact the Kummer \'etale topos has enough points. Restricting to the category of finite Kummer \'etale covers we obtain a fibre functor $F_{\overline{x}}\colon \FEt_X\to \sets$.

The main result of this section is the following.

\begin{thm} \label{thm:FEt-Galois}
    Let $X$ be a connected saturated log scheme. Then $\FEt_X$ is a Galois category, and for every geometric log point $\overline{x}\to X$, the functor $F_{\ov{x}}$ is a fibre functor. 
\end{thm}

In the fs noetherian case see \cite[p.\ 285]{IllusieFKN}. 

\begin{defi}
    Let $X$ be a saturated connected log scheme and let $\overline{x}\to X$ be a geometric log point. We denote by $\pi_1(X, \overline{x})$ the fundamental group of $\FEt_X$ with the  fibre functor given by $\overline{x}$ as base point.
\end{defi}

The proof of \cref{thm:FEt-Galois} will occupy the rest of this subsection. The strategy is to prove that finite Kummer \'etale covers correspond to locally constant sheaves of finite sets in the Kummer \'etale topology, extending the result \cite[Proposition~3.13]{IllusieFKN} from the fs case, and then to apply the following general result.

\begin{prop}[{\cite[Proposition 8.42 + Definition 8.43 + Theorem 8.47]{Johnstone1977:ToposTheory}}] \label{prop:lcc-is-Galois} 
    Let $\cX$ be a connected Grothendieck topos with a point $x$. Then, the category $\cat{lcc}(\cX)$ of locally constant sheaves of finite sets is a Galois category, with fibre functor induced by $x^*$.
\end{prop}

\begin{cor} \label{cor:ket-lcc-is-Galois}
    Let $X$ be a non-empty connected saturated log scheme. Then the category $\cat{lcc}(\KEtsite{X})$ is a Galois category. Every geometric log point of $X$ induces a fibre functor on this category.
\end{cor}

\begin{proof}
The topos $\cX = \cat{Sh}(\KEtsite{X})$ is connected since 
\[ 
    H^0(\cX, \bZ) = H^0(X, \bZ) = \bZ.
\]
Moreover, it is a Grothendieck topos, being the topos of sheaves on $\KEtsite{X}$. Here, strictly speaking, we should have first chosen a universe $U$ and consider the $U$-Kummer \'etale site, so that $\KEtsite{X}$ is small. \Cref{prop:lcc-is-Galois} implies the assertion.
\end{proof}

For the result below, recall from \cref{prop:Ket-subcanonical} that the Kummer \'etale site is subcanonical, so that $\FEt_X$ can be treated as a full subcategory of $\cat{Sh}(\KEtsite{X})$. 
 
\begin{prop} \label{prop:FEt-equals-lcc}
    The Kummer \'etale sheaf associated to a finite Kummer \'etale morphism $Y\to X$ is a locally constant sheaf of sets on $\KEtsite{X}$. This induces an equivalence 
    \[ 
        \FEt_X \simeq \cat{lcc}(\KEtsite{X}).
    \]    
\end{prop}

\begin{proof}
We first prove that for every finite Kummer \'etale $Y\to X$ the associated representable sheaf is locally constant, with values in finite sets. This assertion is local, so by \cref{lem:fKet-local-structure} we may assume that $X=\Spec(P\to A)$ is charted affine and that $Y$ is standard Kummer \'etale, i.e.\ the disjoint union of $Y_i = \Spec(Q_i\to A\otimes_{\bZ[P]}\bZ[Q_i])$ for a finite collection of Kummer \'etale homomorphisms $P\to Q_i$. Then \cref{lem:Vidal lemma on fs diagonal} implies that $Y_i\times_X^\sat Y_i\to Y_i$ is isomorphic to the disjoint union of finitely many copies of $Y_i$. This shows that the Kummer \'etale sheaf represented by $Y$ is locally constant, trivialized on $\prod^\sat(Y_i\to X)$.

For the other direction, we must show that every locally constant sheaf of finite sets on $\KEtsite{X}$ is represented by a finite Kummer \'etale cover. This again is strict \'etale local and we may assume $X = \varprojlim X_i$ as above, and the statement is true for $X_i$ again by \cite[Proposition~3.13]{IllusieFKN}. \Cref{prop:limit-Ket-topos} implies then that 
\[ 
    \cat{Sh}(\KEtsite{X}) \simeq \varprojlim \cat{Sh}(X_{i,\ket}). 
\]
We claim that this implies that 
\[ 
    \cat{lcc}(\KEtsite{X}) \simeq \varinjlim \cat{lcc}(X_{i,\ket})
\]
(filtered colimit under pullback maps). Indeed, the category of coherent objects of $\cat{Sh}(X)$ is the filtered colimit of the categories of coherent objects of $\cat{Sh}(X_i)$ \cite[Exp.\ VI, 8.3.13]{SGA4_2}. It is easy to see that every object of $\cat{lcc}(\KEtsite{X})$ or $\cat{lcc}(X_{i,\ket})$ is coherent. It remains to show that if $\cF_0$ is a Kummer \'etale sheaf on $X_0$ which is coherent as an object of $\cat{Sh}(X_0)$ and whose pullback $\cF$ to $X$ belongs to $\cat{lcc}(\KEtsite{X})$, then the pullback $\cF_i$ of $\cF_0$ is an object of $\cat{lcc}(X_{i,\ket})$ for $i\gg 0$. Being locally constant means that there exist a finite covering family $\{Y_\alpha\to X\}$, finite sets $S_\alpha$, and isomorphisms $\cF\times Y_\alpha \simeq S_\alpha\times Y_\alpha$ over $Y_\alpha$ (here $\times$ denotes product in the topos $\cat{Sh}(\KEtsite{X})$). Since this is a finite diagram of morphisms and isomorphisms between coherent objects, again it can be descended to some $X_i$. 
\end{proof}

Since $U\mapsto \cat{lcc}(\KEtsite{U})$ is trivially a stack in the Kummer \'etale topology, we immediately deduce the following strengthening of \crefpart{prop:fKet-properties}{propitem:fKet-properties-stack}.

\begin{cor} \label{cor:fKet-Ket-stack}
    The association $U\mapsto \FEt_U$ is a stack for the Kummer \'etale topology on $X$.
\end{cor}

We are now ready to finish the construction of the Kummer \'etale fundamental group.

\begin{proof}[Proof of \cref{thm:FEt-Galois}]
By \cref{prop:FEt-equals-lcc}, the category $\FEt_X$ is equivalent to $\cat{lcc}(\KEtsite{X})$. The latter is a Galois category by \cref{cor:ket-lcc-is-Galois}. This identification is compatible with pullback, and a log geometric point induces a point of $\KEtsite{X}$, and it follows that $F_{\ov x}$ is a fibre functor. 
\end{proof}

\subsection{Kummer \'etale coverings and valuation theory}
\label{ss:kummer-val-rings}

In this subsection, we show that finite Kummer \'etale coverings of the spectrum of a valuation ring correspond to tame extensions of the fraction field. 

We first discuss the natural log structure on the spectrum of a~valuation ring. Let $K^+$ be a~valuation ring with fraction field $K$. Endow $S=\Spec(K^+)$ with the log structure 
\[
    \cM_S = \cO_S \cap j_*\cO_{\Spec(K)}^\times
\]
where $j\colon \Spec(K)\to S$ is the inclusion. We call $\cM_S$ the {\bf standard log structure}. Recall that $K^+$ is {\bf microbial} if it admits an element $\pi\in K^+$ (called a {\bf pseudouniformizer}) such that $K = K^+[1/\pi]$. In this case, the map $j$ is an open immersion and $\cM_S$ is the compactifying log structure induced by the open subscheme $\Spec(K) = D(\pi)$. 

By the following \cref{lem:chart-on-val-ring}, the inclusion $M = K^+\cap K^\times \to K^+$ provides a chart for this log structure. Note that $M$ is a valuative monoid (with $M/M^\times \simeq \Gamma_K^+$, the non-negative part of the value group of $K^+$).

\begin{lem} \label{lem:chart-on-val-ring}
    Let $K^+$ be a valuation ring with fraction field $K$. Then, the map $K^+\cap K^\times\to K^+$ is a chart for the standard log structure on $\Spec(K^+)$.
\end{lem}

\begin{proof}
Let $S=\Spec(K^+)$, let $\cM^{\rm std}_S$ be the standard log structure, and let $\cM$ be the log structure induced by the chart $M = K^+\cap K^\times \to K^+$. We want to show that the natural map $\cM\to\cM^{\rm std}_S$ is an isomorphism. To this end, it is enough to show that $\overline{\cM}\to\overline{\cM}{}^{\rm std}_S$ is an isomorphism. Let $\fp\subset K^+$ be a prime ideal, let $(K^+_\fp)^\sh$ be a strict henselisation of $K^+$ at $\fp$, and let $\ov{x}\to S$ be the corresponding geometric point. Then $(K^+_\fp)^\sh$ is a valuation ring with fraction field $K^\sh_\fp$ and value group 
\[
    \Gamma_{(K_\fp)^\sh} = \Gamma_{K_\fp} = \Gamma_K/C_\fp 
\]
where $C_\fp = (K^+_\fp)^\times/(K^+)^\times \subseteq \Gamma_K$ is the convex subgroup corresponding to $\fp$. For the first equality above, see \stacks[Lemma]{0ASK}. 

Then $\overline{\cM}_{\ov{x}} = M/F_{\ov{x}}$ where $F_{\ov{x}}$ is the preimage of $\cO_{S, \ov{x}}^\times$ under $M\to K^+\to (K^+_\fp)^\sh$. But $M/F_{\ov{x}}$ is just $(\Gamma_K/C_\fp)^+$. Moreover,
\[ 
    \cM^{\rm std}_{S, \ov{x}} = (K_\fp^+)^\sh \setminus \{0\},
\]
and hence
\[
    \overline{\cM}^{\rm std}_{S, \ov{x}} = \Gamma_{(K_\fp^+)^\sh}^+.
\]
Thus  $\overline{\cM}_{\ov{x}}\to \overline{\cM}^{\rm std}_{S, \ov{x}}$ is an isomorphism.
\end{proof}

\begin{rmk}
In practice, it is often possible and useful to choose a splitting of the short exact sequence of monoids (or equivalently the short exact sequence of associated groups, see \cref{ss:monoid-prelims})
\begin{equation}  \label{eqn:Gammaplus-seq}
    \begin{tikzcd} 
        1\ar[r] & (K^+)^\times \ar[r] & M\ar[r] & \Gamma_K^+ \ar[r] & 1.
    \end{tikzcd}
\end{equation}
This can be done for example if $\Gamma_K$ is a free abelian group (notably, in the discretely valued case), or if $K$ is perfect and strictly henselian (in which case $(K^+)^\times$ is divisible). The resulting map $\Gamma_K^+\to K^+$ is then a chart for the log structure on $\Spec(K^+)$.    
\end{rmk}

Recall that a finite separable extension of valued fields $(L, L^+)/(K, K^+)$
is {\bf tamely ramified} if the extension of strict henselisations $L^\sh/K^\sh$ is of degree prime to the residue characteristic exponent of $K^+$.

\begin{prop}
\label{prop:ket extension of valuation rings}
    Let $(K, K^+)$ be a henselian valued field and let $(L, L^+)$ be a finite tamely ramified extension of $K$. Endow $S=\Spec(K^+)$ and $T=\Spec(L^+)$ with the standard log structures.
    Then, the map $T\to S$ is Kummer \'etale.
\end{prop}

\begin{proof}
We use \cref{cor:fKet-Ket-stack} and a limit argument to reduce to the case where $(K,K^+)$ is strictly henselian. 
By \cite[Corollary 6.2.14]{GabberRamero}, the extension $L/K$ is Galois with abelian Galois group of order invertible in $K^+$.  
We may decompose $L/K$ into a chain of cyclic extensions. When $L/K$ is cyclic of degree $n$, note that~$K^+$ being strictly henselian, all $n$-th roots of unity are already contained in~$K$, and so by Kummer theory we can then choose $a \in K$ such that 
\[
 L = K(a^{1/n}).
\]
Note that~$K^+$ being strictly henselian, all $n$-th roots of units of $K^+$ ($n$ being coprime to the residue characteristic of~$K^+$) are already contained in~$K^+$.
Therefore, $a$ has to satisfy $\lvert a \rvert \ne 1$.
Possibly replacing~$a$ with $1/a$ we may assume that $a \in K^+ $ with $\lvert a \rvert < 1$. 

Before we can complete the proof, we need the following lemma.

\begin{lem} \label{lem:Lplusequals}
In the above situation, we have
\[
 L^+ = 
 \left\{ y = \sum_{k=0}^{n-1} x_k a^{-k/n} \,:\, x_k\in K, \, \lvert x_k\rvert \leq \lvert a^{k/n}\rvert \right\}.
\]
\end{lem}

\begin{proof}
Since $L = K(a^{-1/n})$, we can write any element~$y$ of~$L^+$ in the form
\[
 y = \sum_{k=0}^{n-1} x_k a^{-k/n}
\]
for elements $x_k \in K$.
The values $\lvert a^{-k/n} \rvert$ for $k = 0,\ldots,n-1$ represent different classes in $\Gamma_L/\Gamma_K$.
Therefore, the values $\lvert x_k a^{-k/n} \rvert$ are pairwise different and we get
\[
 \lvert y \rvert = \max_k \ \lvert x_k a^{-k/n} \rvert.
\]
So $y \in L^+$ if and only if
\[
 \lvert x_k a^{-k/n} \rvert \le 1 \quad \text{for all } k=0,\ldots,n-1.
\]
This proves the claim. 
\end{proof}

Let~$Q$ be the saturation of $P = K^+ \cap K^\times$ in  $K^\times\cdot (a^{1/n})^\bZ$. Thus $Q^\gp/P^\gp \simeq \bZ/n\bZ$, and by \cref{lem:ket-monoid-conditions}, the map $P\to Q$ is Kummer \'etale of index prime to the residue characteristic exponent $p$ of~$K^+$. Therefore the induced map of log schemes $\bA_Q\to \bA_P$ is Kummer \'etale over $\bZ_{(p)}$.
We want to show that the square
\[
 \begin{tikzcd}
     T=\Spec((L^+\cap L^\times)\to L^+)   \ar[d]  \ar[r]  & \bA_Q   \ar[d]  \\
     S=\Spec(P\to K^+)           \ar[r]  & \bA_P
 \end{tikzcd}
\]
is a pull-back diagram in the category of saturated log schemes, which will imply that the map $T\to S$ is Kummer \'etale. That is, we need to verify that 
\[ 
    T\isomto \Spec(Q\to K^+\otimes_{\bZ[P]}\bZ[Q]).
\]

First, we check that $K^+\otimes_{\bZ[P]}\bZ[Q]\to L^+$ is an isomorphism. This map is surjective by \cref{lem:Lplusequals}. To show injectivity, we observe first that the map becomes an isomorphism after tensoring with $K$ over $K^+$. It follows that it is enough to show that its source $K^+\otimes_{\bZ[P]}\bZ[Q]$ is torsion-free as a $K^+$-module. This follows from the fact that, since $P$ is valuative, the map $\bZ[P]\to\bZ[Q]$ is flat by \cref{lem:kummer-of-valuative}. 

Finally, to compare the log structures, we need to write every element $y \in L^+ \setminus \{0\}$ as a~product of an element of~$Q$ and a unit of $L^+$.
Above we have represented~$y$ in the form
\[
 y = \sum_{k=0}^{n-1} x_k a^{-k/n}
\]
such that $\lvert x_k \rvert \le \lvert a^{k/n} \rvert$ for $k=0,\ldots,n-1$.
Let $k_0$ be the unique index where $\lvert x_k a^{-k/n} \rvert$ is maximal.
For $u = y \cdot \big((x_{k_0} a^{-k_0/n})\big)^{-1}$, the proof of \cref{lem:Lplusequals} established 
\[
    \lvert u \rvert = \lvert y \rvert \cdot \lvert x_{k_0} a^{-k_0/n} \rvert^{-1} = 1. 
\]
So $u$ is a unit in $L^+$ and $y = u \cdot x_{k_0} a^{-k_0/n}$ is of the desired form.
\end{proof}

Once we leave the realm of discrete valuation rings, it is very rare that $L^+$ is finite over $K^+$ for a tame extension of henselian valued fields $L/K$. In fact, we can characterize precisely when that happens. See \cref{lem:ext-of-val-mon-fg} for an analogous result for extensions of valuative monoids, which we use in the proof below.

\begin{lem} \label{lem:tame-ext-finiteness} 
    Let $(K, K^+)$ be a henselian valued field and let $(L, L^+)$ be a finite tamely ramified extension of $K$. The following are equivalent:
    \begin{enumerate}[(a)]
        \item \label{lemitem:tame-ext-finiteness-finite} $L^+$ is finite over $K^+$,
        \item \label{lemitem:tame-ext-finiteness-ft} $L^+$ is of finite type over $K^+$,
        \item \label{lemitem:tame-ext-finiteness-rare} one of the following conditions holds.
        \begin{enumerate}[(i)]
            \item The extension $L/K$ is unramified.
            \item \label{lemitem:root-of-gen} The maximal ideal of $K^+$ is principal, and there exists an integer $n\geq 1$ such that for every generator $a\in K^+$ of the maximal ideal there exists a finite unramified extension $K'/K$ such that $LK' = K'(a^{1/n})$.
        \end{enumerate}
    \end{enumerate}
    \begin{enumerate}[label=(\alph*{${}'$})]
        \item \label{lemitem:tame-ext-finiteness-monoid-finite} the monoid $\Gamma^+_L$ is finite as a $\Gamma_K^+$-set,
        \item \label{lemitem:tame-ext-finiteness-monoid-ft} the monoid $\Gamma^+_L$ is finitely generated over $\Gamma^+_K$,
        \item \label{lemitem:tame-ext-finiteness-monoid-rare}
        one of the following conditions holds.
        \begin{enumerate}[(i)]
            \item $\Gamma_K = \Gamma_L$, 
            \item $\Gamma_K$ is discrete (has a smallest positive element) and there exists an integer $n\geq 1$ such that $\Gamma_L = \Gamma_K + \frac{1}{n}\bZ \gamma$ where $\gamma$ is the smallest positive element of $\Gamma_K$.
        \end{enumerate}
    \end{enumerate}
\end{lem}

Note that, in particular, a composition of extensions as in \labelcref{lemitem:tame-ext-finiteness-rare}\labelcref{lemitem:root-of-gen} is again of this form.

\begin{proof}
The equivalence of conditions \labelcref{lemitem:tame-ext-finiteness-monoid-finite}, \labelcref{lemitem:tame-ext-finiteness-monoid-ft}, and \labelcref{lemitem:tame-ext-finiteness-monoid-rare} is the content of \cref{lem:ext-of-val-mon-fg}, and the equivalence \labelcref{lemitem:tame-ext-finiteness-finite}$\Leftrightarrow$\labelcref{lemitem:tame-ext-finiteness-ft} follows from the fact that $K^+\to L^+$ is integral. 

For \labelcref{lemitem:tame-ext-finiteness-rare}$\Rightarrow$\labelcref{lemitem:tame-ext-finiteness-finite} we note that $L^+$ is finite over $K^+$ if $L/K$ is unramified. In the other case, we remark that we may replace $K^+$ by a finite \'etale cover and thus assume that $L = K(a^{1/n})$ for a~generator $a$ of the maximal ideal of $K^+$. Then $L^+ = K^+[a^{1/n}]$ is generated as a $K^+$-module by the elements $a^{i/n}$ for $0\leq i <n$ by \cref{lem:Lplusequals}.

It remains to show \labelcref{lemitem:tame-ext-finiteness-finite}$\Rightarrow$\labelcref{lemitem:tame-ext-finiteness-monoid-finite} and \labelcref{lemitem:tame-ext-finiteness-monoid-rare}$\Rightarrow$\labelcref{lemitem:tame-ext-finiteness-rare}.

\medskip

\labelcref{lemitem:tame-ext-finiteness-finite}$\Rightarrow$\labelcref{lemitem:tame-ext-finiteness-monoid-finite}:
We first reduce to the case of a cyclic totally ramified extension. To this end, we claim that
\begin{enumerate}[(1)]
    \item \label{item:claim-tame-1} The assertion holds if $L/K$ is unramified.
    \item \label{item:claim-tame-2} If $K\subseteq E\subseteq L$ is an intermediate extension, and the assertion holds for $E/K$ and $L/E$, then it holds for $K/L$.
\end{enumerate}
Claim \labelcref{item:claim-tame-1} follows from the fact that if $L/K$ is unramified then $\Gamma^+_L=\Gamma^+_K$.  
For \labelcref{item:claim-tame-2}, suppose that $L^+$ is finite over $K^+$. In particular it is finite over $E^+$. We check that also $E^+$ is finite over $K^+$. Consider the quotient $L^+/E^+$, which is a finitely generated $K^+$-module. Moreover, it is torsion-free since $E^+ = E\cap L^+$. It is therefore free (combine \stacks[Lemma]{0539}, \stacks[Lemma]{0GSE}, and \stacks[Lemma]{00NX}), thus $E^+$ is a direct summand of $L^+$ and hence is finite over $K^+$. By the assertion \labelcref{lemitem:tame-ext-finiteness-finite}$\Rightarrow$\labelcref{lemitem:tame-ext-finiteness-monoid-finite} for $L/E$ and $E/K$ we get that $\Gamma^+_L$ is finite as a $\Gamma^+_E$-set which in turn is finite as a $\Gamma^+_K$-set. Thus $\Gamma^+_L$ is finite as a $\Gamma^+_K$-set. This finishes the proof of \labelcref{item:claim-tame-2}.

With these reductions, we may assume that $K$ is strictly henselian and $L = K(a^{1/n})$ for $a\in K$ with $|a|<1$ and $n=[L:K]$ prime to the residue characteristic of $K^+$ (cf.\ the proof of \cref{prop:ket extension of valuation rings}). 
\Cref{lem:Lplusequals} implies that $L^+$, as a $K^+$-module, is the following sum of fractional ideals
\[ 
    L^+  \simeq \bigoplus_{k=0}^{n-1} I_k,  
    \qquad 
    I_k = \{ x_k \in K \ : \   |x_k| \leq |a|^{k/n} \}.
\]
Thus, if $L^+$ is finite over $K^+$, then each $I_k$ is finitely generated and hence principal, generated by an element $y_k \in K$. 
This means that $L^+$ is a free $K^+$-module with basis $y_k a^{-k/n}$, for $0 \leq k \leq n-1$. Since we assumed $|a| < 1$, we have $y_k\in K^+$. We check that the images $\gamma_0, \ldots, \gamma_{n-1}$ in $\Gamma_L^+$ of these basis elements 
$y_k a^{-k/n}$ generate $\Gamma_L^+$ as a $\Gamma_K^+$-set: let $\gamma \in \Gamma_L^+$ be the image of $y\in L^+$. By the above discussion (see the proof of \cref{lem:Lplusequals}) we may take $y$ of the form $x a^{-k/n}$ where $0\leq k<n$ and $x\in K^+$, i.e.\ $y\in I_k$. Thus $y = z y_k$, and  $\gamma = \gamma_k + \zeta$ where $\zeta\in\Gamma_K^+$ is the image of $z$. 

\medskip

\labelcref{lemitem:tame-ext-finiteness-monoid-rare}$\Rightarrow$\labelcref{lemitem:tame-ext-finiteness-rare}: 
We may pass to strict henselisations as in the proof of \cref{prop:ket extension of valuation rings}. 
If $\Gamma_K = \Gamma_L$, then $L=K$ because by \cite[Corollary 6.2.14]{GabberRamero} the extension $L/K$ is Galois with trivial Galois group. 

Suppose now $\Gamma_K\neq\Gamma_L$ and let $\gamma\in \Gamma_K^+$ be the smallest positive element. Then $\frac{1}{n}\gamma$ is the smallest positive element of $\Gamma_L^+$. Let $b\in L^+$ be an element with image $\frac{1}{n}\gamma$ in $\Gamma_L^+$. It generates the maximal ideal of $L^+$, and $a = b^n$ generates the maximal ideal of $K^+$. We have $L = K(a^{1/n})$, showing \labelcref{lemitem:tame-ext-finiteness-rare}.
\end{proof}

The examples below show that finite Kummer \'etale maps are not finite in general as maps of schemes. 

\begin{ex}[Type 5 point] \label{ex:type5}
    Let $K$ be a field endowed with a henselian discrete valuation $\nu\colon K\to \bZ\cup\{\infty\}$ and residue field $F$, and let $\mu\colon F\to \bZ\cup\{\infty\}$ be another discrete valuation. Let $K^+ \subseteq K$ be their composition, i.e.\ the subring consisting of elements $x\in K$ such that $\nu(x)\geq 0$ and $\mu(\overline{x})\geq 0$ where $\overline{x}$ is the image of $x$ in $F$. Then $K^+$ is a valuation subring of $K$ corresponding to the valuation of rank two $\sigma\colon K\to (\bZ\times\bZ)_{\rm lex}\cup\{\infty\}$. Such rings arise in practice in the geometry of adic spaces; for example, the residue field of a ``type 5'' point of the adic unit disc over a discretely valued non-archimedean field is of this type. 

    Let $x\in K$ be an element with $\sigma(x) = (1,0)$, let $m>1$ be invertible in $K^+$, and let $L = K(x^{1/m})$. Then $L$ is a tame extension of $K$, and hence the morphism $\Spec(L^+)\to \Spec(K^+)$ is finite Kummer \'etale. The homomorphism $\Gamma_K^+\to \Gamma_L^+$ is Kummer \'etale but not of finite type (see the related \crefpart{ex:typeV-not-fg}{exitem:typeV-not-fg2}), and $\Spec(L^+)\to \Spec(K^+)$ is not of finite type.
\end{ex}

\begin{ex}[Type 3 point] \label{ex:type3}
    Let $K$ be a field with a real valuation $\nu\colon K\to \bR\cup\{\infty\}$ whose image is isomorphic to $\bZ^2$ (and hence dense). Such fields arise as the residue fields of ``type 3'' points on the rigid-analytic unit disc. Picking a basis $1, \lambda \bR$ of the value group $\Gamma_K = \nu(K^\times)$, the monoid $\Gamma_+$ is identified with $\{(a,b)\in \bZ\,:\,a+b\lambda\geq 0\}$. Let $x,y\in K$ be such that $\nu(x)=1$ and $\nu(y)=\lambda$ , and let $L = K(x^{1/m}, y^{1/m})$ where $m>1$ is prime to the residue characteristic. Again, $L$ is a tame extension of $K$, the monoid $\Gamma_L^+$ is not finitely generated over $\Gamma_K^+$ (see  \crefpart{ex:typeV-not-fg}{exitem:typeV-not-fg2b}), and $\Spec(L^+)\to \Spec(K^+)$ is finite Kummer \'etale but not of finite type. 
\end{ex}

\begin{ex}[{See \cite[\S 6.4.1, bottom of p.\ 250]{BGR}}] \label{ex:perfectoid}
    Let $p$ be a prime and let $K$ be the completion of $\bQ_p(p^{1/p^\infty})$ (which is a perfectoid field). Let $L = K(p^{1/m})$ for some $m>1$ prime to $p$. Then $L$ is a tame extension of $K$, but $K^+$ is not finite over $L^+$. The homomorphism of valuative monoids $\Gamma_K^+\to \Gamma_L^+$ is the inclusion $\bZ[1/p]\cap \bR_+ \to \frac{1}{m}\bZ[1/p]\cap \bR^+$, which is not of finite type.  
\end{ex}

The following corollary is a restatement of \cref{prop:ket extension of valuation rings} in terms of fundamental groups.

\begin{cor} \label{cor:pi1-vs-Galois}
    Let $(K,K^+)$ be a henselian valued field, and let $\ov{K}$ be an algebraic closure of $K$. Let $\ast$ denote the geometric point $\Spec \ov{K} \to \Spec K$, and let $K^\sep$ denote the separable closure of $K$ contained in $\ov{K}$. The isomorphism $\pi_1(K,\ast) \simeq \Gal(K^\sep/K)$ induces an isomorphism
    \[
    \pi_1(\Spec(K^+\cap K^\times\to K^+),\ast) \simeq \Gal^{\rt}(K^\sep/K)
    \]
    with the maximal tame quotient $\Gal^{\rt}(K^\sep/K)$ of $\Gal(K^\sep/K)$.
\end{cor}

\section{Semistable reduction over valuation rings}
\label{s:semistable}

In this last, short section, we showcase our theory by elucidating the log structures on semistable schemes over arbitrary valuation rings, especially those with algebraically closed fraction fields. This relies on the results on sfp monoids over valuative monoids in \cref{ss:typeV-typeVd}.  

\subsection{\texorpdfstring{Log schemes of type $\typeV$ and of type $\typeVd$}{Log schemes of type V and of type Vdiv}}

The following is a direct analogue of \cref{def:monoid-typeV} for log schemes. 

\begin{defi} \label{defi:log-sch-type-V-Vd}
    A log scheme $X$ is {\bf of type $\typeV$} (resp.\ {\bf of type $\typeVd$} if locally it admits a~chart $P\to \cM_X(X)$ with $P$ of type $\typeV$ (resp.\ of type $\typeVd$). 
\end{defi}

If $Y\to X$ is a locally sfp morphism and $X$ is of type $\typeV$ or $\typeVd$, then so is $Y$. \Cref{lem:localizeVVdiv} implies that for a log scheme $X$ of type $\typeV$ (resp.\ $\typeVd$), the stalks of the sheaf $\overline{\cM}_X$ are monoids of type $\typeV$ (resp.\ $\typeVd$). Moreover, we have the following property in type $\typeVd$, as a consequence of the finiteness theorem for monoids (\cref{cor:GR-finiteness}).

\begin{prop} \label{prop:log-sch-typeVd}
    Let $Y\to X$ be a locally sfp morphism between log schemes of type $\typeVd$. Then the underlying morphism of schemes $\underline{Y}\to \underline{X}$ is locally finitely presented.\footnote{The map $Y\to X$ is also locally finitely presented as a map of log schemes, i.e.\ it satisfies the same condition as a locally sfp map but with saturated pullback replaced with pullback in log schemes.}  
\end{prop}

\begin{proof}
Working locally we may assume $X$ has a global chart by a monoid $P$ of type $\typeVd$, i.e.\ sfp over a divisible valuative monoid $V$. By \cref{cor:sfp-chart-lifting}, this chart can locally be lifted to a~chart for the map $Y\to X$ given by an sfp morphism of monoids $P\to Q$, with $Y\to X\times_{\bA_P} \bA_Q$ strict and of finite presentation. By \cref{cor:GR-finiteness}, $P$ and $Q$ are both finitely presented over $V$, and hence so is the map $P\to Q$. Then $X\times_{\bA_P} \bA_Q\to X$ is of finite presentation, and hence so is $Y\to X$. \end{proof}

The principal example of a log scheme of type $\typeV$ (resp.\ of type $\typeVd$) is one (locally) admitting an sfp morphism $X\to S$ where $S$ (locally) admits a chart by a valuative monoid (resp.\ a divisible valuative monoid). In practice, $S$ will be either the spectrum of a valuation ring $K^+$, with the standard log structure (see \cref{ss:kummer-val-rings}), or a closed subscheme of such a space, with the induced log structure, or a log point charted by a valuative monoid.

\begin{ex}[Spectrum of a valuation ring] \label{ex:SpecKplus}
The key example is the following. Let $K^+$ be a valuation ring with fraction field $K$. Endow $S = \Spec(K^+)$ with the standard log structure, charted by the valuative monoid $M = K^+\cap K^\times$ (see \cref{ss:kummer-val-rings}). As $M$ is valuative, every log scheme $X$ which is locally sfp over $\Spec(K^+)$ or over a closed subscheme of $\Spec(K^+)$ is of type $\typeV$. 

If the value group $\Gamma_K$ of $K$ is divisible and \labelcref{eqn:Gammaplus-seq} splits (for example, if $K$ is algebraically closed), then $\Spec(K^+)$ is of type $\typeVd$, and so is every log scheme $X$ which is locally sfp over $\Spec(K^+)$ or over a closed subscheme of $\Spec(K^+)$. In this case, the map $X\to \Spec(K^+)$ is finitely presented and saturated. We do not know if these assertions hold if we only assume that $\Gamma$ is divisible.
\end{ex}

\subsection{Log schemes over valuation rings}

Let $K^+$ be a valuation ring with fraction field $K$ and value group $\Gamma$. We assume that $K^+$ is microbial (see \cref{ss:kummer-val-rings}) and fix a pseudouniformizer $\pi$ of $K^+$. We endow $\Spec(K^+)$ with the standard log structure charted by $M\to K^+$ where $M=K^+\cap K^\times$ as in \cref{ex:SpecKplus}, which coincides with the compactifying log structure induced by the open subset $\Spec(K) = D(\pi) \subseteq \Spec(K^+)$. 

The following results follows easily from the Reduced Fibre Theorem for monoids  (\cref{cor:RFT}). In the discretely valued case, it is due to Tsuji \cite{Tsuji}.

\begin{lem}[Log reduced fibre theorem] \label{lem:sat-after-bc}
    Let $Y\to X$ be a morphism between sfp log schemes over $K^+$. Then there exists a finite extension $L$ of $K$ such that the saturated base change $Y_{L^+}\to X_{L^+}$ is saturated and finitely presented.
\end{lem}

We record here the following direct corollary of \cref{prop:log-sch-typeVd}.

\begin{cor}[Log Grauert--Remmert finiteness theorem]
    Suppose that $\Gamma_K$ is divisible and \labelcref{eqn:Gammaplus-seq} splits (for example, $K$ is algebraically closed). Then every morphism between locally sfp log schemes over $K^+$ is saturated and locally finitely presented. In particular, a finite Kummer \'etale morphism between locally sfp log schemes over $K^+$ is finite and finitely presented as a~morphism of schemes.
\end{cor}

\subsection{The standard log structure over a valuation ring} 
\label{ss:std-log-str}

The main result of this section relates the log structure on a smooth and vertical log scheme over $K^+$ to the log structure defined by the generic fiber, called the standard log structure. We preserve the setup and notation of the previous subsection. 

Every scheme $X$ over $K^+$ can be endowed with the {\bf standard log structure} $\cM_X^{\rm std}$ induced by the open subset $X_\eta = D(\pi)\subseteq X$. We have
\[ 
    \cM^{\rm std}_X = \{f\in \cO_X \,:\, \text{locally $fg=\pi^n$ for some $g\in \cO_X$ and $n\geq 0$}\} \subseteq \cO_X.
\]
A log scheme $X$ over $K^+$ is {\bf vertical} if the log structure on $X_\eta \subseteq X$ is trivial. For example, any scheme over $K^+$ endowed with the standard log structure is vertical. This notion is compatible with the notion of a vertical map of monoids introduced in \cref{defi:exactintsatvert monoids} in the following sense: $X$ is vertical if and only if for every geometric point $\ov{x}\to X$, the homomorphism of monoids 
\[ 
    \cM_{\Spec(K^+),f(\ov{x})}\la \cM_{X,\ov{x}}
\]
is vertical. Indeed, since $K = K^+[1/\pi]$, the element $\pi$ generates $K^+\cap K^\times$ as a face. Then both assertions mean that for every local section $f$ of $\cM_X$ there exists locally a section $\sigma$ of $\cM_{\Spec(K^+)}$ (which can be taken to be a power of $\pi$) and a~section $g$ of $\cM_X$ such that $fg = \sigma$.

\begin{prop} \label{prop:vertical-compactifying}
    Let $X\to S$ be a smooth and vertical morphism. Suppose that 
    $K$ is discretely valued, algebraically closed, or of equal characteristic zero. Then $\cM_X\to \cO_X$ is injective and its image is equal to $\cM_X^{\rm std}$. 
\end{prop}

For $X\to S$ semistable and $K$ of rank one, this has been proved rather explicitly in \cite[Corollary~2.3.9]{Temkin2017:AlteredLocalUniformization}. Adapting this sort of method to our setting could be complicated. Instead, we will prove the result using approximation, employing the following variant of a result of Zavyalov \cite[Lemma~A.2]{Zavyalov}. 

\begin{lem}[Zavyalov] \label{lem:Zavyalov}
    Let $K^+$ be a microbial valuation ring of a field $K$ satisfying one of the conditions: 
    \begin{enumerate}[(a)]
        \item \label{lemitem:Zavyalov-dvr} $K^+$ is an excellent (e.g.\ complete) discrete valuation ring,
        \item \label{lemitem:Zavyalov-Kbar} $K$ algebraically closed, or
        \item \label{lemitem:Zavyalov-eqchar0} $K$ is of residue characteristic zero.
    \end{enumerate}
    and let $\pi\in K$ be a pseudouniformizer. Then $K^+$ is isomorphic to the filtered colimit $\varinjlim \Lambda_i$ of subalgebras $\Lambda_i\subseteq K^+$ such that
    \begin{enumerate}
        \item $\pi \in \Lambda_i$
        \item $\Lambda_i$ is a local ring and $\Lambda_i\to K^+$ is a local homomorphism,
        \item $\Lambda_i$ is excellent and regular,
        \item the divisor $V(\pi)_{\rm red}\subseteq \Spec(\Lambda_i)$ has strict normal crossings.
    \end{enumerate}
    Moreover, if $K$ is henselian (resp.\ strictly henselian), then the $\Lambda_i$ may be chosen to be henselian (resp.\ strictly henselian) as well. 
\end{lem}

\begin{proof}
In case \labelcref{lemitem:Zavyalov-dvr} we can take $\Lambda_i = K^+$. Case \labelcref{lemitem:Zavyalov-Kbar} follows as in \cite[Lemma A.2]{Zavyalov} (the only difference is that Zavyalov assumes $K$ to be of rank one, but that does not enter the proof), followed by replacing $\Lambda_i$ produced this way with their respective localizations/henselisations/strict henselisations with respect to the map to the residue field of $k$. For case \labelcref{lemitem:Zavyalov-eqchar0} we follow the same path, replacing alterations with resolution of singularities for schemes of finite type over $\bQ$. 
\end{proof}

In the situation of \cref{lem:Zavyalov}, let $M_i = \Lambda_i\cap \Lambda_i[1/\pi]^\times$. Then it is easy to check that
\[ 
    (M\to K^+) = \varinjlim \, (M_i\to \Lambda_i).
\]
(filtered colimit in the category of saturated prelog rings). Even though the $M_i$ are not fs monoids, the log structures they induce on $\underline{S}_i = \Spec(\Lambda_i)$ are fs and even log regular. Thus by \cite[Theorem~8.2]{KatoToricSingularities}, every smooth log scheme $X_i$ over $S_i$ is log regular. Consequently, the log structure on $X_i$ coincides with the one induced by the largest open subset $X_{i,{\rm triv}}\subseteq X_i$ on which the log structure is trivial. If $X_i\to S_i$ is moreover vertical, this open subset coincides with the subset where $\pi$ is invertible, i.e.\ $X_{i,\eta} = X_i\times_{S_i} S_{i,\eta}$ where $S_{i,\eta} = \Spec(\Lambda_i[1/\pi])$.

Before turning to the proof of \cref{prop:vertical-compactifying}, we record here another consequence of \cref{lem:Zavyalov}.

\begin{cor} \label{cor:presentation-after-finite-extn}
    Let $K^+$ be a microbial valuation ring and let $X\to S=\Spec(K^+)$ be a qcqs smooth and vertical morphism. Then, there exists a finite extension $L$ of $K$, a local regular excellent log scheme $S_0$, and a log smooth and vertical morphism $X_0\to S_0$ fitting inside a~cartesian diagram in saturated log schemes
    \[ 
        \begin{tikzcd}
            X\ar[d] & X' \ar[d] \ar[l]\ar[r] & X_0 \ar[d]  \\
            S        & S'  \ar[l]\ar[r]      & S_0 
        \end{tikzcd}
    \]
    where $S' = \Spec(L^+)$, and such that the log structure on $S_0$ is given by an snc divisor $V(\pi_0)\subseteq S_0$ for an element $\pi_0\in\cO(S_0)$ with image $\pi\in \cO(S')=L^+$. If $K$ is discretely valued or of equal characteristic zero, this holds with $L=K$. Enlarging $L$, we can moreover ensure that $X_0\to S_0$ is saturated.
\end{cor}

\begin{proof}
Suppose first that $K$ is algebraically closed. We have $S = \varprojlim S_i$ as in \cref{lem:Zavyalov} and hence by \cref{cor:SmEtKet-approx} there exists a smooth morphism $X_i\to S_i$ whose saturated base change to $S$ is $X\to S$. Increasing $i$, we may assume $X_i\to S_i$ is  vertical (in the sense that the log structure is trivial on $X_{i,\eta} = X_i\times_{S_i} S_{i,\eta}$). For the latter assertion, note that since $X_\eta\to S_\eta$ is strict and $S_\eta = \varprojlim S_{i,\eta}$, by \cref{cor:descent-of-strictness} the map $X_{i,\eta}\to S_{i,\eta}$ will be strict for $i\gg 0$. But $S_{i,\eta}$ has trivial log structure, and hence so does $X_{i,\eta}$, and $X_i\to S_i$ is then vertical. We set $S'=S$ and $S_0=S_i$.

The same argument works if $K$ is discretely valued or of residue characteristic zero. 

For the general case, we apply the above argument to the algebraic closure $\ov{K}$ (endowed with some extension of the valuation) and note that there exists a finite extension $L$ of $K$ containing $\Lambda_0$, so that $\ov{S} = \Spec(\ov{K}^+)\to S_0$ factors through $S'=\Spec(L^+)$.

By \cref{lem:sat-after-bc}, we can make $X'\to S'$ saturated. Then $X_i\to S_i$ will be saturated for $i\gg 0$. 
\end{proof}

We are now ready to prove \cref{prop:vertical-compactifying}.

\begin{proof}[Proof of \cref{prop:vertical-compactifying}]
Clearly the image of $\cM_X$ in $\cO_X$ is contained in $\cM_X^{\rm std}$. We shall prove the other inclusion. Since the assertion is strict \'etale local we might assume that $X$ is qcqs.

Let $f$ be a local section of $\cM_X^{\rm std}$, i.e.\ a function $f$ such that locally $fg=\pi^n$ for some function $g$ and $n\geq 1$. By \cref{lem:Zavyalov} we find that $X\to S$ is the base change of a log smooth vertical map $X_i\to S_i$ such that there exist functions $f_i$, $g_i$ on $X_i$ with $f_ig_i = \pi^n$. Since $\cM_{X_i} = \cO_{X_i}\cap \cO_{X_i}[1/\pi]^\times$, we have that $f_i$ is a section of $\cM_{X_i}$, and hence $f$ is a section of $\cM_X$. (Note that since the maps $\cM_{X_i}\to \cO_{X_i}$ are injective, so is their colimit $\cM_X\to \cO_X$.) 
\end{proof}

Along similar lines, using the fact that log regular rings are normal \cite[Theorem~4.1]{KatoToricSingularities} and that filtered colimits of normal rings are normal \stacks[Lemma]{037D}, we can prove the following result.

\begin{lem} \label{lem:scheme-normal}
    Let $X\to S = \Spec(K^+)$ be a log smooth morphism. Then the underlying scheme of $X$ is normal and flat over $K^+$.
\end{lem}

\begin{rmk}[Cohen--Macaulayness]
    Kato \cite[Theorem~4.1]{KatoToricSingularities} shows that log regular local rings are not only normal but also Cohen--Macaulay. It is possible to extend the notion of Cohen--Macaulayness to non-noetherian local rings in several non-equivalent ways (see \cite{KimWalker} and references therein as well as the discussion of Cohen--Macaulayness of semigroup rings). We do not know if log smooth schemes over $K^+$ are Cohen--Macaulay in either of these senses. Note that Cohen--Macaulayness of filtered colimits with particular interest in semigroup rings has been studied in the recent paper \cite{AsgharzadehDorrehTousi}. 
\end{rmk}

Finally, we link these results to the more classical notion of semistability used in \cite{Temkin2017:AlteredLocalUniformization,Zavyalov}.

\begin{defi}
    A scheme $X$ over $K^+$ is {\bf semistable} if \'etale locally on $X$ there exists a~pseudouniformizer $\pi$ of $K$, an integer $n\geq 0$, and an \'etale map of schemes over $K^+$
    \[ 
        X\la \Spec K^+[x_1, \ldots, x_n]/(x_1\cdots x_n - \pi)
    \]
    We call $X$ {\bf strictly semistable} if such a structure exists Zariski locally on $X$.
\end{defi} 

Note that the pseudouniformizer $\pi$ exists only locally on $X$, which causes some complications. See the discussion in \cite{Temkin2017:AlteredLocalUniformization}.

Then \cref{prop:vertical-compactifying} implies the following result.

\begin{cor}
    Let $X$ be a semistable scheme over $K^+$. Then $X$ is log smooth, saturated, and vertical over $K^+$ when endowed with the standard log structure $\cM_X^{\rm std}$.  
\end{cor}

\begin{proof}
Indeed, $X\to \Spec(K^+)$ is locally charted by the homomorphism of monoids
\[ 
    M \la P_n(\pi) = M[e_1, \ldots, e_n]/(e_1+\cdots+e_n = \pi)
\]
for some $n\geq 0$ and a pseudouniformizer $\pi$. This homomorphism is log smooth, saturated, and vertical (see \cref{ex:semistable-monoid}). Moreover, the induced map $X\to \Spec(K^+\otimes_{\bZ[M]}\bZ[P_n(\pi)])$ is strict \'etale, and hence the result.
\end{proof}

\bibliographystyle{alphaSGA}
\bibliography{beyondfs}

\end{document}